\documentclass[reqno]{amsart}

\pdfoutput = 1

\usepackage{amsmath,amssymb}
\usepackage{amsfonts}
\usepackage{array,dcolumn}
\usepackage{graphicx}

\usepackage[T1]{fontenc}
\usepackage[latin1]{inputenc}
\usepackage[activate={true,nocompatibility},spacing,kerning]{microtype}
\usepackage[pdftex,
            colorlinks,
            hyperfootnotes=false,
            pdffitwindow=true,
            plainpages=false,
            pdfpagelabels=true,
            pdfpagemode=UseOutlines,
            pdfpagelayout=SinglePage,
            pdftitle={Accuracy and Stability of Computing High-Order Derivatives},
            pdfauthor={Folkmar Bornemann, Technische Universität München},
            pagebackref,
            hyperindex,
            filecolor=purple,
            urlcolor=purple,
            citecolor=blue,
            linkcolor=blue]{hyperref}
\usepackage{myharvard}
\usepackage{bm}
\usepackage[sc,osf]{mathpazo}

\graphicspath{%
 {./Images/}%
}

\newtheorem{theorem}{Theorem}[section]

\newtheorem{cor}[theorem]{Corollary}

\theoremstyle{definition}

\newtheorem{example}[theorem]{Example}

\theoremstyle{remark}
\newtheorem{remark}[theorem]{Remark}

\numberwithin{equation}{section}

\makeatletter
\renewcommand{\leq}{\leqslant}
\renewcommand{\geq}{\geqslant}
\renewcommand{\Re}{{\operator@font Re}}
\renewcommand{\Im}{{\operator@font Im}}
\newcommand{\tr}{{\operator@font tr}}
\newcommand{\range}{{\operator@font ran}}
\newcommand{\spn}{{\operator@font span}}
\newcommand{\sinc}{{\operator@font sinc}}
\newcommand{\erf}{{\operator@font erf}}
\newcommand{\Ai}{{\operator@font Ai}}
\newcommand{\Bi}{{\operator@font Bi}}
\newcommand{\cov}{{\operator@font cov}}
\newcommand{\var}{{\operator@font var}}
\newcommand{\sgn}{{\operator@font sign}}
\newcommand{\const}{{\operator@font const}}
\newcommand{\res}{{\operator@font res}}
\newcommand{\hess}{{\operator@font hess}}
\newcommand{\arccot}{{\operator@font arccot}}
\let \limsup \relax
\DeclareMathOperator*{\limsup}{\smash[b]{\operator@font lim\,sup}}
\DeclareMathOperator*{\argmin}{\operator@font arg\,min}
\DeclareMathOperator*{\esssup}{\operator@font ess\,sup}
\makeatother

\newcommand{\C}{{\mathbb  C}}

\newcommand{\R}{{\mathbb  R}}

\newcommand{\Z}{{\mathbb  Z}}
\newcommand{\N}{{\mathbb  N}}
\newcommand{\prob}{{\mathbb  P}}

\newcommand{\projected}[1]{|_{#1}}

\begin{document}

\title[Accuracy and Stability of Computing High-Order Derivatives]{Accuracy and Stability of Computing High-Order Derivatives of Analytic Functions by Cauchy Integrals}

\author{Folkmar Bornemann}
\address{Zentrum Mathematik -- M3, Technische Universität München,
         80290~München, Germany}
\email{bornemann@ma.tum.de}
\date{\today}

\subjclass[2000]{}

\begin{abstract} High-order derivatives of analytic functions are expressible as Cauchy integrals
over circular contours, which can very effectively be approximated, e.g., by trapezoidal sums. Whereas
analytically each radius $r$ up to the radius of convergence is equal, numerical stability strongly depends on $r$. We give a comprehensive
 study of this effect; in particular we show that there is a unique
radius that minimizes the loss of accuracy caused by round-off errors. For large classes of functions, though not for all,
this radius actually gives about full accuracy; a remarkable fact that we explain by the theory of Hardy spaces, by the Wiman--Valiron and Levin--Pfluger
theory of entire functions, and by the saddle-point method of asymptotic analysis. Many examples and non-trivial applications are discussed in detail.
\end{abstract}

\maketitle

\setcounter{tocdepth}{1}
\vspace*{-0.25cm}
\tableofcontents

\vspace*{-0.25cm}
\section{Introduction}\label{sect:intro}

Real variable formulae for the numerical calculation of high-order derivatives severely suffer  from the ill-conditioning of real differentiation.
Balancing approximation errors with round-off errors yields an inevitable minimum amount of error that blows up as the order of differentiation increases \citeaffixed[Thm.~2]{MR780799}{see, e.g.,}.
It is therefore quite tricky, using these formulae with hardware arithmetic, to obtain any significant digits for derivatives of
orders, say, hundred or higher. For functions which extend analytically to the complex plane, numerical quadrature applied to Cauchy integrals has on various occasions been
suggested as a remedy \citeaffixed[p.~152/187]{MR1454125}{see}. To be specific, let us consider an analytic function $f$ with the Taylor
series\footnote{Without loss of generality, the point of development is $z=0$, which we choose for ease of notation throughout this paper. Though such series are often named after Maclaurin, we keep the name {\em Taylor series}
to stress that we really do not use anything specific to $z=0$.}
\begin{equation}\label{eq:taylor}
f(z) = \sum_{k=0}^\infty a_k z^k\qquad \qquad (|z|<R)
\end{equation}
having radius of convergence  $R>0$ (with $R=\infty$ for entire functions). Cauchy's integral formula applied to circular contours yields ($n=0,1,2,\ldots$, $0 < r < R$)
\begin{align}
a_n &= \frac{f^{(n)}(0)}{n!} \notag\\*[1mm]
&= \frac{1}{2\pi i} \int_{|z|=r} \frac{f(z)}{z^{n+1}}\,dz\notag\\*[1mm]
&= \frac{1}{2\pi r^n}\int_0^{2\pi} e^{-in\theta}f(re^{i\theta})\,d\theta.\label{eq:an}
\end{align}
Since trapezoidal sums\footnote{Recall that, for periodic functions, the trapezoidal sum and the rectangular rule are just the same.} are known to converge geometrically for periodic analytic functions \cite{MR0100354},
the latter integral is amenable to the very simple and yet effective approximation\footnote{For other quadrature rules see the remarks in §\ref{subsect:rules}.}
\begin{equation}\label{eq:anm}
a_n(r,m) = \frac{1}{m r^n} \sum_{j=0}^{m-1} e^{- 2\pi i j n/m}f(re^{2\pi i j/m}).
\end{equation}
This procedure for approximating $a_n$ was suggested by \citeasnoun{805983}. Later, \citeasnoun{362820} observed that the correspondence
\[
\big(r^n a_n(r,m)\big)_{n=0}^{m-1} \quad\leftrightarrow\quad \big(f(r e^{2\pi ij/m})\big)_{j=0}^{m-1}
\]
induced by (\ref{eq:anm}) is, in fact, the discrete Fourier transform; accordingly they published an algorithm for calculating a set of normalized Taylor coefficients $r^n a_n$ based on the FFT.

Whereas all radii $0< r <R$ are, by Cauchy's Theorem, analytically equal, they are not so numerically. On the one hand, the geometric convergence rate of the trapezoidal
sums improves for smaller $r$. On the other hand, for $r\to 0$ there is an increasing amount of cancelation in the Cauchy integral which leads to a blow-up
of relative errors \cite[p.~130]{805983}. Moreover, there is generally also a problem of numerical stability for $r\to R$ (see §\ref{sect:cond} of this paper).
So, once again there arises the question of a proper balance between approximation errors and round-off errors: what choice of $r$ is best and what is the minimum error thus
obtained?

There is not much available about this problem in the literature. \citeasnoun{362820} circumnavigate it altogether
by just considering the \emph{absolute} errors of the normalized Taylor coefficients $r^n a_n$ instead of relative errors,
leaving the choice of $r$ to the user as an application-specific scale factor; on p.~670 they write:

\medskip
\begin{quote}
It is natural to ask why this choice of output [i.e., $r^n a_n$] was made, rather than perhaps a set of Taylor coefficients $a_n$ or a set of derivatives $f^{(n)}(0)$.
The most immediate reason is that the algorithm naturally provides a set of normalized Taylor coefficients to a uniform absolute accuracy. If, for example, one is
interested in a set of derivatives, the specification of the accuracy requirements becomes very much more complicated. However, if one looks ahead to the use
to which the Taylor coefficients are to be put, one finds in many cases that uniform accuracy in normalized Taylor coefficients corresponds to the sort of accuracy
requirement which is most convenient.
\end{quote}
\medskip

\citename{MR654904} \citeyear{MR654904,MR654904b} addresses the choice of a suitable radius $r$ by suggesting a simple search procedure that tries to make $(r^n a_n)_{n=0}^{m-1}$ approximately
proportional to the geometric sequence $0.75^n$. If accomplished, this results, for $m=32$, in a loss of at most about $m |\log_{10}(0.75)| \doteq 4.0$ digits;\footnote{We write ``$\doteq$'' to
indicate that a number has been correctly \emph{rounded} to the digits given, ``$\sim$'' to denote a rigorous asymptotic equality, and ``$\approx$'' to informally assert some approximate agreement.} see §\ref{sect:abs} below. Further, he applies
Richardson extrapolation to the last three radii of the search process to enhance the convergence rate of the trapezoidal sums. However, the success of both devices
is limited to functions whose Taylor coefficients approximately follow a geometric progression. In fact, \citeasnoun[p.~542]{MR654904b} identifies some problems:

\medskip
\begin{quote}\label{qt:fornberg}
Some warning about cases in which full accuracy may not be reached. Such cases are
\smallskip
\begin{enumerate}
\item very low-order polynomials (for example, $f(z)=1+z$);\\*[-3mm]
\item functions whose Taylor coefficients contain very large isolated terms (for example, $f(z)=10^6 + 1/(1-z)$);\\*[-3mm]
\item certain entire functions (for example, $f(z)=e^z$);\\*[-3mm]
\item functions whose radius of convergence is limited by a branch point at which the function remains many times [real] differentiable
(for example, $f(z)=(1+z)^{10} \log(1+z)$ expanded around $z=0$).
\end{enumerate}
\end{quote}

\medskip

\begin{figure}[tbp]
\begin{center}
\begin{minipage}{0.45\textwidth}
\begin{center}
\includegraphics[width=\textwidth]{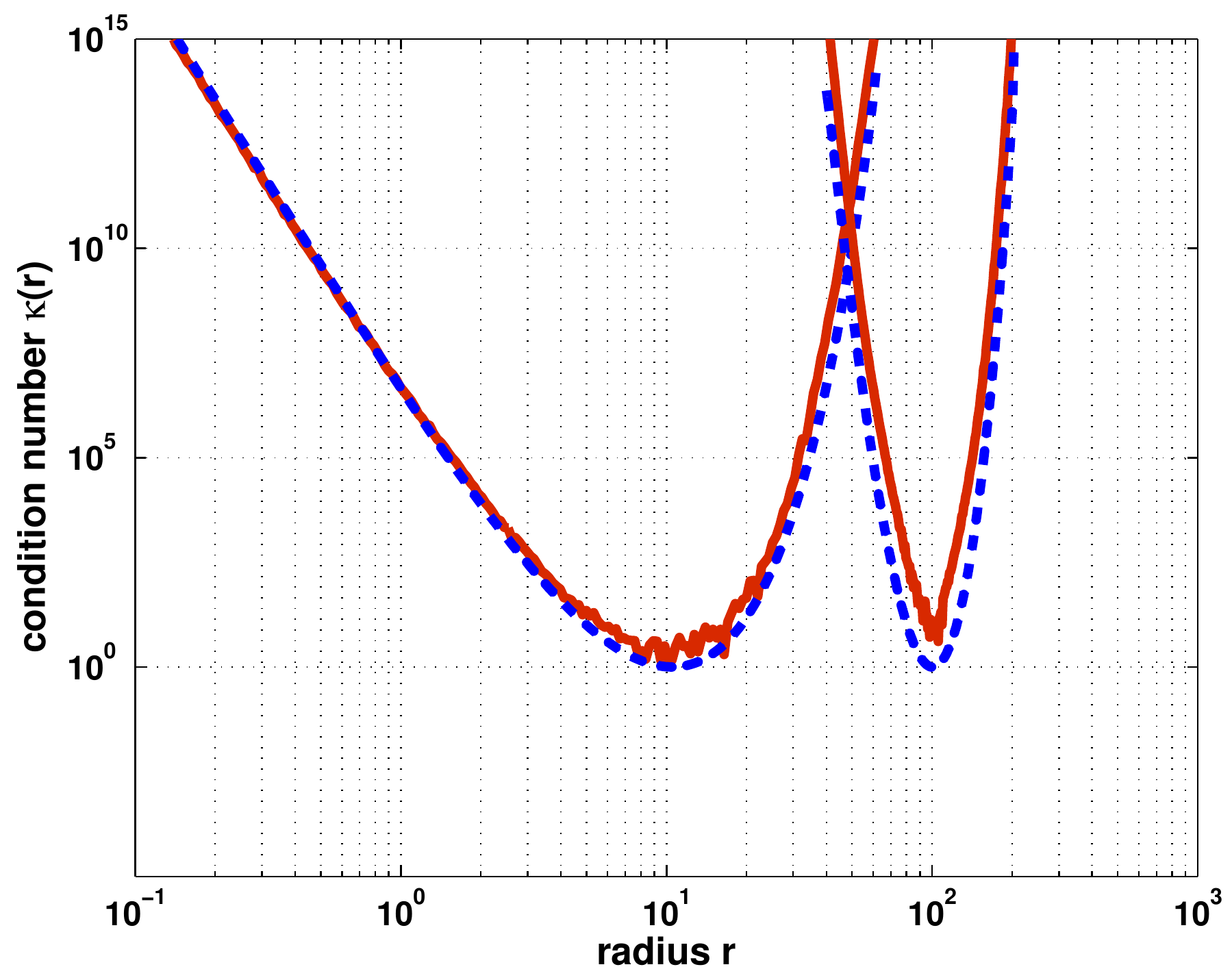}\\*[-1.5mm]
{\tiny a.\;\; $f(z)=\exp(z)$ (Example~\ref{ex:exp})}
\end{center}
\end{minipage}
\hfil
\begin{minipage}{0.45\textwidth}
\begin{center}
\includegraphics[width=\textwidth]{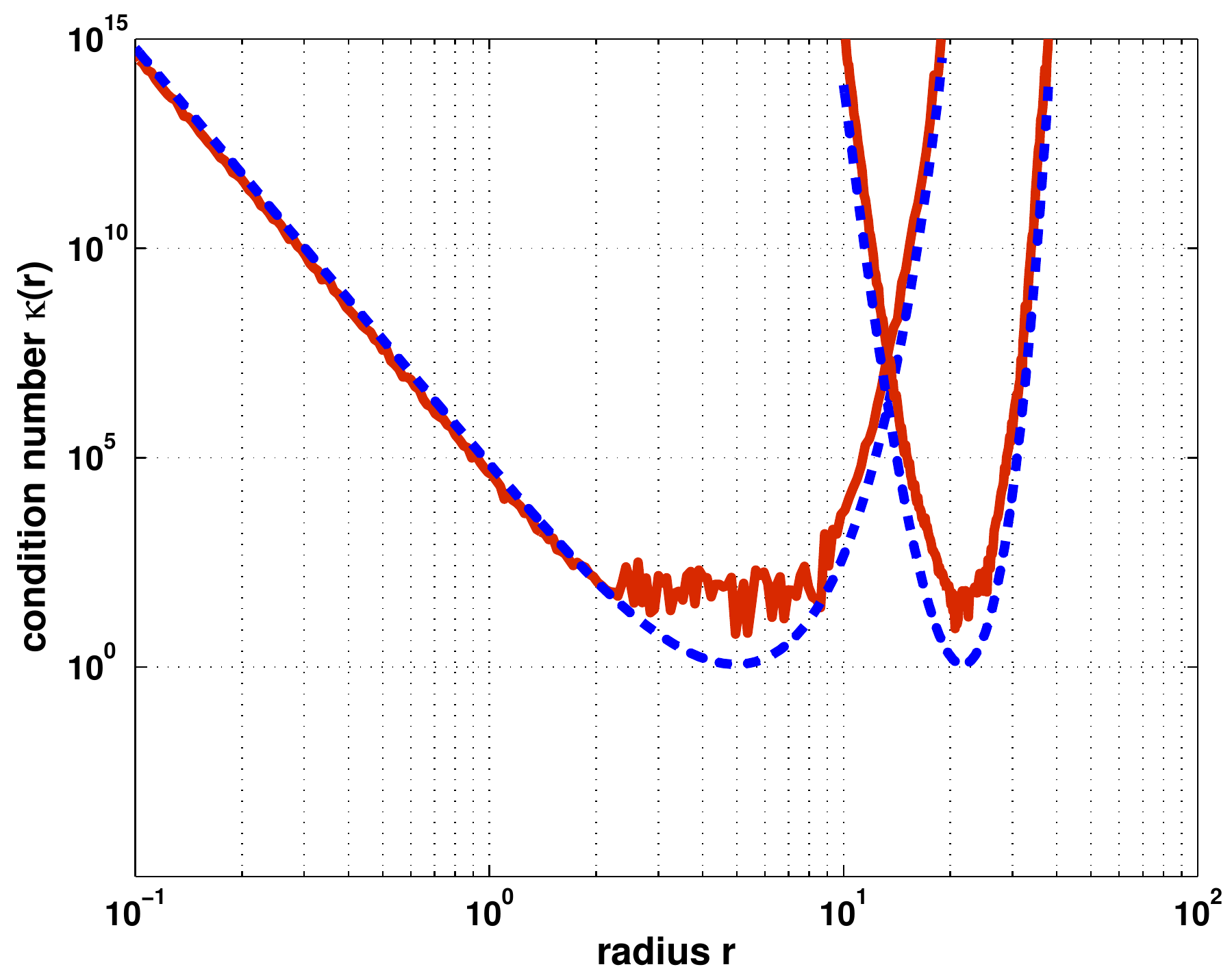}\\*[-1.5mm]
{\tiny b.\;\; $f(z)=\Ai(z)$ (Example~\ref{ex:airy})}
\end{center}
\end{minipage}\\*[2mm]
\begin{minipage}{0.45\textwidth}
\begin{center}
\includegraphics[width=\textwidth]{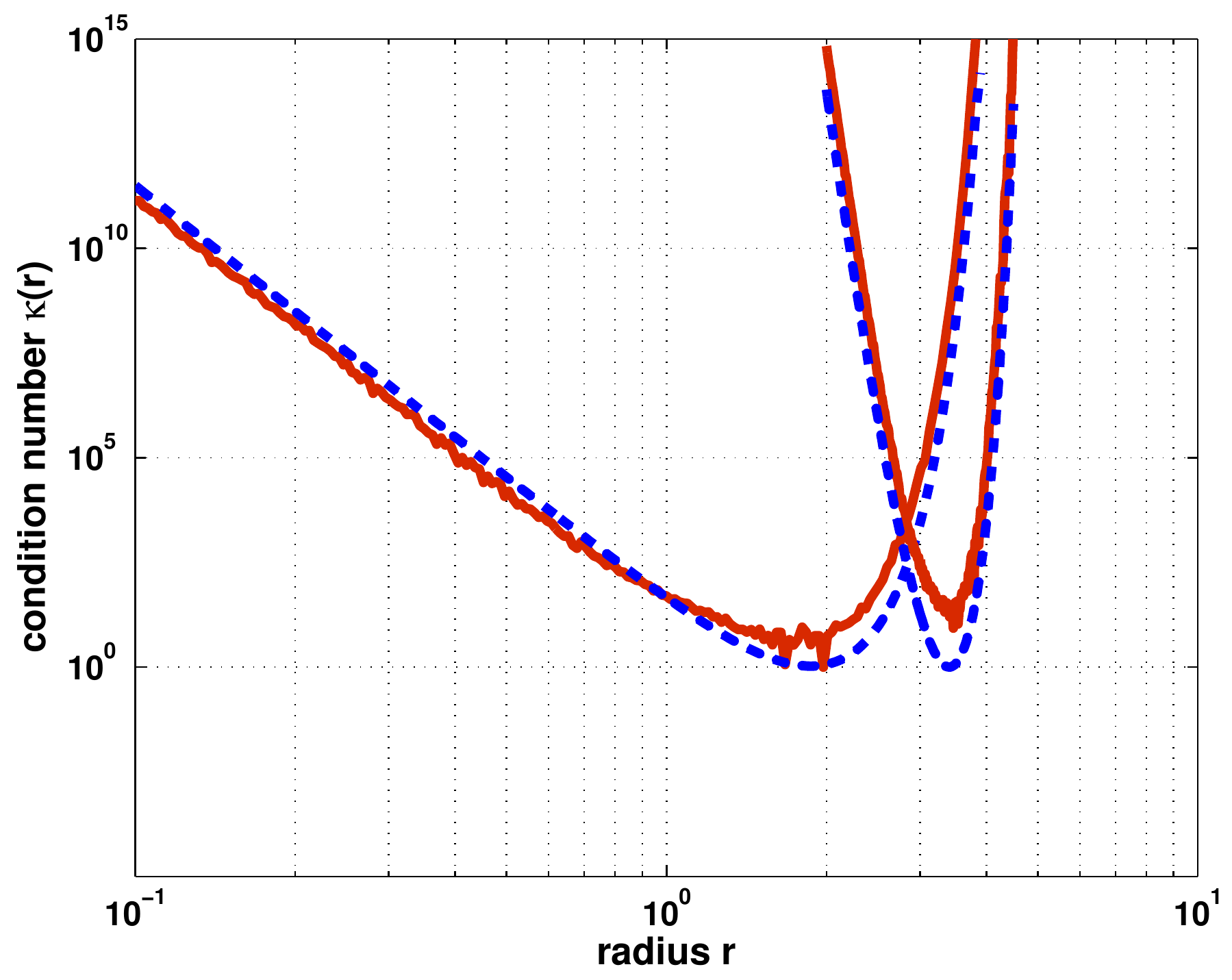}\\*[-1.5mm]
{\tiny c.\;\; $f(z)=\exp(\exp(z)-1)$ (Example~\ref{ex:diamondbell})}
\end{center}
\end{minipage}
\hfil
\begin{minipage}{0.45\textwidth}
\begin{center}
\includegraphics[width=\textwidth]{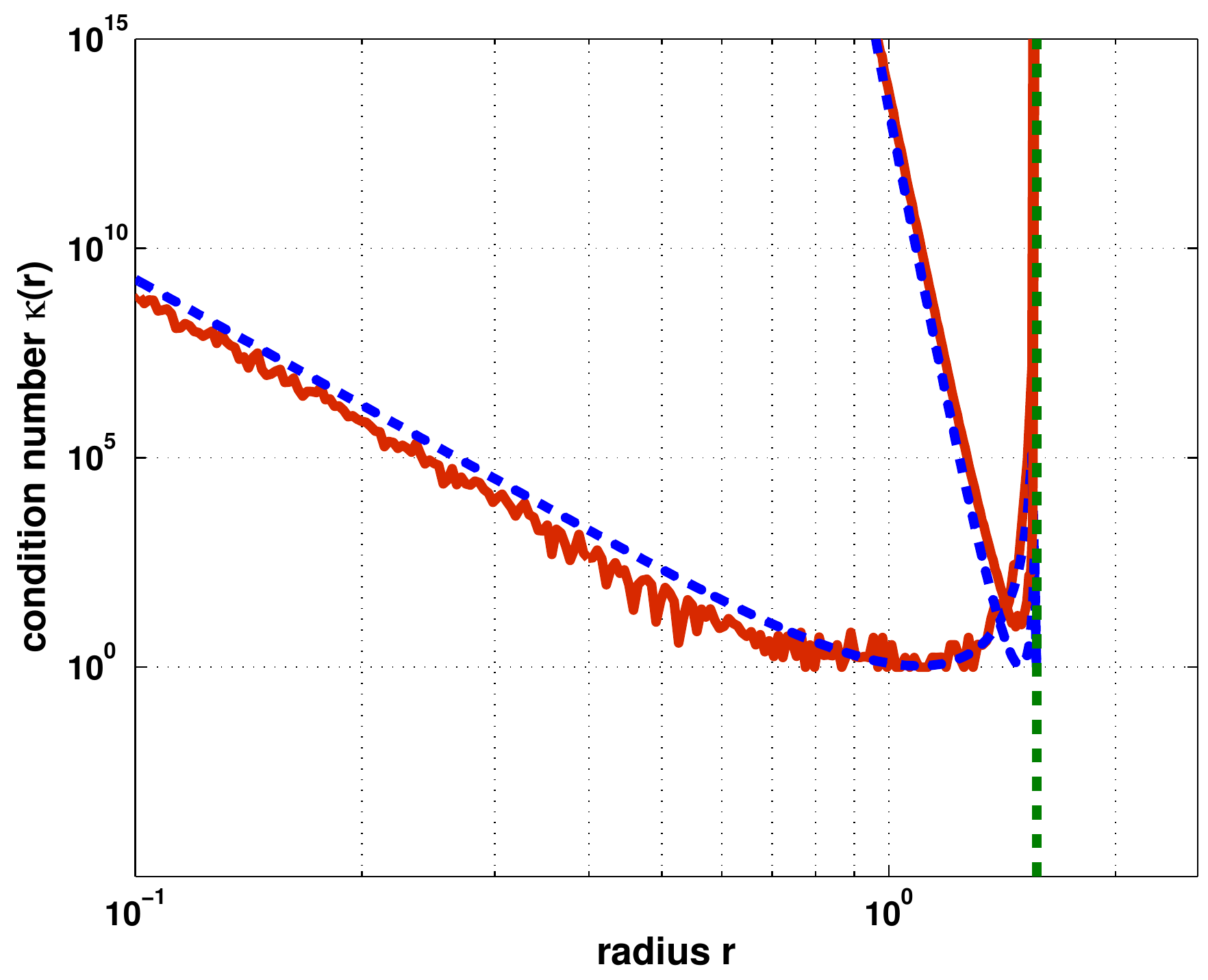}\\*[-1.5mm]
{\tiny d.\;\; $f(z)=\sec(z)^6$ (Example~\ref{ex:sec})}
\end{center}
\end{minipage}\\*[2mm]
\begin{minipage}{0.45\textwidth}
\begin{center}
\includegraphics[width=\textwidth]{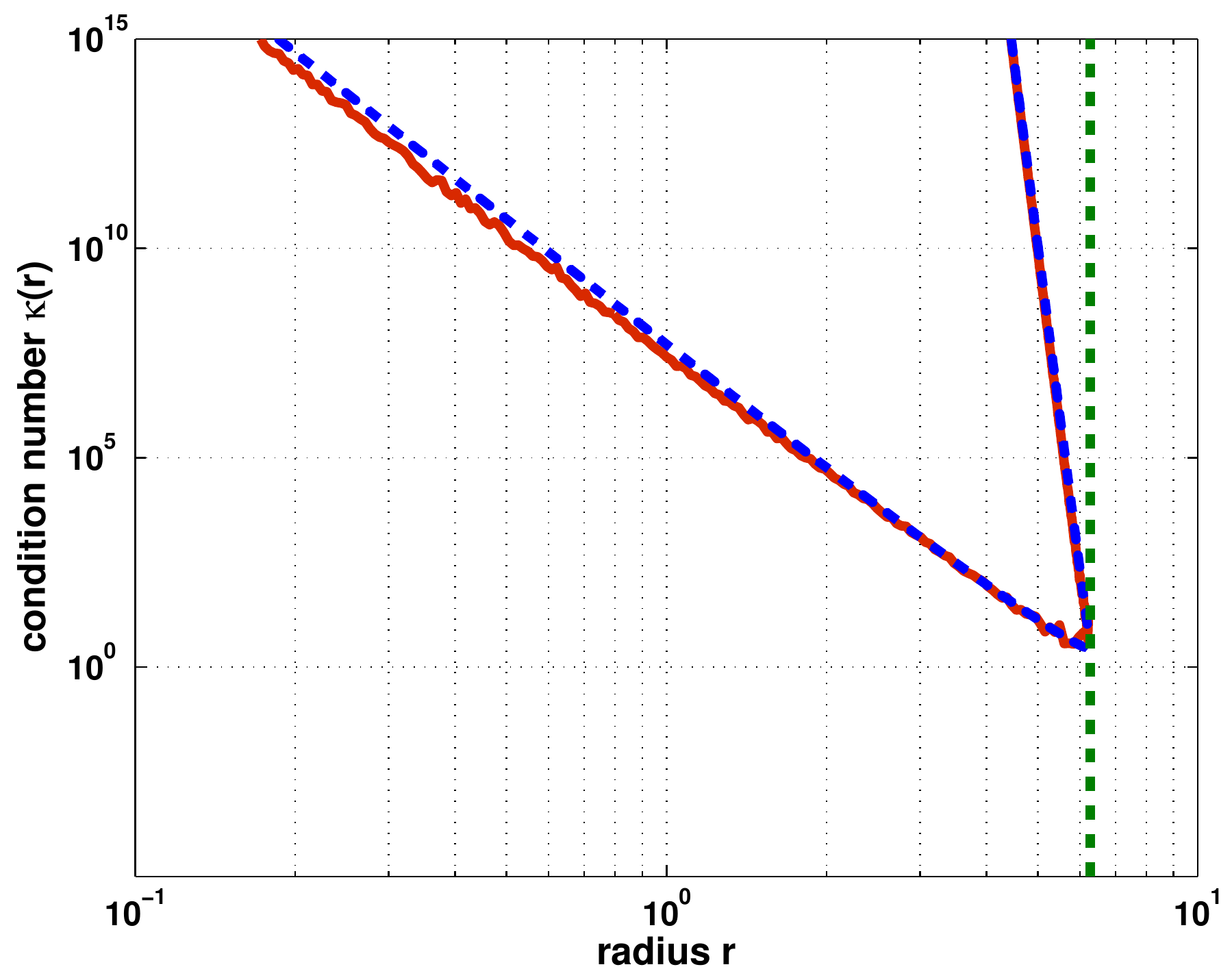}\\*[-1.5mm]
{\tiny e.\;\; $f(z)=z/(e^z-1)$ (Example~\ref{ex:bern})}
\end{center}
\end{minipage}
\hfil
\begin{minipage}{0.45\textwidth}
\begin{center}
\includegraphics[width=\textwidth]{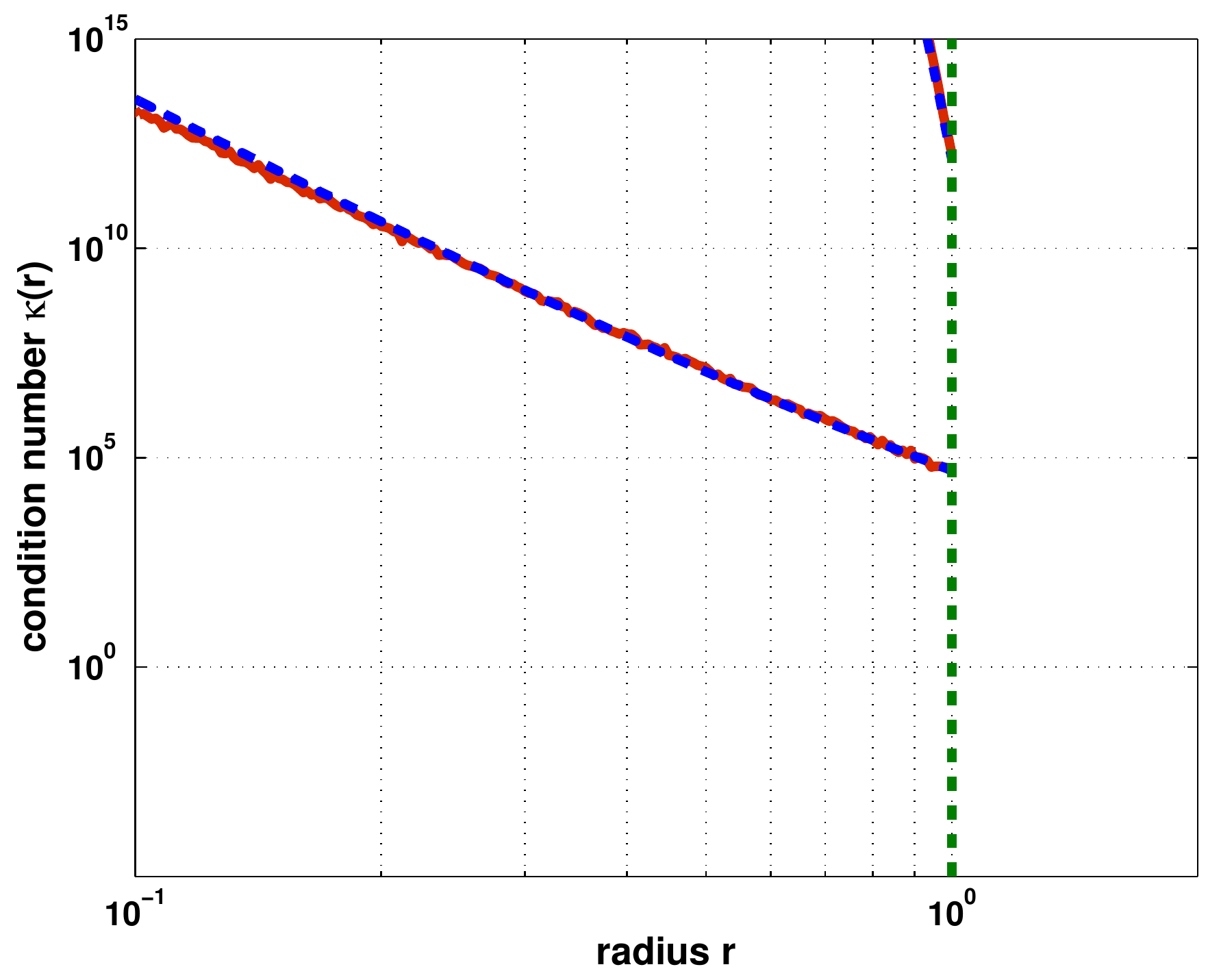}\\*[-1.5mm]
{\tiny f.\;\; $f(z)=(1-z)^{11/2}$ (Example~\ref{ex:fbeta})}
\end{center}
\end{minipage}
\end{center}
\caption{Numerical stability of using Cauchy integrals to compute $f^{(n)}(0)$: plots of the empirical loss of digits (solid red line), that is,
the ratio of the relative error divided by the machine precision, and its prediction by the condition number $\kappa(n,r)$ (dashed blue line)  vs.
the radius $r$. The vertical lines (dashed green) of the last three plots visualize a {\em finite} radius of convergence $R<\infty$. In each plot the results for \emph{two} different orders of differentiation are shown: $n=10$
(the less steep curves starting from the left)
and $n=100$ (the steeper curves starting farther to the right). The number $m$ of nodes of the trapezoidal sum approximation was chosen large enough not to change
the picture. The qualitative shape (convexity in the double logarithmic scale, coercivity and monotonicity properties) of these condition number plots can be
completely understood from the general results in §\ref{sect:radius}.}\label{fig:condcomp}
\end{figure}

As illustrated by the numerical experiments of Figure~\ref{fig:condcomp}, an answer to
the question of choosing a proper radius $r$ becomes absolutely mandatory for derivatives of orders of about $n=100$ and higher: outside a
narrow region of radii there is a complete loss of accuracy. However, rather surprisingly, Figure~\ref{fig:condcomp} also shows that about full accuracy can be obtained for
some functions if we choose the optimal radius that minimizes the loss of accuracy. We observe that such an optimal radius strongly depends on $n$ (and $f$). This strong dependence, together with the complete
loss of accuracy far off the optimal radius, prevents us from using,  for larger $n$,
 just a single radius $r$ to calculate all the leading Taylor coefficients $a_0,\ldots,a_n$ in one go; it thus puts the FFT effectively out of business for the problem at hand.

The goal of this paper is a deeper mathematical understanding of all these effects. In particular, we would like to automate the choice
of the parameters $m$ and $r$ and to predict the possible loss of accuracy. This turns out to be a surprisingly rich and multi-faceted topic, with relations to some classical
results of complex analysis such as Hadamard's three circles theorem (§\ref{sect:upper}) as well as to some more
advanced topics such as the theory of Hardy spaces (§§\ref{sect:radius}/\ref{sect:darboux}), the Wiman--Valiron theory of the maximum term of entire functions (§\ref{sect:prg}), the
Levin--Pfluger theory of the distribution of zeros of entire functions (§\ref{sect:growth}); and with relations to some advanced tools of asymptotic analysis and analytic combinatorics
such as the saddle-point method (§\ref{sect:saddle}) and the concept of $H$-admissibility (§\ref{sect:hayman}).

\subsection*{Outline of the Paper}

To guide the reader through the thicket of this paper, we summarize its most relevant findings:

\begin{itemize}
\item from the point of approximation theory and convergence rates as $m\to\infty$, smaller radii are better than larger ones (§\ref{sect:conv});  there are
useful explicit upper bounds of the number of nodes $m$ in terms of the desired relative error $\epsilon$, the order of differentiation $n$, and the chosen radius $r$ (Eqs.~(\ref{eq:mbound1}) and (\ref{eq:mbound2}));\\*[-3mm]
\item with respect to \emph{absolute} errors, the calculation of the normalized Taylor coefficients $r^n a_n$ is numerically stable for {\em any} radius $r<R$ (§\ref{sect:abs});\\*[-3mm]
\item with respect to \emph{relative} errors, the loss of significant digits is modeled by $\log_{10} \kappa(n,r)$ where $\kappa(n,r)$ denotes the condition number of
the Cauchy integral (§\ref{sect:rel}, see also Figure~\ref{fig:condcomp}), which is {\em independent} of the particular quadrature rule chosen for the actual approximation; it can be estimated on the fly (algorithm given in Figure~\ref{fig:D});\\*[-3mm]
\item $\log \kappa(n,r)$ is a convex function of $\log r$ (Corollary~\ref{cor:Hardy}) and there exists an (essentially unique) optimal radius $r_*(n) = \argmin_r \kappa(n,r)$ that minimizes the loss of accuracy
caused by round-off errors; these
optimal radii form an increasing sequence satisfying $r_*(n) \to R$ as $n\to\infty$ (Theorem~\ref{thm:rn});\\*[-3mm]
\item for finite radius of convergence $R<\infty$, the corresponding optimal condition number $\kappa_*(n)$ blows up if $f$ belongs to the Hardy space~$H^1$ (Theorem~\ref{thm:hardy});
on the other hand, $\kappa_*(n)$ remains essentially bounded if $f$ does \emph{not} belong to the Hardy space $H^1$ and is amenable to Darboux's method (§§\ref{sect:examples} and \ref{sect:darboux}), in which case there are
useful explicit (asymptotic) formulae for $r_*(n)$ and $\kappa_*(n)$ (Eqs.~(\ref{eq:darbouxradius}) and (\ref{eq:darbouxest}));\\*[-3mm]
\item for entire transcendental functions it is more convenient to analyze a certain upper bound $\bar\kappa(n,r)$ of the condition number (§\ref{sect:upper}); this
yields a \emph{unique} radius $r_\diamond(n) = \argmin_r \bar\kappa(n,r)$, called the \emph{quasi-optimal} radius, with a corresponding quasi-optimal condition number
$\kappa_\diamond(n)=\kappa(n,r_\diamond(n)) \geq \kappa_*(n)$; the quasi-optimal radii also form an increasing sequence with $r_*(n) \to R$ as $n\to\infty$ (Theorem~\ref{thm:rdiamondgrowth});\\*[-3mm]
\item for entire functions of perfectly regular growth there is a simple asymptotic formula for $r_\diamond(n)$ in terms of the order and type of such a function (Theorem~\ref{thm:rdiamond});\\*[-3mm]
\item $r_\diamond(n)$ is the modulus of certain saddle points of $|z^{-n} f(z)|$ in the complex plane (Theorem~\ref{thm:zn}); the saddle-point method offers a methodology to obtain
asymptotic results for $\kappa_\diamond(n)$ (§\ref{sect:saddlepointmethod});\\*[-3mm]
\item for entire functions of completely regular growth (satisfying certain conditions on the zeros), the circular contour of radius $r_\diamond(n)$ is
optimal in the sense that it passes the saddle points approximately in the direction of steepest descent (§\ref{sect:growth}); this yields the extremely simple asymptotic condition number bound $\limsup_n \kappa_\diamond(n) \leq \Omega$ where $\Omega$ is the number of maxima of the Phragmén--Lindelöf indicator function of $f$ (Theorem~\ref{thm:condcrg}); in fact,
    there is an explicit asymptotic formula for $\kappa_\diamond(n)$ in terms of a finite sum (Theorem~\ref{thm:crg}) that turns out to yield $\kappa_\diamond(n) \sim 1$  in many relevant examples;\\*[-3mm]
\item for $H$-admissible entire functions we have $\kappa_\diamond(n)\sim 1$ (Corollary~\ref{cor:admissible});\\*[-3mm]
\item for entire functions $f$ with non-negative Taylor coefficients the quasi-optimal radius $r_\diamond(n)$ can
be calculated as the solution of the scalar convex optimization problem $r_\diamond(n) = \argmin_r r^{-n} f(r)$ (Theorem~\ref{thm:nonneg});
we prove $\kappa_\diamond(n) \sim 1$ for a model of a Fredholm determinant with non-negative Taylor coefficients (Eq.~(\ref{eq:kappaoneexponent})).
\end{itemize}

We shall comprehensively discuss many concrete examples and applications throughout this paper: most notably the functions illustrated in Figure~\ref{fig:condcomp},
the functions from the list of the Fornberg quote on p.~\pageref{qt:fornberg}, the functions whose properties are listed in Table~\ref{tab:functions},
the functions $f(z)=(1-z)^\beta$ ($\beta\in\R\setminus\N_0$) (Example~\ref{ex:fbeta}), the generalized hypergeometric functions (Example~\ref{ex:hyper}), the reciprocal Gamma function $f(z)=1/\Gamma(z)$ (§\ref{sect:gamma}),
a generating function from the
theory of random matrices (Examples~\ref{ex:rmt1} and \ref{ex:rmt2}), and a generating function from the theory of random permutations (Example~\ref{ex:gessel}).

\section{Approximation Theory}\label{sect:conv}

\subsection{Convergence Rates}

In this section we recall some basic facts about the convergence of the trapezoidal sums applied to Cauchy integrals on circular contours.
We use the notation
\[
D_r = \{ z \in C : |z| < r\},\qquad C_r = \{z \in \C : |z| = r\},
\]
for (open) disks and circles of radius $r$. Let $f$ be an analytic function as in §\ref{sect:intro}, $\mathcal{P}_m$ be the set of all polynomials of degree $\leq m$ and let
\[
E_m(f;r) = \inf_{p \in \mathcal{P}_m} \|f - p \|_{L^\infty(\overline{D_r})}\qquad (0< r <R)
\]
denote the error of best polynomial approximation of $f$ on the closed disk $\overline{D_r}$. Equivalently, by the maximum modulus principle, we have
\[
E_m(f;r) = \inf_{p \in \mathcal{P}_m} \|f - p \|_{L^\infty(C_r)}\qquad (0< r <R).
\]
The following theorem  belongs certainly to the ``folklore'' of numerical analysis; pinning it down, however, in the literature in exactly the form that we need turned out to be difficult. For  accounts
of the general techniques used in the proof see, for the aliasing relation, \citeasnoun[§§13.2/4]{MR822470} and, for the use of best approximation in estimating quadrature errors, \citeasnoun[§4.8]{MR760629}.
\begin{theorem}\label{thm:conv} Let $f$ be analytic in $D_R$ and $0< r <R$. Then, with the $n$-th Taylor coefficient $a_n$ and its approximation $a_n(r,m)$ as in (\ref{eq:an}) and (\ref{eq:anm}), we have
the aliasing relation
\begin{equation}\label{eq:alias}
r^{n} a_{n}(r,m) = r^{n'} a_{n'}(r,m) \qquad (n \equiv n' \bmod{m})
\end{equation}
and the error estimate
\begin{equation}\label{eq:best}
r^n |a_n - a_n(r,m) | \leq 2 E_{m-1} (f;r) \qquad (0 \leq n < m).
\end{equation}
\end{theorem}
\begin{proof}
The key to this theorem is the observation that $a_n(r,m)$, with $0\leq n < m$, is the {\em exact} Taylor coefficient of the polynomial $p_* \in \mathcal{P}_{m-1}$ that interpolates $f$ in the
nodes $r e^{2\pi i j/m}$ ($j=0,\ldots,m-1$). This fact, and also the aliasing relation, easily follows from the discrete orthogonality
\[
\frac{1}{m} \sum_{j=0}^{m-1} e^{-2\pi i j n/m} e^{2 \pi ij n'/m} =
\begin{cases}
1 & n \equiv n' \bmod{m}; \\*[1mm]
0 & \text{otherwise}.
\end{cases}
\]
Now, by introducing the averaging operators
\begin{equation}\label{eq:operators}
I_n(f;r) = \frac{1}{2\pi} \int_0^{2\pi} e^{-in\theta} f(r e^{i\theta})\,d\theta,\qquad
Q_n(f;r,m) = \frac{1}{m}  \sum_{j=0}^{m-1} e^{- 2\pi i j n/m}f(re^{2\pi i j/m}),
\end{equation}
we have $r^n a_n = I_n(f;r)$ and $r^n a_n(r,m) = Q_n(f;r,m)$. The observation about the approximation being exact for polynomials implies, for $p \in  \mathcal{P}_{m-1}$ and $0\leq n <m$,
that $I_n(p;r) = Q_n(p;r,m)$ and hence
\begin{multline*}
|I_n(f;r) - Q_n(f;r,m)| \\*[1mm] \leq |I_n(f;r) - I_n(p;r)| + |Q_n(p;r,m) - Q_n(f;r,m)|
\leq 2 \| f - p\|_{L^\infty(C_r)}.
\end{multline*}
Taking the infimum over all $p$ finally implies (\ref{eq:best}).
\end{proof}

From the aliasing relation we immediately infer an important basic criterion for the choice of the parameter $m$, namely the
\begin{equation}\label{eq:sampling}
\text{\bf Sampling Condition:} \quad m>n.
\end{equation}
For otherwise, if $m \leq n$, the value $a_n(r,m)$ is just a good approximation of $r^{k-n} a_k$, with $0\leq k < m$ the remainder of dividing $n$ by $m$. However, in general,
 $r^{k-n} a_k$ will differ considerably from $a_n$.

\subsection{Estimates of the Number of Nodes}

To obtain more quantitative bounds of the approximation error as $m\to\infty$, we have a closer look at the error of best approximation.
With $R$ the radius of convergence of the Taylor series (\ref{eq:taylor}) of $f$, the asymptotic geometric rate of convergence of this error is given by \cite[§4.7]{MR0218588}
\begin{equation}\label{eq:geomconv}
\limsup_{m\to\infty} E_m(f;r)^{1/m} = \frac{r}{R}.
\end{equation}
Thus, if we introduce the relative error (assuming $a_n \neq 0$)
\begin{equation}\label{eq:relerr}
\delta_m(n,r) = \frac{|a_n - a_n(r,m)|}{|a_n|},
\end{equation}
we get from (\ref{eq:best}) and (\ref{eq:geomconv}) that
\begin{equation}\label{eq:relerrgeom}
\limsup_{m\to \infty} \delta_m(n,r)^{1/m} \leq \frac{r}{R}.
\end{equation}

\subsubsection{Finite Radius of Convergence}\label{sect:convfinite}
If $R<\infty$, we obtain from (\ref{eq:relerrgeom}) that, for $n$ and $r$ fixed,
\[
\frac{1}{m} \log \delta_m(n,r)^{-1} \geq \log(R/r) + o(1) \qquad (m\to \infty).
\]
Therefore, if $m_\epsilon$ denotes the \emph{smallest} value such that $\delta_m(n,r) \leq \epsilon$ for $m\geq m_\epsilon$ (which implies $\delta_{m_\epsilon} \sim \epsilon$ as $\epsilon\to 0$), we get the asymptotic bound
\begin{equation}\label{eq:mbound1}
m_\epsilon \leq \frac{\log(\epsilon^{-1})}{\log(R/r)} (1+o(1)) \qquad (\epsilon \to 0).
\end{equation}

\begin{example} To illustrate the sharpness of this bound, we consider the function $f(z) = z/(e^z-1)$ for $n=100$, taking the radius $r=6.22$ that is about the optimal one shown in
Fig.~\ref{fig:condcomp}.e.
Here $R=2\pi$ and, for a relative error $\epsilon = 10^{-12}$ (which is, for this particular choice of $r$, large enough to exclude any finite precision effects of the hardware arithmetic), we get
\[
m_\epsilon = 2734 \leq \underbrace{\frac{\log(\epsilon^{-1})}{\log(R/r)}}_{\doteq\, 2733.80}\cdot 1.00007;
\]
thus, the bound (\ref{eq:mbound1}) is an excellent prediction. In Example~\ref{ex:bern} we will see that, for general $n$, the radius $r_n = 2\pi(1-n^{-1})$ is, in terms of numerical stability, about optimal  and
yields the estimate $m_\epsilon \approx n \log \epsilon^{-1}$. That is, for $\epsilon$ fixed, we get $m_\epsilon = O(n)$ as $n\to \infty$, which is the best we could expect in view of the sampling condition (\ref{eq:sampling}).
Further examples of this kind are in §§\ref{sect:examples} and \ref{sect:darboux}.
\end{example}

\subsubsection{Entire Functions}\label{subsec:entire}

If $f$ is entire, that is, $R=\infty$, the estimate (\ref{eq:relerrgeom}) shows that the trapezoidal sums converge even faster than geometric:
\[
\lim_{m\to \infty} \delta_m(n,r)^{1/m} = 0.
\]
In fact, if $f$ is a polynomial of degree $d$, we already know from Theorem~\ref{thm:conv} that the trapezoidal sum is exact for $m>d$, which implies\footnote{Recall that we have assumed
$a_n\neq 0$ in the definition of $\delta_m$, which restricts us to $n\leq d<m$.}  $\delta_m(n,r) = 0$. If $f$ is entire and transcendental, a more detailed resolution of the behavior of $\delta_m$ depends
on the properties of $f$ at its essential singularity in $z=\infty$. For example, entire functions of finite order $\rho>0$ and type $\tau >0$ (for a definition see
§\ref{sect:prg} below)
yield \cite{MR0039122,MR559368}
\begin{equation}\label{eq:bestentire}
\limsup_{m\to\infty} m^{1/\rho} E_m(f;r)^{1/m} = r (e\rho\tau)^{1/\rho}.
\end{equation}
We thus get
\begin{equation}\label{eq:relerrentire}
\limsup_{m\to \infty} m^{1/\rho} \delta_m(n,r)^{1/m} \leq  r (e\rho\tau)^{1/\rho}
\end{equation}
and therefore, for $n$ and $r$ fixed,
\[
\frac{1}{m} \log \delta_m(n,r)^{-1} - \frac{1}{\rho} \log(m/(e\rho\tau))  \geq \log(1/r) + o(1) \qquad (m\to \infty).
\]
Solving for $m_\epsilon$, as defined in §\ref{sect:convfinite}, yields the asymptotic bound
\begin{equation}\label{eq:mbound2}
m_\epsilon \leq \frac{\rho \log(\epsilon^{-1})}{W(\log(\epsilon^{-1})/(e \tau r^{\rho}))}(1+o(1)) \qquad (\epsilon\to 0).
\end{equation}
Here $W(z)$ denotes the principal branch of the Lambert $W$-function defined by the equation $z = W(z) e^{W(z)}$. In Remark~\ref{rm:mboundprg} we will specify this bound, for entire functions of perfectly regular growth,
using a particular radius that is about optimal in the sense of numerical stability.

\begin{example} To illustrate the sharpness of this bound, we consider $f(z) = e^z$ for $n=10$ taking the radius $r=10$, which we read off from Figure~\ref{fig:condcomp}.a to be close to optimal.
Here, the order and type of the exponential functions are $\rho=\tau=1$ (see Table~\ref{tab:functions}) and we get the results of Table~\ref{tab:m_exp} (that were computed using high-precision arithmetic in Mathematica).
As we can see, (\ref{eq:mbound2}) turns out to be a very useful upper bound.

\begin{table}[tbp]
\caption{Sharpness of the bound (\ref{eq:mbound2}) for $f(z)=e^z$ ($n=10$, $r=10$).}
\vspace*{0mm}
\centerline{%
\setlength{\extrarowheight}{3pt}
{\small\begin{tabular}{ccc}\hline
$\epsilon$ & minimal $m_\epsilon$ & $\rho \log \epsilon^{-1}/W(\log\epsilon^{-1}/e\tau r^\rho)$\\ \hline
$10^{-12}$ & $32$ & $\,\,48.21$\\
$10^{-100}$ & $126$ & $140.30$\\
$10^{-1000}$ & $694$ & $706.73$\\*[0.5mm]\hline
\end{tabular}}}
\label{tab:m_exp}
\end{table}

\end{example}

\subsection{Other Quadrature Rules}\label{subsect:rules}

To approximate the Cauchy integral~(\ref{eq:an}), there are other quite effective quadrature rules available besides the trapezoidal sums; examples are
Gauss--Legendre and Clenshaw--Curtis quadrature.
From the point of complexity theory, however, \citeasnoun{MR2170870} have shown (drawing from the pioneering work of Nikolskii in the 1970s)
that the trapezoidal sums are, for the problem at hand, \emph{optimal} in the sense of Kolmogorov.\footnote{That is, the $m$-point trapezoidal sum minimizes, among all $m$-point quadrature formulas,
the worst case quadrature error for the Cauchy integral (\ref{eq:an}) over
all analytic functions whose modulus is bounded  by some constant in an open disk containing $|z|\leq r$.} Hence, for definiteness and simplicity, we stay with trapezoidal
sums in this paper.

It is, however, important to note that the results of this paper apply to other families of quadrature rules as well: first, the estimates (\ref{eq:mbound1}) and (\ref{eq:mbound2})
remain valid if the quadrature error
is bounded by the error of polynomial best approximation (as in (\ref{eq:best}), up to some different constant); which is, e.g.,
the case for Gauss--Legendre and Clenshaw--Curtis quadrature \citeaffixed{Tref08}{see}. Second, the discussion of numerical stability in the next section applies to quadrature rules
with {\em positive} weights in general. In particular, the estimated digit loss (\ref{eq:loss}) depends just on the condition number of the Cauchy integral itself, an
analytic quantity \emph{independent}
of the chosen quadrature rule. Then, starting with §\ref{sect:radius}, optimizing that condition number is the main objective of this paper.

\section{Numerical Stability}\label{sect:cond}

As we have seen in §\ref{sect:intro} and Figure~\ref{fig:condcomp} there are stability issues with using (\ref{eq:anm}) in the realm of finite precision arithmetic.
Specifically, small finite precision errors in the evaluation of the function $f$ can be amplified to large errors in the resulting
evaluation of the sum (\ref{eq:anm}). This error propagation is described by the condition number of the Cauchy integral and depends very much on the chosen
radius $r$ and on the underlying error concept.

\subsection{Absolute Errors}\label{sect:abs}

Any perturbation $\hat f$ of the function $f$ within a bound of the \emph{absolute} error,
\[
\| f - \hat f\|_{L^\infty(C_r)} \leq \epsilon,
\]
induces perturbations $\hat a_n(r)$ and $\hat a_n(r,m)$ of the Cauchy integral (\ref{eq:an}) and of its approximation (\ref{eq:anm}) by the trapezoidal sum. Note that even though
the value of the Cauchy integral does not depend on the specific choice of the radius $r$ (within the range $0 < r <R$), the perturbed value $\hat a_n(r)$ generally does
depend on it. Because both the integral and the sum are re-scaled mean values of $f$, we get the simple estimates
\begin{equation}\label{eq:abserr}
|r^n a_n - r^n \hat a_n(r)| \leq \epsilon,\qquad |r^n a_n(r,m) - r^n \hat a_n(r,m)| \leq \epsilon.
\end{equation}
Thus, the {\em normalized} Taylor coefficients $r^n a_n$ are well conditioned with respect to {\em absolute} errors (with condition number one); a fact that has basically already been
observed by \citeasnoun[p.~670]{362820}. There are indeed applications were absolute errors of normalized Taylor coefficients are a reasonable concept to look at, which then typically leads to a proper choice of the radius $r$.
We give one such example from our work on the numerical evaluation of distributions in random matrix theory \cite{Bornemann2}.

\begin{figure}[tbp]
\begin{center}
\begin{minipage}{0.495\textwidth}
\begin{center}
\includegraphics[width=\textwidth]{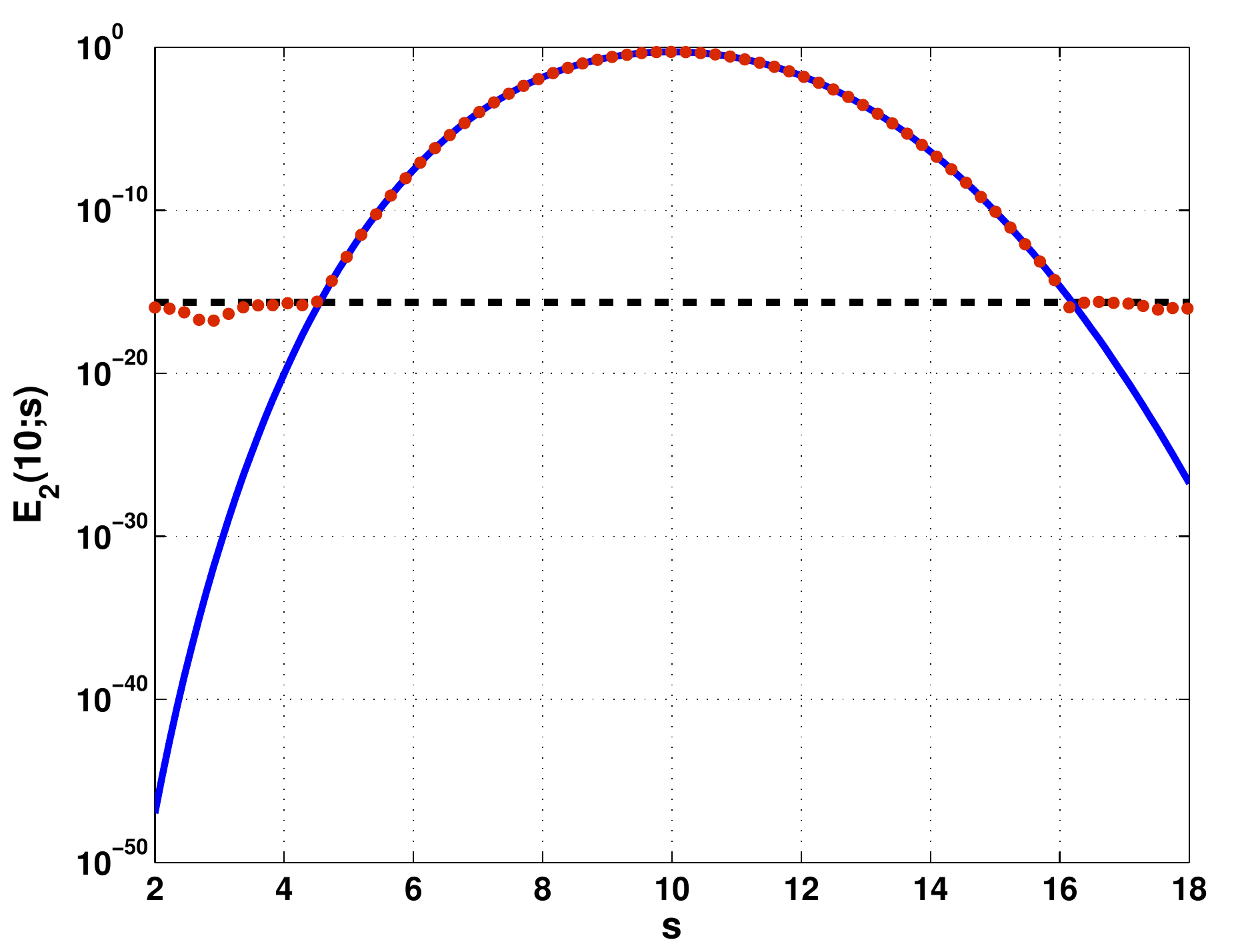}\\*[-1.5mm]
{\tiny a.\;\; gap probability $E_2(10;s)$ of GUE }
\end{center}
\end{minipage}
\hfill
\begin{minipage}{0.495\textwidth}
\begin{center}
\includegraphics[width=\textwidth]{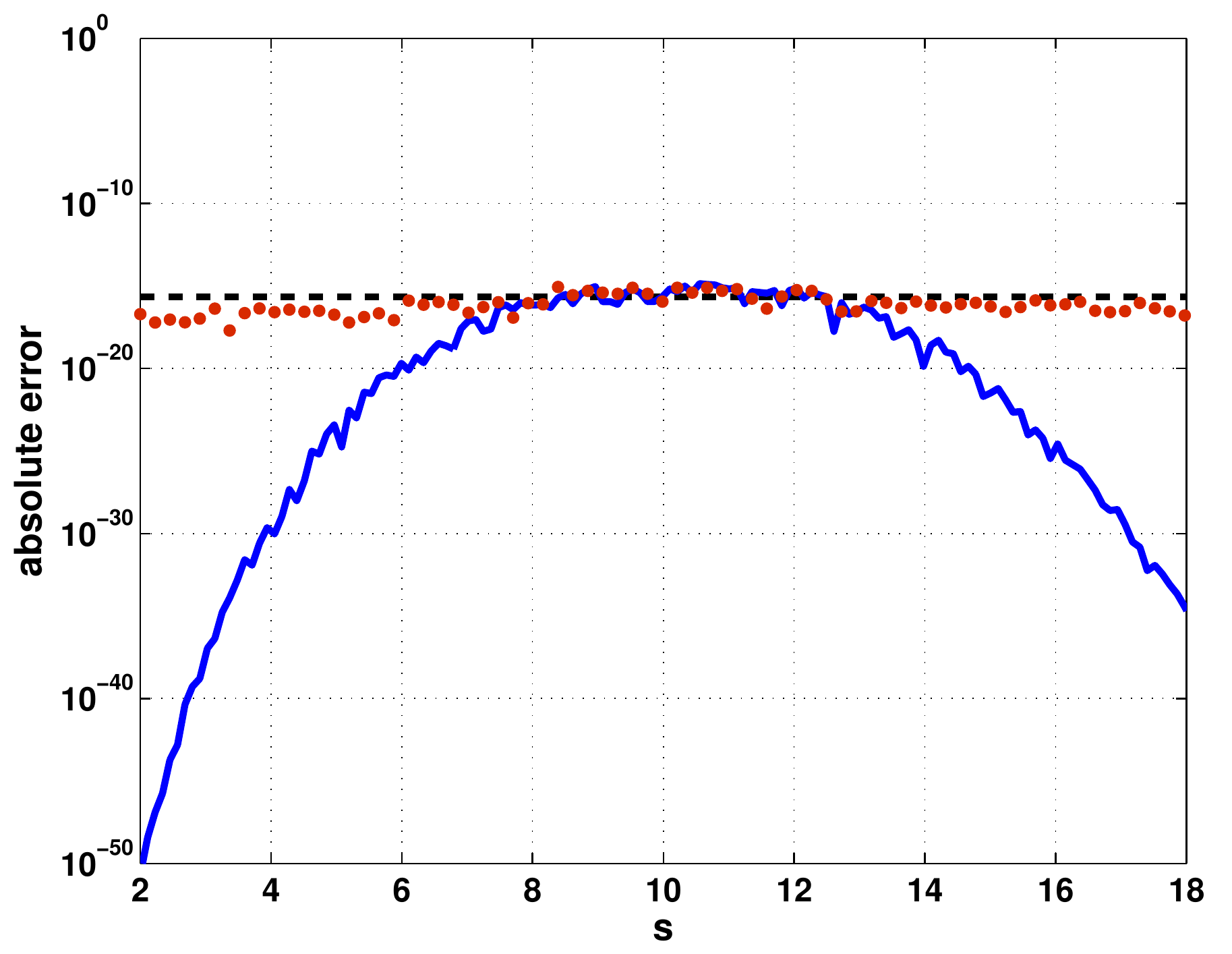}\\*[-1.5mm]
{\tiny b.\;\; absolute error}
\end{center}
\end{minipage}
\end{center}
\caption{Left: the gap probability $E_2(10;s)$ of GUE calculated as the $10$-th Taylor coefficient of a Fredholm determinant; right: the absolute error of the calculation. The dotted lines (red) show the results
for the radius $r=1$; the solid lines (blue) show the results for the quasi-optimal radius $r_\diamond$, which depends on $s$ (see Example~\ref{ex:rmt2} and Figure~\ref{fig:rmt2}). The dashed horizontal lines show the level of
machine precision.}\label{fig:rmt1}
\end{figure}

\begin{example}\label{ex:rmt1}
The sequence of hermitian random matrices $X_N \in \C^{N\times N}$ with entries
\[
(X_N)_{j,j} = \xi_{j,j},\quad (X_N)_{j,k} = \frac{\xi_{j,k}+i \eta_{j,k}}{\sqrt{2}},\quad (X_N)_{k,j} = \frac{\xi_{j,k}-i \eta_{j,k}}{\sqrt{2}}\qquad(j<k),
\]
formed from i.i.d. families of real standard normal random variables $\xi_{i,j}$ and $\eta_{i,j}$, is called the \emph{Gaussian Unitary Ensemble} (GUE).\footnote{In Matlab,
the sequence of commands
\[
\text{\tt X = randn(N) + 1i*randn(N); X = (X+X')/2;}
\]
can be used to sample from the $N \times N$ GUE.} The GUE is of considerable interest since, on the one hand,
various statistical properties of the spectrum $\sigma(X_N)$ enjoy explicit analytic formulas. One the other hand, in the large matrix limit $N \to \infty$,
by a kind of ``universal'' limit law, these properties are often known (or conjectured) to hold for other families of random matrices, too.
An example of such a property concerns the \emph{bulk scaling} $\hat{X}_N= \pi^{-1} N^{1/2}X_N$, for which the mean spacing of the scaled
eigenvalues goes, in the large matrix limit, to one. Basic statistical quantities then considered are the \emph{gap probabilities}\footnote{We denote by $\# S$ the number of elements in a finite set $S$.}
\[
E_2(n;s) = \lim_{N\to\infty} \prob(\#(\sigma(\hat{X}_N) \cap [0,s]) = n),
\]
the probability that, in the large matrix limit, exactly $n$ of the scaled eigenvalues are located in the interval $[0,s]$.
(For Wigner hermitian matrices with a subexponential decay, \citeasnoun{Tao} have, just recently, established the universality of $E_2(0;s)$.)
The generating function of the sequence $E_2(0;s), E_2(1;s), E_2(2,s), \ldots$ is given by the Fredholm determinant of Dyson's sine kernel $K(x,y))=\sinc(\pi(x-y))$ \citeaffixed[§6.4]{MR2129906}{see, e.g.,}, namely,
\[
\sum_{k=0}^\infty E_2(k;s)\, z^k = \det\left(I-(1-z) K |_{L^2(0,s)}\right).
\]
For given values of $n$ and $s$, the strategy to calculate $E_2(n;s)$ is as follows. First, by using the method of \citeasnoun{Bornemann} for the numerical evaluation of
Fredholm determinants, the function
\[
f(z) = \det\left(I-(1-z) K |_{L^2(0,s)}\right)
\]
can be evaluated for {\em complex} arguments of $z$ up to an absolute error of about $\epsilon = 10^{-15}$. Second, the Taylor coefficients $E_2(n;s)$ of $f$ are calculated by means of Cauchy integrals.
Now, since these Taylor coefficients are \emph{probabilities}, the number $1$ is the natural scale for the \emph{absolute} errors, which makes $r=1$ the proper choice for the radius \cite[§4.3]{Bornemann2}.
By (\ref{eq:abserr}), we expect an absolute error of about $\epsilon = 10^{-15}$, which is confirmed by numerical experiments, see Figure~\ref{fig:rmt1}.
However, the figure also illustrates that there is a complete loss
of information about the \emph{tails} (that is, those very small probabilities which are about the size of the error level or smaller). By controlling the radius with respect to relative errors using the method
exposed in the rest of this paper, we were able to increase the accuracy of the tails considerably. The reader should note, however, that in most applications of
random matrix theory the accurate calculation of the tails would be irrelevant. It typically suffices to just identify such small probabilities as being very small;
thus the concept of absolute error is completely appropriate in this example.
\end{example}

There are examples, were small \emph{absolute} errors of the normalized Taylor coefficients $r^n a_n$ are not accurate enough. Because of
the super-geometric growth of the factorial, examples of such cases are the derivatives $f^{(n)}(0) = n!\,a_n$, for high orders $n$. Accuracy will only survive the scaling
by $n!$ if the Taylor coefficients themselves already have small {\em relative} errors.

\subsection{Relative Errors}\label{sect:rel}

We now consider perturbations $\hat f$ of the function $f$ whose relative error can be rendered in the form
\begin{equation}\label{eq:relmodel}
\hat f(r e^{i\theta}) = f(re^{i\theta}) (1+\epsilon_r(\theta)),\qquad \|\epsilon_r\|_\infty \leq \epsilon.
\end{equation}
Such a perturbation induces a perturbation $\hat a_n(r)$ of the Cauchy integral (\ref{eq:an}) which satisfies the straightforward bound \cite[Lemma~9.1]{MR1949263}
\begin{equation}\label{eq:relbound}
\frac{|a_n - \hat a_n(r)|}{|a_n|} \leq \kappa(n,r) \cdot \epsilon
\end{equation}
of its relative error (assuming $a_n \neq 0$), where
\begin{equation}\label{eq:kappa}
\kappa(n,r) = \frac{\displaystyle\int_0^{2\pi} \left|f(r e^{i\theta})\right|\,d\theta}{\left|\displaystyle\int_0^{2\pi} e^{-in\theta} f(r e^{i\theta})\,d\theta\right|} \geq 1
\end{equation}
is the condition number of the Cauchy integral.\footnote{This condition number is completely \emph{independent} of how the Cauchy integral is actually computed.} Note that this number measures the amount of cancelation within the Cauchy integral: $\kappa(n,r) \gg 1$ indicates a large amount
of cancelation, whereas $\kappa(n,r) \approx 1$ if there is virtually no cancelation; see Figure~\ref{fig:exp} for an illustration.

Correspondingly there are perturbations $\hat a_n(r,m)$ of the trapezoidal sum approximations (\ref{eq:anm}) of the Cauchy integrals. They satisfy the same type of bound, namely
\begin{equation}\label{eq:relbounddiscrete}
\frac{|a_n(r,m) - \hat a_n(r,m)|}{|a_n(r,m)|} \leq \kappa_m(n,r) \cdot \epsilon,
\end{equation}
of its relative error (assuming $a_n(r,m) \neq 0$), where
\begin{equation}\label{eq:kappadiscrete}
\kappa_m(n,r) = \frac{\displaystyle\sum_{j=0}^{m-1} \left|f(r e^{2\pi ij /m})\right|}{\left|\displaystyle\sum_{j=0}^{m-1} e^{-2\pi ijn/m} f(r e^{2\pi ij /m})\right|} \geq 1
\end{equation}
is the condition number of the trapezoidal sum \cite[p.~538]{MR1927606}.

\begin{figure}[tbp]
\begin{center}
\begin{minipage}{0.79\textwidth}
{\small
\begin{verbatim}
function [val,err,kappa,m] = D(f,n,r)

fac = exp(gammaln(n+1)-n*log(r));
cauchy = @(t) fac*(exp(-n*t).*f(r*exp(t)));
m = max(n+1,8); tol = 1e-15;
s = cauchy(2i*pi*(1:m)/m); val1 = mean(s); err1 = NaN;
while m < 1e6
    m = 2*m;
    s = reshape([s; cauchy(2i*pi*(1:2:m)/m)],1,m);
    val = mean(s); kappa = mean(abs(s))/abs(val);
    err0 = abs(val-val1)/abs(val); err = (err0/err1)^2*err0;
    if err <= kappa*tol || ~isfinite(kappa); break; end
    val1 = val; err1 = err0;
end
\end{verbatim}}
\end{minipage}
\end{center}
\caption{Matlab implementation of calculating $f^{(n)}(0)$ using the Cauchy integral (\ref{eq:an}) with radius $r$, approximated by
trapezoidal sums. It assumes
$f$ to be evaluated up to an relative error {\tt tol}. The number {\tt m} of nodes is determined by a successive doubling procedure until the estimated
relative error satisfies a bound corresponding to the level of round-off error given by (\ref{eq:relbound}). The error estimate \protect\citeaffixed[Eq.~(4.12)]{805983}{see}
is based on the assumption of a geometric rate of convergence (\ref{eq:geomconv}) which is excellent if $R<\infty$ and an overestimate if $R=\infty$. The initialization
of {\tt m} satisfies the sampling condition (\ref{eq:sampling}). The doubling of nodes is arranged in a way that already computed values of $f$ are re-used.}\label{fig:D}
\end{figure}

If $m$ is chosen large enough such that the trapezoidal sum $a_n(r,m)$ is a good approximation of the Cauchy integral $a_n$, then we typically also have
\[
\frac{1}{m} \displaystyle\sum_{j=0}^{m-1} \left|f(r e^{2\pi ij /m})\right| \approx \frac{1}{2\pi} \int_0^{2\pi} \left|f(r e^{i\theta})\right|\,d\theta.
\]
This is because the integrand $|f(r e^{i\theta})|$ is a smooth \emph{periodic} function of $\theta$ and the trapezoidal sum therefore gives
 excellent approximations of this integral, too.\footnote{By the Euler--Maclaurin
summation formula, the approximation error is of \emph{arbitrary} algebraic order \cite[Thm.~9.16]{MR1949263}.} Moreover, because of positivity, there are no additional
stability issues here. That said, for reasonably large $m$, we have
\[
\kappa_m(n,r) \approx \kappa(n,r)
\]
as long as the computation of $a_n(r,m)$ is not completely unstable.
We use $\kappa(n,r)$  in the theory developed in this paper; but we use $\kappa_m(n,r)$ to monitor stability
in our implementation that is given in Figure~\ref{fig:D}.
In fact, the examples of Figure~\ref{fig:condcomp} show that $\kappa(n,r)$ gives an excellent prediction of the actual loss of (relative)
accuracy in the calculation of the Taylor coefficients; it thus models the dominant effect of the choice
of the radius $r$ (in fact, for \emph{any} stable and accurate quadrature rule):
\begin{equation}\label{eq:loss}
\text{\# lost significant digits} \approx \log_{10} \kappa(n,r).
\end{equation}

\section{Optimizing the Condition Number}\label{sect:radius}

\subsection{General Results on the Condition Number}\label{sect:genres}

It is convenient to rewrite the expression (\ref{eq:kappa}) that defines the condition number briefly as
\begin{equation}\label{eq:kappa2}
\kappa(n,r) = \frac{M_1(r)}{|a_n| r^n},
\end{equation}
using the mean of order $1$ of the modulus of $f$ on the circle $C_r$,
\begin{equation}\label{eq:M1}
M_1(r) = \frac{1}{2\pi} \int_0^{2\pi} \left|f(r e^{i\theta})\right|\,d\theta\,.
\end{equation}
Concerning the properties of $M_1$ we recall the following classical theorem, for the standard proof see \citeasnoun[p.~156]{MR0089895} or \citeasnoun[III.310]{MR0170986}.\nocite{Hardy1915}

\begin{theorem}[Hardy 1915] Let $f$ be given by a Taylor series with radius of convergence $R$. The mean value function $M_1$ satisfies, for $0< r< R$:
\begin{itemize}
\item[(a)] $M_1(r)$ is continuously differentiable;\\*[-3mm]
\item[(b)] if $f \not\equiv 0$, $\log M_1(r)$ is a convex function of $\log r$;\\*[-3mm]
\item[(c)] if $f \not\equiv \const$, $M_1(r)$ is strictly\footnote{The fact that the monotonicity is \emph{strict} has been added to Hardy's theorem by \citeasnoun{MR0033878}.} increasing.
\end{itemize}
\end{theorem}

Because of $\log \kappa(n,r) = \log M_1(r) - \log|a_n| - n  \log r$, there are some immediate consequences for the condition number.

\begin{cor}\label{cor:Hardy} Let $f \not\equiv 0$ be given by a Taylor series with radius of convergence $R$.
Then, for $n$ with $a_n \neq 0$ and for $0< r< R$:
\begin{itemize}
\item[(a)] $\kappa(n,r)$ is continuously differentiable with respect to $r$;\\*[-3mm]
\item[(b)] $\log \kappa(n,r)$ is a convex function of $\log r$.
\end{itemize}
\end{cor}

We now study the behavior of $\kappa(n,r)$ as $r \to 0$ and $r \to \infty$. The first direction is simple and gives us the expected numerical instability for small radii.

\begin{theorem}\label{thm:R0} Let $f \not\equiv 0$ be given by a Taylor series with radius of convergence $R$ and let $a_{n_0}$ be its first non-zero coefficient. Then, for $n>n_0$,
\[
\kappa(n,r) \to \infty \qquad (r \to 0);
\]
but $\kappa(n_0,r) \to 1$.
\end{theorem}
\begin{proof}
From the expansion
\[
M_1(r) = \frac{1}{2\pi} \int_0^{2\pi} |f(r e^{i\theta})|\,d\theta = |a_{n_0}| r^{n_0} + O(r^{n_0+1})\qquad (r\to 0)
\]
we get
\[
\kappa(n,r) \sim \frac{|a_{n_0}|}{|a_n|} r^{n_0-n} \qquad (r\to 0)
\]
which implies both assertions.
\end{proof}

The other direction, $r \to R$, is more involved and depends on further properties of $f$. Let us begin with entire functions ($R=\infty$).

\begin{theorem}\label{thm:entire} Let $f$ be an entire function. If $f$ is transcendental then, for all $n \in \N$,
\[
\kappa(n,r) \to \infty \qquad (r \to \infty).
\]
If $f$ is a polynomial of degree $d$ then this results holds for all $n \neq d$, but $\kappa(d,r)\to 1$.
\end{theorem}
\begin{proof}
Let us assume that, for a particular $m \in \N$,
\[
\liminf_{r\to \infty} \frac{M_1(r)}{r^m} = \liminf_{r\to \infty} \frac{1}{2\pi r^m} \int_0^{2\pi} |f(r e^{i\theta})|\,d\theta < \infty.
\]
Then, for all $n>m$,
\[
0 \leq |a_n| \leq \liminf_{r\to \infty} \frac{1}{2\pi r^n} \int_0^{2\pi} |f(r e^{i\theta})|\,d\theta = 0;
\]
that is, $a_n = 0$; implying that $f$ is a polynomial of degree $d \leq m$. This proves
the assertion for transcendental $f$; and for the cases $n < d$ if $f$ is a polynomial of degree $d$.
The cases $n>d$ follow trivially from $a_n=0$ which implies $\kappa(n,r)=\infty$. Finally, the case $n=d$
gives, because of $|f(z)| = |a_d| |z|^d + O(|z|^{d-1})$ as $z \to \infty$,
\[
\kappa(d,r) = \frac{1}{2 \pi |a_d| r^d} \int_0^{2\pi} |f(r e^{i \theta})|\,d\theta = 1 + O(r^{-1}) \qquad (r \to \infty),
\]
which completes the proof.
\end{proof}

For finite radius of convergence, $R<\infty$, we recall the definition of the Hardy norm (the last equality follows from the monotonicity of $M_1$):
\begin{equation}\label{eq:Hardynorm}
\|f\|_{H^1(D_R)} = \sup_{0<r<R} M_1(r) = \lim_{r\to R} M_1(r).
\end{equation}
If $\|f\|_{H^1(D_R)} < \infty$ the function $f$ belongs to the Hardy space $H^1(D_R)$. From the \emph{strict} monotonicity
and differentiability of $M_1(r)$ we infer that the function
\[
\sigma(r) = \log M_1(r)
\]
satisfies $\sigma'(r)>0$ ($0<r<R$). Since $\log M_1(r)$ is convex in $\log r$, the function $r \sigma'(r)$ is monotonically increasing. Therefore, the limit
\begin{equation}\label{eq:nu}
\nu = \sup_{0<r<R} r \sigma'(r) = \lim_{r \to R} r \sigma'(r) > 0
\end{equation}
exists (with $\nu=\infty$ a possibility, however).

\begin{theorem}\label{thm:Rfin} Let $f$ be given by a Taylor series with finite radius of convergence $R<\infty$. Then, for $a_n \neq 0$,
\[
\lim_{r\to R} \kappa(n,r) = \frac{\|f\|_{H^1(D_R)}}{|a_n| R^n}.
\]
This is finite if and only if $f$ belongs to the Hardy space $H^1(D_R)$. If $n>\nu$ then $\kappa(n,r)$ is strictly decreasing for $0<r<R$;
whereas if $\nu=\infty$ then, for all $n$, $\kappa(n,r)$ is strictly increasing in the vicinity of $r=R$.
\end{theorem}
\begin{proof} The limit can be directly read-off  from (\ref{eq:Hardynorm}). If $n > \nu$, we have
\[
\frac{d}{dr} \log \kappa(n,r) = \sigma'(r) - n r^{-1} \leq (\nu - n) r^{-1}< 0\qquad (0<r<R),
\]
which shows that $\kappa(n,r)$ is strictly decreasing.
If $\nu = \infty$ then $\sigma'(r) \to \infty$ as $r\to R$, which implies
\[
\frac{d}{dr} \log \kappa(n,r) = \sigma'(r) - n r^{-1} \to \infty \qquad (r\to R).
\]
Hence, $\kappa(n,r)$ must be, for $r$ close to $R$, strictly increasing.
\end{proof}

\subsection{The Optimal Radius}\label{sect:optrad}

Optimizing the numerical stability of the Cauchy integrals means, by (\ref{eq:loss}), to choose a radius $r$ that minimizes the condition number $\kappa(n,r)$. The general
results of §\ref{sect:genres} imply that such a minimum actually exists. Indeed, assuming $n>n_0$ (see Theorem~\ref{thm:R0}), $a_n \neq 0$, and that $f$ is not a polynomial,\footnote{Polynomials are addressed by Theorem~\ref{thm:entire}: First, one detects the degree $d$ from $\lim_{r \to \infty} \kappa(d,r) = 1$; then, the cases $n < d$ are  dealt with as for entire transcendental $f$ of order $\rho=0$ (see §\ref{sect:prg}).}
we have the
following ingredients allowing the optimization:
\begin{itemize}
\item {\em continuity}\/: $\kappa(n,r)$ is continuous for $0 < r <R$ (Corollary~\ref{cor:Hardy}) and, if $R<\infty$ and $\|f\|_{H^1(D_R)} < \infty$, can be continuously continued to $r=R$ (Theorem~\ref{thm:Rfin});\\*[-3mm]
\item {\em convexity}\/: $\log \kappa(n,r)$ is convex in $\log r$ (Corollary~\ref{cor:Hardy});\\*[-3mm]
\item {\em coercivity}\/: $\kappa(n,r) \to \infty$ as $r \to 0$ (Theorem~\ref{thm:R0}) and, if $R=\infty$ (Theorem~\ref{thm:entire}) or if $R<\infty$  and $\|f\|_{H^1(D_R)} = \infty$ (Theorem~\ref{thm:Rfin}), as $r\to R$.\\*[-3mm]
\end{itemize}
Hence, by the strict monotonicity of the logarithm, the optimal condition number
\begin{equation}
\kappa_*(n) = \min_{0< r\leq R} \kappa(n,r)
\end{equation}
exists and is taken for the optimal radius\footnote{Since we have no proof of \emph{strict} convexity, we cannot exclude that the minimizing radius happens to be not unique (even though we have not encountered a single such example).
However, because of convexity, the set of all minimizing radii would form a closed interval. We therefore define $r_*(n)$ as the smallest minimizing radius; which, in view of (\ref{eq:relerrgeom}) and (\ref{eq:relerrentire}),
gives the best rates of approximation of the trapezoidal sums.}
\begin{equation}
r_*(n) = \argmin_{0< r \leq R} \kappa(n,r).
\end{equation}
Because the functions $r^{-n} M_1(r)$ and $\kappa(n,r)$ just differ by a factor that is independent of $r$ (namely, $|a_n|$), it is convenient to extend the definition of the optimal radius $r_*(n)$
to the case $a_n = 0$ by setting\footnote{Note that all the \emph{qualitative} results that we stated in §\ref{sect:genres} for $\kappa(n,r)$ hold \emph{verbatim} for $r^{-n}M_1(r)$, independently of whether $a_n\neq0$ or not.}
\begin{equation}\label{eq:roptdef}
r_*(n) = \argmin_{0< r \leq R} r^{-n} M_1(r).
\end{equation}

\begin{theorem}\label{thm:rn} Let the non-polynomial analytic function $f$ be given by a Taylor series with radius of convergence $R$. Then, the sequence $r_*(n)$ satisfies the monotonicity
\[
r_*(n) \leq r_*(n+1)\qquad (n > n_0)
\]
and has the limit $\lim_{n\to\infty} r_*(n) = R$. Furthermore, the case $r_*(n)=R$ is characterized by
\[
r_*(n)=R \quad\Rightarrow\quad \text{$R<\infty$, $\|f\|_{H^1(D_R)}<\infty$, \emph{and} $\nu < \infty$},
\]
and
\[
\text{$R<\infty$, $\|f\|_{H^1(D_R)}<\infty$, \emph{and} $\nu < n$} \quad\Rightarrow\quad r_*(n)=R.
\]
\end{theorem}
\begin{proof} Because of the optimality of $r_*(n)$ and since $M_1(r)>0$, we have, for $0< r < r_*(n)$,
\[
r_*(n)^{-(n+1)} M_1(r_*(n)) \leq r_*(n)^{-1} r^{-n} M_1(r) < r^{-(n+1)} M_1(r).
\]
Hence, the optimal radius $r_*(n+1)$ must satisfy $r_*(n+1) \geq r_*(n)$. This monotonicity implies that $r_0 = \lim_{n\to\infty} r_*(n)$ exists. Let us assume that $r_0 < R$. Then, for each $r_0<r<R$,
by taking the
limit $n \to \infty$ in
\[
r_*(n)^{-1} M_1(r_*(n))^{1/n} \leq r^{-1} M_1(r)^{1/n},
\]
and recalling the continuity of $M_1$, we conclude $r_0^{-1} \leq r^{-1}$. Since this contradicts the choice $r_0<r$, we must have $r_0 = R$. The characterization of $r_*(n)=R$ follows straightforwardly
from Theorem~\ref{thm:Rfin}.
\end{proof}

Bounded analytic functions $f$ that belong to the Hardy space $H^1(D_R)$ are known to possess boundary values \cite[§II.3]{MR628971}; that is, the radial limits
\[
f(R e^{i\theta}) = \lim_{r\to R} f(r e^{i\theta})
\]
exist for almost all angles $\theta$. These boundary values form an $L^1$-function,
\[
\|f\|_{H^1(D_R)} = \frac{1}{2\pi} \int_0^{2\pi} |f(Re^{i\theta})|\,d\theta,
\]
whose Fourier coefficients are just the normalized Taylor coefficients of $f$:
\[
a_n R^n = \frac{1}{2\pi} \int_0^{2\pi} e^{-i n\theta} f(Re^{i\theta})\,d\theta\qquad (n=0,1,2,\ldots).
\]
As the following theorem shows, this fact is bad news for the optimal condition number of such functions for large $n$: it grows beyond all bounds, at a rate that is all the more faster the
more regular the boundary values of $f$ are.

\begin{theorem}\label{thm:hardy} Let the analytic function $f$ be given by a Taylor series with finite radius of convergence $R<\infty$. If $f \in H^1(D_R)$ then
\[
\kappa_*(n) \to \infty\qquad (n\to\infty).
\]
For boundary values of $f$ belonging to the class\footnote{$C^{k,\alpha}$ denotes the functions that are $k$ times continuously differentiable with a $k$-derivative satisfying a Hölder condition of order $0 \leq \alpha \leq 1$.}
$C^{k,\alpha}(\mathcal{C}_R)$ the optimal condition number grows at least as fast as $\kappa_*(n) \geq c n^{k+\alpha}$ for some constant $c>0$.
\end{theorem}
\begin{proof} Since $a_n R^n$ are the Fourier coefficients of the $L^1$-function formed by the radial boundary values of $f$, the Riemann--Lebesgue Lemma implies
\[
a_n R^n \to 0 \qquad (n\to\infty);
\]
with a rate $O(n^{-k-\alpha})$ if these boundary values belong to the class $C^{k,\alpha}$ \citeaffixed[§II.4]{MR0236587}{see, e.g.,}. By Theorem~\ref{thm:rn} we have $r_*(n) \to R$. Hence, for $n \to \infty$,
\[
\kappa_*(n) = \kappa(n,r_*(n)) = \frac{M_1(r_*(n))}{|a_n| r_*(n)^n} \sim \frac{\|f\|_{H^1(D_R)}}{|a_n| r_*(n)^n} \geq  \frac{\|f\|_{H^1(D_R)}}{|a_n| R^n}  \to \infty,
\]
since $\|f\|_{H^1(D_R)} > 0$ (otherwise we would have $f=0$ and $R=\infty$).
\end{proof}

\section{Examples of Optimal Radii}\label{sect:examples}

Qualitatively, the general results of Section~\ref{sect:genres} are nicely illustrated by the examples of Figure~\ref{fig:condcomp}. In this section we study a couple of important examples more quantitatively
for large $n$.

\begin{example}\label{ex:exp} This example illustrates the excellent behavior of certain entire functions; a general theory will be developed in §§\ref{sect:upper}--\ref{sect:positive}.
Here, we consider one of the simplest such functions, namely the exponential function
\[
f(z) = e^z,
\]
which is an entire function ($R=\infty$) with the Taylor coefficients $a_n = 1/n!$.
The mean value of the modulus is explicitly given in terms  of the modified Bessel function of the first kind of order zero \cite[§3.71]{MR0010746},
\[
M_1(r) = \frac{1}{2\pi} \int_0^{2\pi} |\exp(r e^{i\theta})|\,d\theta = \frac{1}{2\pi} \int_0^{2\pi} e^{r\cos\theta}\,d\theta = I_0(r) \qquad (r\geq 0).
\]
Hence, the condition number is
\[
\kappa(n,r) = r^{-n}n!\,I_0(r).
\]
Figure~\ref{fig:exp} illustrates the vast cancelations that occur in the Cauchy integral for large condition numbers $\kappa(n,r)$, that is, for far-from-optimal radii $r$.
Using Stirling's formula and the asymptotic expansion of the modified Bessel function \cite[Eq.~(4.12.7)]{MR1688958},
\[
I_0(r) = \frac{e^r}{\sqrt{2\pi r}}\left(1 + \frac{1}{8r} + \frac{9}{128 r^2} + O(r^{-3})\right)\qquad (r\to \infty),
\]
we get an explicit description of the optimal radius and its condition number: namely, as $n\to\infty$,
\begin{subequations}\label{eq:expdata}
\begin{align}
r_*(n) &= n + \frac{1}{2} + \frac{1}{8n} + O(n^{-2}),\\*[1mm]
\kappa_*(n) &= 1 + \frac{1}{12n} + \frac{7}{288 n^2} + O(n^{-3}).
\end{align}
\end{subequations}
In fact, already the first term of this expansion for $r_*(n)$ gives uniformly excellent condition numbers:
\[
1 < \kappa(n,n) < 1.3 \qquad (n \geq 1).
\]
Thus the derivatives of the exponential function can be calculated to full accuracy
using Cauchy integrals, for \emph{all} orders $n$. On the other hand, Figure~\ref{fig:condcomp}.a shows that, by choosing a fixed radius $r$ independently of $n$, it would be impossible to get condition numbers that remain moderately bounded for orders of differentiation between, say, $1$ and $100$. This explains the failure that \citeasnoun[p.~542]{MR654904b} has documented using his implementation for the exponential function.

\begin{figure}[tbp]
\begin{center}
\begin{minipage}{0.325\textwidth}
\begin{center}
\includegraphics[width=\textwidth]{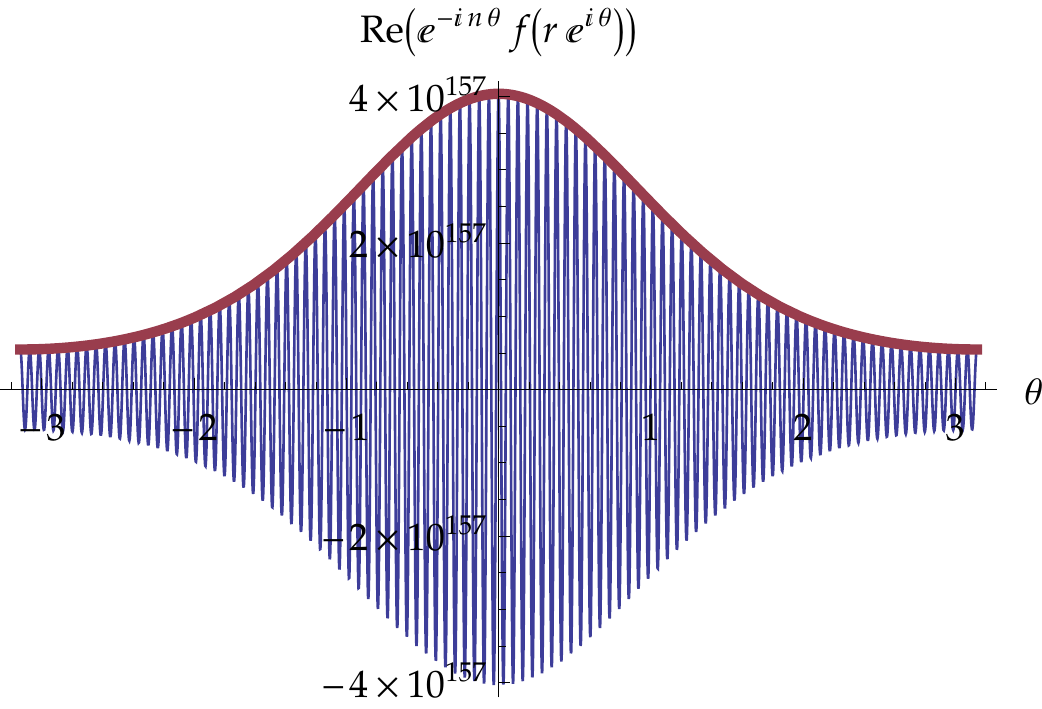}\\*[1.5mm]
{\tiny a.\;\; $r=1$, $\kappa(n,r)=1.182\times 10^{158}$}
\end{center}
\end{minipage}
\hfill
\begin{minipage}{0.325\textwidth}
\begin{center}
\includegraphics[width=\textwidth]{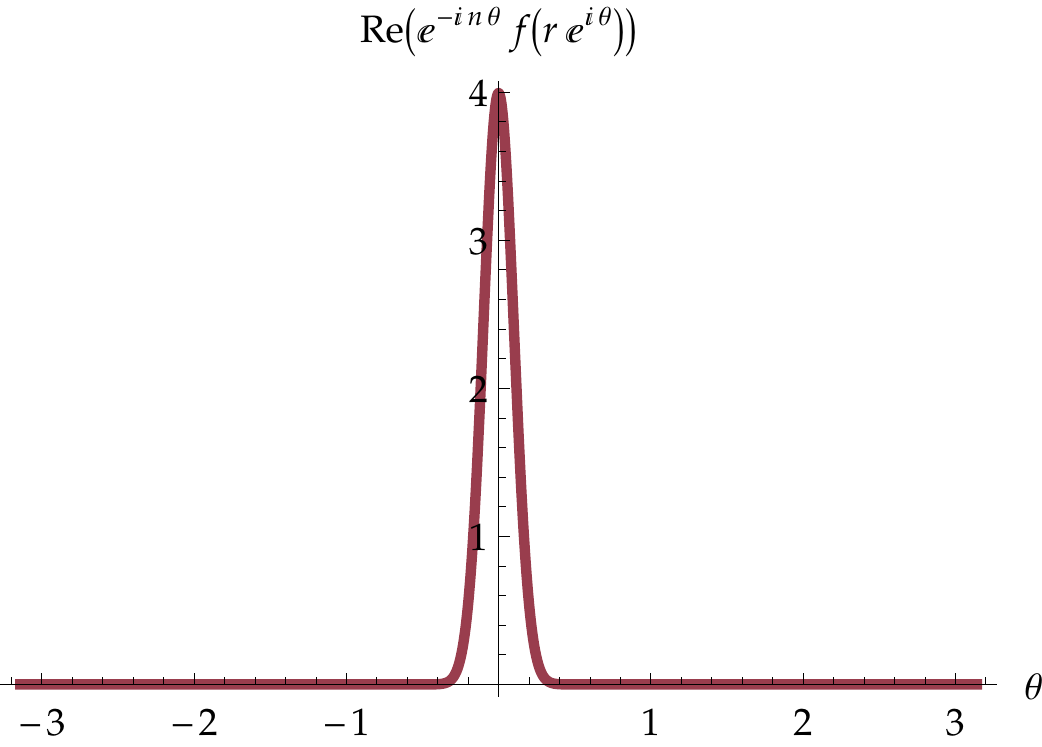}\\*[0mm]
{\tiny b.\;\; $r=100$, $\kappa(n,r)=1.002$}
\end{center}
\end{minipage}
\hfill
\begin{minipage}{0.325\textwidth}
\begin{center}
\includegraphics[width=\textwidth]{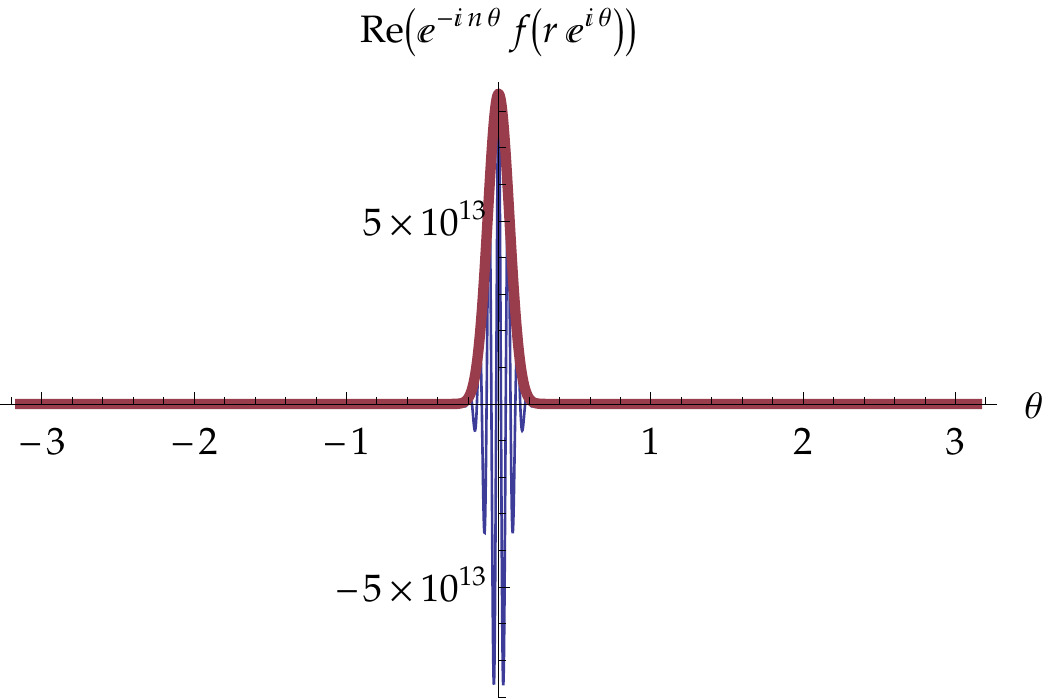}\\*[1.5mm]
{\tiny c.\;\; $r=200$, $\kappa(n,r)=1.502\times 10^{13}$}
\end{center}
\end{minipage}
\end{center}
\caption{Real part (oscillatory, blue line) and absolute modulus (envelope, red line) of the integrand of the Cauchy integral (\ref{eq:an}) for various radii~$r$; $f(z)=e^z$, $n=100$. Clearly visible
is the huge amount of cancelation if the condition number $\kappa(n,r)$ is large. Note that this is \emph{not} an issue of frequency, which is moderate and perfectly dealt with by the
sampling condition (\ref{eq:sampling}), but rather an issue of amplitude.}\label{fig:exp}
\end{figure}

\end{example}

\begin{example}\label{ex:fbeta} In preparation of §\ref{sect:darboux} we consider the family
\[
f_\beta(z) = (1-z)^\beta \qquad (\beta \in \R\setminus \N_0)
\]
of analytic functions, which are not polynomials for the values of $\beta$ considered. The radius of convergence of the Taylor series is $R=1$ and the Taylor coefficients are given by
\[
a_n = \binom{n-\beta-1}{n}\qquad (n=0,1,2,\ldots).
\]
By a simple transformation of Euler's integral representation \cite[Thm.~2.2.1]{MR1688958}, the mean value of the modulus can explicitly be expressed in terms
of the hypergeometric function $_2F_1(a,b;c;z)$:
\begin{multline}\label{eq:M2F1}
M_1(r) = \frac{1}{2\pi}\int_0^{2\pi} |1-r e^{i\theta}|^\beta \,d\theta = \frac{1}{2\pi}\int_0^{2\pi} \left(\sqrt{1+r^2-2r\cos\theta}\right)^\beta\,d\theta\\*[2mm]
= (1+r)^\beta\/\; _2F_1\left(\frac{1}{2},-\frac\beta 2; \,1;\, \frac{4r}{(1+r)^2}\right)\qquad (0\leq r < 1).
\end{multline}
The classical results of Gauss \cite[Thms.~2.1.3/2.2.2]{MR1688958} about the hypergeometric function $_2F_1(a,b;c;z)$ as $z\to 1$ imply, as $r \to 1$ from below,
\begin{equation}\label{eq:gauss}
M_1(r) \sim
\begin{cases}
\dfrac{2^\beta\,\Gamma\left(\frac{\beta+1}{2}\right)}{\sqrt{\pi}\,\Gamma\left(\frac{\beta}{2}+1\right)}&\qquad (\beta>-1);\\*[7mm]
\dfrac{1}{\pi}\log\left(\dfrac{1}{1-r}\right)&\qquad (\beta=-1);\\*[5mm]
\dfrac{\Gamma\left(-\frac{\beta+1}{2}\right)}{2\sqrt{\pi}\,\Gamma\left(-\frac{\beta}{2}\right)}\,(1-r)^{\beta+1}&\qquad (\beta<-1).
\end{cases}
\end{equation}
Therefore, we have to distinguish three cases.

\subsubsection*{Case I: $\beta>-1$} Here, (\ref{eq:gauss}) implies that $f_\beta$ belongs to the Hardy space $H^1(D_1)$ with norm
\[
\|f_\beta\|_{H^1(D_1)} = \lim_{r\to1} M_1(r) = \dfrac{2^\beta\,\Gamma\left(\frac{\beta+1}{2}\right)}{\sqrt{\pi}\,\Gamma\left(\frac{\beta}{2}+1\right)} \geq 1.
\]
(The estimate from below follows from the fact that $\|f_\beta\|_{H^1(D_1)}$ is a convex and coercive function of $\beta$, taking its minimum at $\beta=0$.) The constant $\nu$, defined in (\ref{eq:nu}), can be computed from
\begin{multline*}
M_1'(r) = \beta (1+r)^{\beta-3} \left((1+r)^2\/\; _2F_1\left(\frac{1}{2},-\frac\beta 2; \,1;\, \frac{4r}{(1+r)^2}\right) \right.\\*[2mm]
\left. +(r-1)\/\; _2F_1\left(\frac{3}{2},1-\frac\beta 2; \,2;\, \frac{4r}{(1+r)^2}\right)\right)
\end{multline*}
to have the value
\[
\nu = \lim_{r\to 1} r \sigma'(r) = \frac{M_1'(1)}{M_1(1)} = \frac{\beta}{2}.
\]
Thus, by Theorems~\ref{thm:Rfin} and \ref{thm:rn}, the condition number $\kappa(n,r)$ is strictly decreasing for $n>\beta/2$ (see Figure~\ref{fig:condcomp}.f for an example); hence
\begin{subequations}
\begin{equation}
r_*(n) = 1 \qquad (n > \beta/2),
\end{equation}
which induces (by Stirling's formula)
\begin{equation}
\kappa_*(n) = \frac{\|f_\beta\|_{H^1(D_1)}}{\left|\binom{n-\beta-1}{n}\right|}
 \geq \frac{1}{\left|\binom{n-\beta-1}{n}\right|} \sim |\Gamma(-\beta)| n^{\beta+1} \to \infty \qquad (n\to\infty).
\end{equation}
\end{subequations}
This means that for each radius $r$ there will be a complete loss of digits for $n$ large enough (e.g., there is already a more than 12 digits loss for $\beta=11/2$ and $n=100$, see Figure~\ref{fig:condcomp}.f); an effect that will be the more pronounced the larger $\beta$ is. Note that a larger $\beta$ corresponds to
higher order real differentiability at the branch point $z=0$;
an observation which is in accordance with Theorem~\ref{thm:hardy} and  which helps to explain the failure that \citeasnoun[p.~542]{MR654904b} has documented using his implementation for such functions.

\subsubsection*{Case II: $\beta=-1$} Now, (\ref{eq:gauss}) shows that $f_{\beta}$ does not belong to the Hardy space $H^1(D_1)$ anymore. Thus, by Theorems~\ref{thm:Rfin} and \ref{thm:rn},
we have $0< r_*(n) < 1$ with $r_*(n) \to 1$ as $n \to \infty$. Because of $a_n=1$, and by (\ref{eq:gauss}) once more, there is the asymptotic expansion
\[
\kappa(n,r) = \frac{M_1(r)}{r^n} \sim \frac{1 }{\pi r^n} \log\left(\dfrac{1}{1-r}\right) \qquad (r\to 1).
\]
It is now a more or less straightforward exercise in asymptotic analysis \cite[Chap.~2]{MR671583} to get from here to the following expansions of the optimal radius and condition
number: as $n\to \infty$,
\begin{subequations}
\begin{align}
r_*(n) & = 1 - \frac{1}{n \log n} + O\left(\frac{\log\log n}{n(\log n)^2}\right),\\*[1mm]
\kappa_*(n) &= \frac{\log n}{\pi}  + O(\log\log n).
\end{align}
\end{subequations}
This logarithmic growth is very moderate; indeed, one has
\[
1 < \kappa\left(n,1-\frac{1}{n\log n}\right) < 4.8 \qquad (3 \leq n \leq 10\,000),
\]
which means that less than one digit is lost for a significant range of $n$.

\begin{remark}\label{rem:subopt} In practice it is not always advisable to use the optimal radius: a small sacrifice in accuracy might considerably speed up the
approximation of the Cauchy integral by the trapezoidal sum. In fact, if we recall (\ref{eq:mbound1}), we realize that the near-optimal choice $r_n = 1-(n \log n)^{-1}$ would
need about\footnote{Note that, by (\ref{eq:mbound1}) and (\ref{eq:mbound2}), estimates of the form $m_\epsilon \approx \cdots$ include, among other approximations, a factor of
the form $1+o(1)$ as $\epsilon \to 0$. Therefore, one should not expect too much precision of such estimates, in particular not if additionally finite precision effects come into play for
$\epsilon$ close to machine precision. Even then, however, in all the examples of this paper, we observe ratios of the actual values of $m_\epsilon$ to their estimates
that are smaller than 1.3; thus, these rough estimates are, in practice, quite useful devices to predict the actual computational effort.}
\begin{equation}\label{eq:mboundrn}
m_\epsilon \approx n \log n \cdot \log \epsilon^{-1}
\end{equation}
nodes to achieve an approximation of relative error $\epsilon$. We can actually get rid of the factor $\log n$ here if we use the sub-optimal radius $\tilde r_n = 1- \alpha n^{-1}$ ($\alpha>0$) instead.
Asymptotically, as $n \to \infty$, the condition number is then
\begin{equation}\label{eq:f1subopt}
\kappa(n,\tilde r_n) \sim \frac{1}{\pi \tilde r_n^n} \log\left(\dfrac{1}{1-\tilde r_n}\right) = \frac{1}{\pi(1-\alpha n^{-1})^n} \log(n/4)\sim \frac{e^\alpha}{\pi}\log n,
\end{equation}
and therefore still of logarithmic growth: compared to $r_n$ we additionally sacrifice just about $\log_{10} e^\alpha \doteq 0.43\,\alpha $ digits, independently of $n$. However, the corresponding number of nodes now grows like
\begin{equation}\label{eq:mboundrtilde}
m_\epsilon \approx \frac{n}{\alpha} \log \epsilon^{-1},
\end{equation}
which is about an $\alpha \log n$ improvement in speed.

To be specific, let us run some numbers for $n=100$: Since $\kappa(100,r_{100}) \doteq 3.25$, we are about to lose $0.51$ digits using $r_n$; in hardware arithmetic we could therefore strive for a relative error of $\epsilon=2\times 10^{-15}$. By~(\ref{eq:mboundrn}) we have to take about $m_\epsilon \approx 16\,000$ nodes; actually, a computation with $m=20\,000$ gives
us the relative error $2.6 \times 10^{-15}$. In contrast, for $\alpha=4$, we have $\kappa(100,\tilde r_{100}) \doteq 101.63$, so we are about to lose $2.0$ digits using $\tilde r_n$; we could therefore strive
for a relative error of $\epsilon=5\times 10^{-14}$ here. Because of (\ref{eq:mboundrtilde}) we now have to take just about $m_\epsilon\approx 800$ nodes; and indeed, a computation with $m=800$ gives
us the relative error $4.9 \times 10^{-14}$. Thus, sacrificing just a little more than one digit cuts the number of nodes by a factor of 25 (the prediction was
$4\log 100 \doteq 18.4$).

\end{remark}

\subsubsection*{Case III: $\beta<-1$} As for $\beta=-1$, (\ref{eq:gauss}) shows that these $f_{\beta}$ do not belong to the Hardy space $H^1(D_1)$. Thus, by Theorems~\ref{thm:Rfin} and \ref{thm:rn},
we have $0< r_*(n) < 1$ with $r_*(n) \to 1$ as $n \to \infty$; hence, (\ref{eq:gauss}) implies the asymptotic expansions
\begin{subequations}
\begin{equation}
r_*(n)  = 1 + \frac{\beta+1}{n} + O(n^{-2}) \qquad(n\to\infty)
\end{equation}
of the optimal radius and
\begin{align}
\kappa_*(n) &\sim \frac{1}{2\sqrt{\pi} \left|\binom{n-\beta-1}{n}\right|} \frac{\Gamma\left(-\frac{\beta+1}{2}\right)}{\Gamma\left(-\frac{\beta}{2}\right)} \frac{\left(-\frac{\beta+1}{n}\right)^{\beta+1}}{\left(1+\frac{\beta+1}{n}\right)^{n}}\notag\\*[2mm]
&\sim \frac{(2e)^{-\beta-1}(-\beta-1)^\beta}{\pi} \,  \Gamma\left(\frac{1-\beta}{2}\right)^2 = c_\beta \label{eq:cbeta}
\end{align}
of the optimal condition number.
\end{subequations}
Note that there is no explosion in $n$ and that $c_\beta \to 1$ monotonically from above as $\beta\to-\infty$. Quantitatively we have
\[
1 \leq c_\beta \leq 2 \qquad (\beta \leq -1.362),
\]
that is, we are just about to lose one binary digit of accuracy within this range of values of $\beta$ (for large $n$). Finally,
to accomplish an approximation of relative error $\epsilon$ by using a trapezoidal sum, we would need, in view of
(\ref{eq:mbound1}), about the following number of nodes:
\begin{equation}\label{eq:mboundrn2}
m_\epsilon \approx \frac{n}{-\beta-1} \cdot \log \epsilon^{-1}.
\end{equation}

Here are some actual numbers: for $\beta=-6$, $n=100$, $r_n = 1 +(\beta+1) n^{-1}$, and the accuracy requirement $\epsilon=10^{-15}$, we get
\[
\kappa(n,r_n) \doteq 1.0769,\qquad m_\epsilon \approx 700.
\]
In fact, a computation in hardware arithmetic secures a relative error of $4\times 10^{-15}$ using $m=900$ nodes.
\end{example}

\begin{example} We analyze a further example that \citeasnoun[p.~542]{MR654904b} has documented to fail his implementation:
\[
f(z) = (1+z)^{10} \log(1+z)
\]
with radius of convergence $R=1$. Having norm $\|f\|_{H^1(D_1)} \doteq 180.14$, this function belongs to the Hardy space $H^1(D_1)$. Theorem~\ref{thm:hardy} gives $\kappa_*(n)\to\infty$ as $n\to \infty$. More
quantitatively we get, by Theorem~\ref{thm:Rfin},
\[
\kappa(n,r) \geq \kappa(n,1) \doteq \frac{180.14}{|a_n|}\qquad (n>\nu \doteq 5.727).
\]
The asymptotics (the first equality is valid for $n\geq 11$)
\[
a_n = \frac{(-1)^{n-1}}{11\binom{n}{11}}\sim \frac{(-1)^{n-1} 10!}{n^{11}}\qquad (n\to \infty)
\]
implies
\[
\kappa(n,r) \geq \kappa(n,1) \doteq 1981.57 \binom{n}{11} \sim 5.46 \times 10^{-4} \cdot n^{11}\qquad (n\to \infty).
\]
For instance, $n=50$ gives $\kappa(50,r) \geq \kappa(50,1) \doteq 7.4\times 10^{13}$; meaning that a loss of more than about 14 digits is unavoidable here.
\end{example}

\begin{example} The final example of this section is also taken from the list of failures documented by \citeasnoun[p.~542]{MR654904b}:
\[
f(z) = 10^6 + \frac{1}{1-z}
\]
with radius of convergence $R=1$. This function is a perturbation of the function $f_{-1}$ from Example~\ref{ex:fbeta}. Denoting by $M_1(f_{-1};r)$ the mean value of the modulus of $f_{-1}$ we get,
using (\ref{eq:gauss}),
\[
M_1(r) \leq 10^6 + M_1(f_{-1};r) \sim 10^6 + \frac{\log\left(\frac{1}{1-r}\right)}{\pi} \qquad (r\to 1).
\]
The sub-optimal choice $r_n = 1-n^{-1}$ (see Remark~\ref{rem:subopt}) yields
\[
\kappa(n,r_n) \leq 10^6 e+ \frac{e}{\pi} \log n \approx 3\times 10^6 \qquad (1\leq n \leq 10^{100\,000}).
\]
Hence, we expect a loss of (at most) about $6.5$ digits throughout this huge range of $n$. The estimate is, in fact, quite sharp: for instance, $n=100$ yields
\[
\kappa(100,r_{100}) \doteq 2.7\times 10^6.
\]
An actual calculation using a trapezoidal sum with $m=4096$ nodes yields a relative error of $3.13\times 10^{-10}$ which corresponds to a loss of a little more than 6 digits in hardware arithmetic.
\end{example}

\section{Functions Amenable to Darboux's Theorem}\label{sect:darboux}

Example \ref{ex:fbeta} contains, in fact, all the information that is needed to address a large class of analytic functions:
\[
f(z) = (1-z)^\beta v(z)\qquad (\beta \in \R\setminus\N_0),
\]
where $v(z)$ is analytic in a neighborhood of $\overline{D_1}$, $v(1)\neq 0$. In particular, the radius of convergence is $R=1$. By Darboux's theorem \cite[Thm.~5.3.1]{MR2172781}, the
Taylor coefficients are asymptotically given by
\begin{equation}\label{eq:darboux}
a_n = v(1) \frac{n^{-\beta-1}}{\Gamma(-\beta)} \,(1+O(n^{-1}))\qquad (n\to\infty).
\end{equation}
Hence, the condition number is asymptotically described by
\begin{equation}\label{eq:darbouxcond}
\kappa(n,r) \sim \frac{M_1(r)}{|v(1)|r^n} |\Gamma(-\beta)| n^{\beta+1}\qquad (n\to\infty).
\end{equation}
The mean value of the modulus satisfies, as $r\to 1$, (compare with (\ref{eq:gauss}))
\[
M_1(r) = \frac{1}{2\pi} \int_0^{2\pi} |1-r e^{i\theta}|^\beta \cdot |v(r e^{i\theta})|\,d\theta \sim
\begin{cases}
c & \qquad (\beta>-1);\\*[2.5mm]
c\log\left(\dfrac{1}{1-r}\right)& \qquad (\beta=-1);\\*[4mm]
c (1-r)^{\beta+1} & \qquad (\beta<-1).
\end{cases}
\]
Here, $c$ denotes some positive constant that depends on $v$ and $\beta$. This implies, such as in Example~\ref{ex:fbeta}, that, as $n\to\infty$,
\begin{equation}\label{eq:darbouxradius}
r_*(n)
\begin{cases}
=  1 & \quad (\beta>-1);\\[2.5mm]
\sim  1- \dfrac{1}{n\log n} & \quad (\beta=-1); \\*[4mm]
\sim  1 + \dfrac{\beta+1}{n}  & \quad (\beta<-1);
\end{cases}\qquad
\kappa_*(n) \sim
\begin{cases}
c\, n^{\beta+1} & \quad (\beta>-1);\\[1mm]
c\,  \log n & \quad (\beta=-1); \\*[1mm]
c & \quad (\beta<-1).
\end{cases}
\end{equation}
For large orders of differentiation, this means that, once more in accordance with Theorem~\ref{thm:hardy}, the Hardy space case $\beta>-1$ yields polynomial growth of the condition numbers; whereas for $\beta=-1$
we get just logarithmic growth and for $\beta<-1$ there is a uniform bound of the condition number.

To address the last two cases more quantitatively, we can estimate the mean modulus by
\[
M_1(r) \leq \frac{\|v\|_{H^\infty(D_1)}}{2\pi} \int_0^{2\pi} |1-r e^{i\theta}|^\beta \,d\theta\qquad (0<r<1)
\]
with the help of yet another Hardy space norm, defined by
\[
\|f\|_{H^\infty(D_r)} = \esssup_{0\leq \theta\leq 2\pi} |f(r e^{i\theta})|.
\]
Denoting the condition number of the Cauchy integral for the function $f_\beta$ by $\kappa(f_\beta;n,r)$ (recall that this expression can be evaluated in terms of the hypergeometric function, see (\ref{eq:M2F1})), we thus obtain a useful estimate of the condition number itself, namely
\begin{equation}\label{eq:darbouxest}
\kappa(n,r) \leq \frac{\|v\|_{H^\infty(D_1)}}{|v(1)|}\, \kappa(f_\beta;n,r)\qquad (0<r<1).
\end{equation}
Note that there is nothing special about $R=1$ here. For functions of the form
\[
f(z) = (z_0 - z)^\beta v(z)\qquad (\beta\in\R\setminus \N_0)
\]
with $|z_0|=R$, $v(z)$ analytic in a neighborhood of $\overline{D_R}$, and $v(z_0)\neq 0$ we get accordingly
\begin{equation}\label{eq:darbouxest2}
\kappa(n,r) \leq \frac{\|v\|_{H^\infty(D_R)}}{|v(z_0)|}\, \kappa(f_\beta;n,r/R)\qquad (0<r<R).
\end{equation}
If there is more than one singularity on the circle $C_R$, we would have to use symmetry arguments or we would have to consider superpositions of these estimates.

\begin{example}\label{ex:sec} We study the example of Figure~\ref{fig:condcomp}.d, that is,
\[
f(z) = \sec(z)^6
\]
which has radius of convergence $R=\pi/2$.
To begin with, we extract the poles at $z=\pm \pi/2$ by the factorization
\[
f(z) = g(z)^6 \cdot v(z),\qquad v(\pm\pi/2) = 1,
\]
with the rational function
\[
g(z) = \frac{\res_{\pi/2} \sec}{z-\pi/2} + \frac{\res_{-\pi/2}\sec}{z+\pi/2} = \frac{4\pi^2}{\pi^2-4z^2}.
\]
One easily checks that $\|v\|_{H^\infty(D_1)}=1$, so that, by (\ref{eq:darbouxest2}) and by a symmetry argument,
\[
1\leq \kappa(n,r) \leq \kappa(f_{-6};n,r/R)\qquad (R=\pi/2).
\]
In view of (\ref{eq:darbouxradius}) we choose the radius
\[
r_n = \frac{\pi}{2}\left(1-\frac{5}{n}\right)
\]
and  obtain (see (\ref{eq:cbeta}) for a definition of $c_\beta$)
\[
1 \leq \kappa(n,r_n) \leq \kappa(f_{-6};n,1-5n^{-1}) \sim c_{-6} = \frac{9e^5}{1250} \doteq 1.0686\qquad (n\to \infty).
\]
We should thus be able to get about full accuracy for large orders of differentiation. In fact, for $n=100$, we have
\[
\kappa(n,r_n) \doteq 1.0767 \leq 1.0769 \doteq \kappa(f_{-6};n,r_n).
\]
Striving for a relative error of $\epsilon = 10^{-15}$ requires, see (\ref{eq:mboundrn2}), a trapezoidal sum with a number of nodes of about
\[
m_\epsilon \approx \frac{n}{5}\log\epsilon^{-1} \approx 700.
\]
In fact, an actual computation with $m=880$ yields a little more than 14 correct digits in hardware arithmetic.
\end{example}

\begin{example}\label{ex:bern} In this example we address the accurate computation of the Bernoulli numbers $B_k$ given by their exponentially generating function (see Figure~\ref{fig:condcomp}.e)
\[
f(z) = \frac{z}{e^z-1} = \sum_{k=0}^\infty \frac{B_k}{k!} z^k,
\]
which has radius of convergence $R=2\pi$.  We extract the poles at $z=\pm 2\pi i$ by the factorization
\[
f(z) = g(z) \cdot v(z),\qquad v(\pm2\pi i) = 1,
\]
with the rational function
\[
g(z) = \frac{\res_{2\pi i} f}{z-2\pi i} + \frac{\res_{-2\pi i}f}{z+2\pi i} = -\frac{8\pi^2}{4\pi^2+z^2}.
\]
One easily checks that $\|v\|_{H^\infty(D_1)}=2\pi/(1-e^{-2\pi}) \doteq 6.2949$, so that, by (\ref{eq:darbouxest2}) and by a symmetry argument,
\[
1\leq \kappa(n,r) \leq 6.3\, \kappa(f_{-1};n,r/R)\qquad (R=2\pi).
\]
Because of (\ref{eq:f1subopt}) we expect just a moderate loss of accuracy using the choice $r_n = 2\pi(1-n^{-1})$. In fact, for $n=100$ we
get $\kappa(100,r_{100}) \doteq 7.2355$, meaning a loss of less than one digit. In view of (\ref{eq:mboundrtilde}) we expect to accomplish an approximation error $\epsilon= 10^{-15}$
using a trapezoidal sum with a number of nodes of about
\[
m_\epsilon \approx n \log\epsilon^{-1} \approx 3500.
\]
In fact, an actual calculation with $m=4096$ gives more than 15 correct digits in hardware arithmetic.
\end{example}

\section{The Quasi-Optimal Radius}\label{sect:upper}

For entire transcendental functions, it turns out that an upper bound of the condition number is actually easier to analyze, namely
\[
\kappa(n,r) = \frac{M_1(r)}{|a_n|r^n} \leq \frac{M(r)}{|a_n|r^n} = \bar\kappa(n,r),
\]
where
\[
M(r) = \max_{0\leq \theta\leq 2\pi} | f(re^{i\theta})|
\]
denotes the maximum modulus function of $f$. In fact, we will see in §§\ref{sect:saddle}--\ref{sect:positive} that the radius that is optimal for this upper bound is in many
cases already close to optimal for the condition number itself.

For the maximum modulus, the analogue of Hardy's theorem is a classical theorem of complex analysis (the three circles theorem); for the standard proof see \cite[Vol.~II, p.~221]{MR0444912} or
\citeasnoun[III.304/305]{MR0170986}:\nocite{Hadamard1896,Blumenthal1907,Faber1907}
\begin{theorem}[Hadamard 1896, Blumenthal 1907, Faber 1907]\label{thm:threecircle}
Let $f$ be given by a Taylor series with radius of convergence $R$. The maximum modulus function $M$ satisfies, for $0<r<R$:
\begin{itemize}
\item[(a)] $M(r)$ is continuously differentiable, except for a set of isolated $r$;\\*[-3mm]
\item[(b)] if $f(z)$ is not a monomial, $\log M(r)$ is a strictly convex function of $\log r$;\\*[-3mm]
\item[(c)] if $f \not\equiv \const$, $M(r)$ is strictly increasing.
\end{itemize}
\end{theorem}

With the same proofs as in §\ref{sect:genres} for the condition number, we deduce from this theorem the following results (restricting ourselves to entire transcendental functions, though).

\begin{theorem}\label{thm:barcond} Let $f$ be an entire transcendental function with Taylor coefficients $a_n$, and let $a_{n_0}$ be its first non-zero coefficient. Then, for $n>n_0$, with $a_n \neq 0$ and $r>0$:
\begin{itemize}
\item[(a)] $\bar\kappa(n,r)$ is continuously differentiable, except for a set of isolated $r$;\\*[-3mm]
\item[(b)] $\log \bar\kappa(n,r)$ is a strictly convex function of $\log r$;\\*[-3mm]
\item[(c)] $\bar\kappa(n,r)\to\infty$, as $r\to 0$ and $r\to\infty$.
\end{itemize}
\end{theorem}

The same reasoning as in §\ref{sect:optrad} shows the existence of the optimal upper bound
\begin{equation}
\bar\kappa_\diamond(n) = \min_{r>0} \bar\kappa(n,r),
\end{equation}
which is now taken for the radius
\begin{equation}
r_\diamond(n) = \argmin_{r>0} \bar\kappa(n,r).
\end{equation}
Note that $r_\diamond(n)$ is unique because of the \emph{strict} convexity stated in Theorem~\ref{thm:barcond}. As for $r_*(n)$ it is convenient to extend the definition of $r_\diamond(n)$ to the case of $a_n = 0$ by setting
\begin{equation}\label{eq:rdiamond}
r_\diamond(n) = \argmin_{r>0} r^{-n} M(r).
\end{equation}
We call $r_\diamond(n)$ the {\em quasi-optimal} radius and define, accordingly, the quasi-optimal condition number by
\begin{equation}
\kappa_\diamond(n) = \kappa(n,r_\diamond(n)) \geq \kappa_*(n).
\end{equation}
Finally, by repeating the proof of Theorem~\ref{thm:rn} we get:

\begin{theorem}\label{thm:rdiamondgrowth} Let $f$ be an entire transcendental function. Then, the sequence $r_\diamond(n)$ satisfies the monotonicity
\[
r_\diamond(n) \leq r_\diamond(n+1)\qquad (n>n_0)
\]
and has the limit $\lim_{n\to\infty} r_\diamond(n) = \infty$.
\end{theorem}

It turns out that the radius $r_\diamond(n)$ is generally much easier to calculate than the optimal radius $r_*(n)$ (see Theorems~\ref{thm:rdiamond} and \ref{thm:zn}). Surprisingly,
in all of these cases the radius $r_\diamond(n)$ is also very close to optimal and the condition number $\kappa_\diamond(n)$ is close to one. Before giving a theoretical frame for these effects,
we illustrate them by two examples.

\begin{example}\label{ex:diamondexp} Since its Taylor coefficients are positive, the exponential function $f(z)=e^z$ has the maximum modulus function $M(r)=e^r$. A short calculation shows that
\[
r_\diamond(n) = n,\qquad \bar\kappa_\diamond(n) = n! \left(\frac{e}{n}\right)^n = \sqrt{2\pi n}\, (1+O(n^{-1}))\qquad (n\to \infty);
\]
where the asymptotics follows from Stirling's formula. However, the quasi-optimal condition number $\kappa_\diamond(n)$ behaves much better than just being of order $O(n^{1/2})$. In fact,
a comparison with (\ref{eq:expdata}) yields, as $n\to\infty$
\[
r_\diamond(n) \sim r_*(n),\qquad \kappa_\diamond(n)=1+\frac{5}{24n} + \frac{97}{1152n^2} + O(n^{-3}),
\]
which is very close to optimal indeed.
\end{example}

\begin{example}\label{ex:diamondbell} We consider the example of Figure~\ref{fig:condcomp}.c, that is, the entire function
\[
f(z) = e^{e^z-1}.
\]
By the positivity of the Taylor coefficients, the maximum modulus function is also given by $M(r) =e^{e^r-1}$. A short calculation yields an
explicit formula for the quasi-optimal radius,
\[
r_\diamond = W(n),
\]
with the Lambert $W$-function as introduced in §\ref{subsec:entire}. To get our hand on the corresponding condition number bound, we
realize that $n! a_n$ is the $n$-th Bell number whose asymptotics is well studied in the literature. \citeasnoun[Prop.~VIII.3]{MR2483235} prove (using the concept of $H$-admissibility that we will
study in~§\ref{sect:hayman})
\[
a_n \sim \frac{e^{e^{r_\diamond}-1}}{r_\diamond^n \sqrt{2\pi\,r_\diamond(r_\diamond+1)e^{r_\diamond}}} = \frac{e^{e^{r_\diamond}-1}}{r_\diamond^n \sqrt{2\pi\,n(r_\diamond+1)}}\qquad (n\to\infty).
\]
Hence, asymptotically, we obtain the condition number bound
\begin{equation}\label{eq:kappabarbell}
\bar\kappa_\diamond(n) = \frac{e^{e^{r_\diamond}-1}}{a_n r_\diamond^n} \sim \sqrt{2\pi\,n(r_\diamond+1)} \sim \sqrt{2\pi\, n\log n}\qquad (n\to \infty),
\end{equation}
where we have used the asymptotic expansion \cite[Eq.~(2.4.3)]{MR671583}
\begin{equation}\label{eq:Wasympt}
W(t) = \log t - \log\log t + O\left(\frac{\log\log t}{\log t}\right)\qquad (t\to\infty).
\end{equation}
Even though (\ref{eq:kappabarbell}) looks like a possible, though moderate, $O(n^{1/2}(\log n)^{1/2})$ growth of the condition number, things turn out to be much better than this. For instance, $n=100$ yields the excellent
quasi-optimal condition number $\kappa_\diamond(100) \doteq 1.013$. In §\ref{sect:hayman} we will explain the surprising effect that $\kappa_\diamond(n)$ is close to one for any order $n$,
see Corollary~\ref{cor:admissible}.

\end{example}

\section{Entire Functions of Perfectly Regular Growth}\label{sect:prg}

\subsection{Order and Type of Entire Functions}

Since $r_\diamond(n) \to \infty$, an explicit asymptotic description of the optimization (\ref{eq:rdiamond}) requires a detailed study of the growth of the maximum modulus function $M(r)$
 as $r \to\infty$. A fruitful characterization is by the order and type of $f$; for the following see \citeasnoun[Thms.~II.9.2--9.5]{MR0444912}.

The order $\rho$ of an entire function $f$ is given by
\begin{equation}\label{eq:orderM}
\rho = \limsup_{r\to\infty} \frac{\log\log M(r)}{\log r} \geq 0.
\end{equation}
Note  that polynomials have order $\rho=0$. If $0 < \rho < \infty$ (which means that $f$ is transcendental), the type $\tau$ of $f$ is given by
\begin{equation}\label{eq:typeM}
\tau = \limsup_{r\to\infty} \frac{\log M(r)}{r^\rho} \geq 0.
\end{equation}
We call $f$ to be of minimal type if $\tau=0$, of normal type if $0<\tau<\infty$, and of maximal type if $\tau=\infty$. Order and type can also be read off from the coefficients~$a_n$ of the
Taylor series; if $f$ is of order~$\rho$, then
\begin{equation}\label{eq:order}
\rho = \limsup_{n\to\infty} \frac{n \log n}{\log(1/|a_n|)};
\end{equation}
if $f$ is of order $\rho$ and type $\tau$, then
\begin{equation}\label{eq:type}
\tau = \frac{1}{e\rho}\limsup_{n\to\infty}\, n |a_n|^{\rho/n}.
\end{equation}

To arrive at an explicit asymptotic formula for $r_\diamond(n)$ (see Theorem~\ref{thm:rdiamond}) we need to consider a somewhat stricter class of entire functions \cite[p.~45]{Valiron}, though:
an entire transcendental function of order $0< \rho <\infty$ is called to be of {\em perfectly regular growth}  if the limit
\begin{equation}\label{eq:prg}
\tau = \lim_{r \to \infty} \frac{\log M(r)}{r^\rho}
\end{equation}
exists and is positive and finite; $f$ is then of normal type $\tau$. The following fundamental theorem is extremely helpful for the purpose of identifying such functions; for a proof see \citeasnoun[p.~108]{Valiron}.\nocite{Wiman1916,Valiron1923}

\begin{theorem}[Wiman 1916, Valiron 1923]\label{thm:wiman1} Let $f$ be an entire transcendental function. If $f$ is the solution of a
holonomic\footnote{Holonomic differential equations are homogeneous linear with polynomial coefficients.}  differential equation of order $q$,
 then $f$ is of perfectly regular growth with a rational order $\rho \geq 1/q$.
\end{theorem}

\begin{example}\label{ex:hyper}
The generalized hypergeometric functions
\begin{equation}\label{eq:pFq}
_pF_q(b_1,\ldots,b_p;c_1,\ldots,c_q;z) = \sum_{n=0}^\infty \frac{(b_1)_n\cdots (b_p)_n}{(c_1)_n\cdots (c_q)_n}\frac{z^n}{n!} \qquad (-b_j, -c_k \not\in \N_0)
\end{equation}
are known to be \cite[§§3.3/5.1]{MR0241700}
\begin{itemize}
\item entire transcendental if and only if $p\leq q$;\\*[-3mm]
\item satisfying a holonomic differential equation of order $\max(p,q+1)$.\\*[-3mm]
\end{itemize}
Thus, by Theorem~\ref{thm:wiman1}, if $p \leq q$, these functions are entire transcendental of perfectly regular growth with a rational order $\rho \geq 1/(q+1)$.
It is an easy exercise in dealing with Stirling's formula\footnote{Stirling's formula implies, for $-c \not\in \N_0$, that $\log |(c)_n| = n \log n -n + O(\log n)$ as $n\to\infty$.}
to calculate from (\ref{eq:order}) and (\ref{eq:type}) the order and type of these functions:
\begin{equation}\label{eq:rhotaupFq}
\rho = \frac{1}{q+1-p},\qquad  \tau = q+1-p \qquad (p\leq q).
\end{equation}
Many transcendental functions can be identified as a generalized hypergeometric function \citeaffixed[§6.2]{MR0241700}{see}; if this relation is of the
form
\[
f(z) = \alpha\, z^\mu\cdot\/ _pF_q(b_1,\ldots,b_p;c_1,\ldots,c_q;\beta\, z^\nu)\qquad (\alpha,\beta \neq 0, \mu \in \N_0, \nu \in \N)
\]
then $f$ is also of perfectly regular growth and we easily obtain, using (\ref{eq:rhotaupFq}), that the order and type of $f$ are given by
\[
\rho = \frac{\nu}{q+1-p},\qquad \tau = (q+1-p) |\beta|^{1/(q+1-p)}.
\]
With the exception of the Airy functions, all the functions in the first section of
Table~\ref{tab:functions} can directly be dealt with this way; it suffices to demonstrate just one such example in detail:
\[
\cos z =\/ _0F_1(;\tfrac12;-\tfrac14 z^2)
\]
has $p=0$, $q=1$, $\nu=2$, and $\beta = -1/4$; therefore $\rho = \tau=1$.
\end{example}

\begin{table}[tbp]
\caption{Various growth characteristics of some entire transcendental functions; all the functions with normal type are of completely regular growth and,
a fortiori, of perfectly regular growth. The column for $r_\diamond(n)$ gives the asymptotics as $n\to \infty$. The angle~$\theta$ is understood to be
restricted to $-\pi\leq\theta\leq \pi$. For $1/\Gamma(z)$ the limit given is meant to be the interval $(\liminf \kappa_\diamond(n),\limsup \kappa_\diamond(n))$. For the $q$-series $(-z;q)_\infty$ we assume that $0<q<1$. }
\vspace*{0mm}
\centerline{%
\setlength{\extrarowheight}{4pt}
{\small
\begin{tabular}{cccccccc}\hline
$f(z)$ & order $\rho$ & type $\tau$ & $r_\diamond(n)$ & $\lim \kappa_\diamond(n)$ & indicator $h(\theta)$ & $\Omega$ & $\omega$ \\*[0.75mm] \hline
$e^z$    & $1$ & $1$ & $n$ & $1$ & $\cos\theta$ & $1$ & $1$\\
$\cos(z)$ & $1$ & $1$ & $n$ & $1$ & $|\sin\theta|$ & $2$ & $1/2$\\
$\sin(z)$ & $1$ & $1$ & $n$ & $1$ & $|\sin\theta|$ & $2$ & $1/2$\\
$J_k(z)$ & $1$ & $1$ & $n$ & $1$ & $|\sin\theta|$ & $2$ & $1/2$\\
$I_k(z)$ & $1$ & $1$ & $n$ & $1$ & $|\cos\theta|$ & $2$ & $1/2$\\
$z^{-k/2} I_k(2\sqrt{z})$ & $1/2$ & $2$ & $n^2$ & $1$ & $2 \cos(\theta/2)$ & $1$ & $1$\\
$\erf(z)$ & $2$ & $1$ & $\sqrt{n/2}$ & $1$ & $(-\cos(2\theta))_+$ & $2$ & $1/2$\\
$e^{-z^2}$ & $2$ & $1$ & $\sqrt{n/2}$ & $1$ & $-\cos(2\theta)$ & $2$ & $1/2$\\
$\Ai(z)$ & $3/2$ & $2/3$ & $n^{2/3}$ & $2/\sqrt{3}$ & $-\frac23\cos(\frac32\theta)$ & $2$ & $1/\sqrt3$\\
$\Bi(z)$ & $3/2$ & $2/3$ & $n^{2/3}$ & $4/3$ & $\frac23|\cos(\frac32\theta)|$ & $3$ & $2/3$\\
$C(z)$ & $2$ & $\pi/2$ & $\sqrt{n/\pi}$ & $1$ & $\frac\pi 2|\sin(2\theta)|$ & $4$ & $1/4$\\
$S(z)$ & $2$ & $\pi/2$ & $\sqrt{n/\pi}$ & $1$ & $\frac\pi 2|\sin(2\theta)|$ & $4$ & $1/4$\\*[0.75mm]\hline\\*[-5.5mm]
$(-z;q)_\infty$ & $0$ & --- & $q^{\frac12 - n}$ & $1$ & --- & --- & ---\\
$1/\Gamma(z)$ & $1$ & $\infty$ & $e^{\Re W(\frac12-n)}$ & $(1,\infty)$ & --- & --- & --- \\
$e^{e^z-1}$ & $\infty$ & --- & $W(n)$ & $1$ & --- & --- & --- \\*[0.5mm]\hline
\end{tabular}}}
\label{tab:functions}
\end{table}

\begin{example} The Airy functions $\Ai(z)$ and $\Bi(z)$ satisfy a holonomic differential equation of second order,
\[
y''(z) - z y(z) = 0.
\]
By the theory of linear analytic differential equations \cite[p.~70]{MR658490}, because the leading coefficient of this equation is $1$,
the Airy functions are entire transcendental.
Thus, Theorem~\ref{thm:wiman1} tells us that the Airy functions are of perfectly regular growth with a rational order $\rho \geq 1/2$. The precise
values of the order and type can be read off from the asymptotic expansions \cite[Eq.~(10.4.59--65)]{MR0167642} of the Airy functions as $z\to\infty$, which imply
\[
M(r) = \frac{c}{\sqrt{\pi}\, r^{1/4}} e^{\tfrac{2}{3} r^{3/2}} (1+O(r^{-3/2}))\qquad (r\to \infty),
\]
with $c=1/2$ for $\Ai(z)$ and $c=1$ for $\Bi(z)$. Hence, by (\ref{eq:orderM}) and (\ref{eq:typeM}), we get
\[
\rho = \tfrac{3}{2},\qquad \tau = \tfrac{2}{3}.
\]
\end{example}

\subsection{The Asymptotics of the Quasi-Optimal Radius}
A short calculation shows that any entire function $f$ with the maximum modulus function
\[
\log M(r) = \tau r^\rho\qquad (\rho,\tau>0)
\]
would have the quasi-optimal radius
\begin{equation}\label{eq:rdiamondsugg}
r_\diamond(n) = \left(\frac{n}{\tau\rho}\right)^{1/\rho}.
\end{equation}
By the definition (\ref{eq:prg}), functions of perfectly regular growth satisfy the asymptotic relation
\begin{equation}\label{eq:rdiamond1}
\log M(r) = \tau r^\rho (1+ o(1)) \qquad(r\to\infty),
\end{equation}
which suggests that (\ref{eq:rdiamondsugg}) might still hold, at least asymptotically as $n\to \infty$. The
following theorem shows that this is indeed the case; however, the proof is quite involved.\footnote{Under the additional assumption of
the non-negativity of the Taylor coefficients of $f$, it is possible to give a much shorter proof of this theorem; see Remark~\ref{rem:rdiamond}.} Concrete examples of the result can be found in Table~\ref{tab:functions}.

\begin{theorem}\label{thm:rdiamond} Let $f$ be an entire transcendental function of perfectly regular growth having order $\rho$ and type $\tau$.
Then, the quasi-optimal radius satisfies
\begin{equation}\label{eq:rdiamondasympt}
r_\diamond(n) \sim \left(\frac{n}{\tau\rho}\right)^{1/\rho}\qquad (n\to \infty).
\end{equation}
\end{theorem}
\begin{proof} The difficulty of the proof is to deal with the simultaneous limits $r\to \infty$ and $n\to\infty$ whose coupling
has yet to be established. To this end we introduce a transformed variable $\eta$ by
\[
r = \left(\frac{n e^\eta}{\tau}\right)^{1/\rho}.
\]
We rewrite (\ref{eq:rdiamond1}) in the form
\[
\log(r^{-n}M(r)) = n e^\eta (1+ o(1)) - \frac{n}{\rho} \eta - \frac{n}{\rho} \log \frac{n}{\tau}  =  n\cdot f_n(\eta) - \frac{n}{\rho} \log \frac{n}{\tau},
\]
defining functions $f_n(\eta)$ that satisfy
\[
f_n(\eta) = e^\eta(1+o(1)) -  \rho^{-1} \eta;
\]
note that the estimate $o(1)$ holds locally uniform in $\eta$ as $n\to \infty$.
By the properties of the maximum modulus function $M$ stated in Theorem~\ref{thm:threecircle}, we know that these functions $f_n$
are strictly convex in $\eta$ and coercive, which means
\[
f_n(\eta) \to \infty \qquad (\eta \to \pm\infty).
\]
The quasi-optimal radius $r_\diamond(n)$, which, by definition, minimizes $r^{-n} M(r)$, is now given in the form
\[
r_\diamond(n) = \left(\frac{n e^{\eta_n}}{\tau}\right)^{1/\rho},
\]
where $\eta_n$ is the unique minimizer of $f_n(\eta)$. The assertion of the theorem is therefore equivalent to $\lim_{n\to \infty} \eta_n = \log \rho^{-1}$, which remains to be proven.

Establishing the limit of $\eta_n$ proceeds by constructing a convex enclosure of $f_n$ for large $n$: for $\epsilon > 0$ small, we define the strictly convex functions
\[
f_{\pm \epsilon}(\eta) = e^\eta(1\pm\epsilon) - \rho^{-1} \eta.
\]
The minimizer of $f_\epsilon$ is explicitly given by
\[
\eta_\epsilon = \argmin f_{\epsilon}(\eta) = \log\frac{1}{\rho(1+\epsilon)}.
\]
Since $f_{-\epsilon}(\eta) < f_\epsilon(\eta)$ for all $\eta$, and because $f_{-\epsilon}$ is convex and coercive, there exist points $\underline{\eta}_{\epsilon}$ and
$\overline{\eta}_{\epsilon}$ with $\underline{\eta}_{\epsilon} < \eta_\epsilon < \overline{\eta}_{\epsilon}$ satisfying
\[
f_{-\epsilon}(\underline{\eta}_{\epsilon}) = f_{-\epsilon}(\overline{\eta}_{\epsilon}) = f_\epsilon(\eta_\epsilon).
\]
It is clear that $\underline{\eta}_{\epsilon},\overline{\eta}_{\epsilon} \to \log \rho^{-1}$ as $\epsilon\to 0$; in particular, $\underline{\eta}_{\epsilon}$ and
$\overline{\eta}_{\epsilon}$ remain bounded. By the asymptotics of $f_n$ as $n\to\infty$, we have, for $n \geq n_\epsilon$,
\begin{align*}
f_n(\eta_\epsilon) \leq f_\epsilon(\eta_\epsilon) = f_{-\epsilon}(\underline{\eta}_{\epsilon}) \leq f_n(\underline{\eta}_{\epsilon}),\\*[1mm]
f_n(\eta_\epsilon) \leq f_\epsilon(\eta_\epsilon) = f_{-\epsilon}(\overline{\eta}_{\epsilon}) \leq f_n(\overline{\eta}_{\epsilon}).
\end{align*}
Thus,
the strictly convex function $f_n$ is neither strictly increasing nor strictly decreasing between the points $\underline{\eta}_{\epsilon}$ and $\overline{\eta}_{\epsilon}$. Hence, its minimizer $\eta_n$ must lie there,
\[
\underline{\eta}_{\epsilon} < \eta_n < \overline{\eta}_{\epsilon}.
\]
Now, taking the limit $n\to \infty$ yields
\[
\underline{\eta}_{\epsilon} \leq \liminf_{n\to\infty} \eta_n \leq \limsup_{n\to\infty} \eta_n \leq \overline{\eta}_{\epsilon}.
\]
Finally, letting $\epsilon\to0$ proves that $\lim_{n\to\infty} = \log \rho^{-1}$ as required.
\end{proof}

\begin{remark}\label{rm:mboundprg}
By means of (\ref{eq:rdiamondasympt}) and (\ref{eq:mbound2}) we can estimate the number of nodes $m_\epsilon$ that a trapezoidal sum would need to achieve
the relative approximation error $\epsilon$ if we choose the quasi-optimal radius $r=r_\diamond(n)$. To this end we recall the Taylor series
\begin{equation}\label{eq:Wasympt2}
W(z) = \sum_{n=1}^\infty (-1)^{n-1} n^{n-1}\, \frac{z^n}{n!} \qquad (|z|< e^{-1})
\end{equation}
of the Lambert $W$-function \citeaffixed[§2.3]{MR671583}{see} and obtain
\begin{equation}\label{eq:mboundprg}
m_\epsilon \approx e n + \rho \log\epsilon^{-1}.
\end{equation}
Note how close this is already to the lower bound $m>n$ given by the sampling condition~(\ref{eq:sampling}).
\end{remark}

\subsection{An Upper Bound of the Quasi-Optimal Condition Number}\label{sect:WimanValiron}

At a first sight the precise asymptotic description (\ref{eq:rdiamondasympt}) of the quasi-optimal radius $r_\diamond(n)$ does
not tell us much about the size of the corresponding condition number $\kappa_\diamond(n)$. In fact, restricting ourselves to subsequences
of $n$ which make the limes superior in~(\ref{eq:type}) a proper limit, we just get
\begin{equation}\label{eq:firstguess}
\log \kappa_\diamond(n) \leq \log \bar\kappa_\diamond(n) = o(n)\qquad (n\to \infty).
\end{equation}
Such a weak estimate could not even exclude a super-polynomial growth of the condition number.
However, we can do much better (see the explicit asymptotic bound (\ref{eq:kappabarbound}) below) by optimizing the upper bound
\[
\bar\kappa_\diamond(n) \leq \frac{M(r)}{|a_n| r^n} \qquad (r>0)
\]
from a dual point of view: by choosing the radius $r$ in a way, such that the modulus of $a_n r^n$ becomes maximal among all normalized Taylor coefficients;
which directly leads us into studying the Wiman--Valiron theory of entire functions.
For an account of the basics of this theory see \citeasnoun[IV.1--76]{MR0170986}; surveys of some more refined recent results can be found in \citeasnoun{MR0385095} and \citeasnoun[Chap.~1.4]{MR1464198}.

The fundamental quantities of the Wiman--Valiron theory are the \emph{maximum term} of an entire function $f$ with Taylor coefficients $a_n$ at a given radius $r$, defined by
\begin{equation}\label{eq:maxterm}
\mu(r) = \max_{n} |a_n|r^n,
\end{equation}
and the corresponding maximal index taking this value, called the {\em central index},
\begin{equation}\label{eq:centindex}
\nu(r) = \max \{n: |a_n|r^n = \mu(r)\}.
\end{equation}
The asymptotic properties of these quantities are described in the following theorem; for a proof see \citeasnoun[IV.68]{MR0170986}.\nocite{45.0641.02}\nocite{47.0301.01}

\begin{theorem}[Wiman 1914]\label{thm:mu1} If the entire function $f$ is of perfectly regular growth with order $\rho$ and type $\tau$, then
\[
\log M(r) \sim \log \mu(r) \sim \tau r^\rho,\qquad \nu(r) \sim \tau \rho r^\rho \qquad (r \to \infty).
\]
\end{theorem}

We restrict ourselves to those entire functions $f$ of perfectly regular growth for which eventually, if $n$ is only large enough, each term $|a_n| r^n$  (with $a_n \neq 0$)
can be made the \emph{unique} maximum term for a properly chosen radius. All the functions of Table~\ref{tab:functions} belong to this class.

\begin{remark}\label{rem:hyper}
If $a_n\neq 0$ for $n$ large enough, then this property is known \citeaffixed[IV.43]{MR0170986}{see} to be equivalent to the fact that $|a_n/a_{n+1}|$ becomes eventually a strictly increasing sequence. This criterion is, for instance, satisfied by the generalized hypergeometric functions (\ref{eq:pFq}) with $p\leq q$:
we find
\[
\left|\frac{a_n}{a_{n+1}}\right| = (n+1) \left|\frac{(n+c_1)\cdots(n+c_q)}{(n+b_1)\cdots(n+b_p)}\right| \sim n^{q+1-p} + O(n^{q-p})\qquad (n\to\infty),
\]
which is therefore strictly increasing if $n$ is only large enough.
\end{remark}

Thus, if $a_n\neq 0$ and $n$ is large enough, then there will be a radius $\bar r_n$ with
\[
n = \nu(\bar r_n),\qquad |a_n| \bar r_n^{n} = \mu(\bar r_n).
\]
Theorem~\ref{thm:mu1} yields the asymptotics (where $n$ runs only through those indices with $a_n \neq 0$)
\[
n = \nu(\bar r_n) \sim \tau \rho \bar r_n^\rho \qquad (n\to\infty),
\]
which implies, in view of Theorem~\ref{thm:rdiamond}, the remarkable asymptotic duality
\begin{equation}\label{eq:dualres}
\bar r_n \sim \left(\frac{n}{\tau\rho}\right)^{1/\rho} \sim r_\diamond(n)\qquad (n \to \infty).
\end{equation}
We thus expect the bound (recall that $r_\diamond(n)$ is defined as the minimizer of $\bar\kappa(n,r)$)
\begin{equation}\label{eq:Mmubound}
\bar\kappa_\diamond(n) = \frac{M(r_\diamond(n))}{|a_n| r_\diamond(n)^n} \leq \frac{M(\bar r_n)}{|a_n| \bar r_n^n} = \frac{M(\bar r_n)}{\mu(\bar r_n)}
\end{equation}
to be quite sharp for large $n$. Now, one of the deep results of the Wiman--Valiron theory is the following explicit bound of the ratio $M(r)/\mu(r)$ in general;
for a proof see \citeasnoun[Thm.~6]{MR0385095}.

\begin{theorem}[Wiman 1914, Valiron 1920]\label{thm:wiman2} Let $f$ be an entire function of finite order $\rho$. Then, for each $\epsilon >0$, there is an exceptional set $E_\epsilon$
of relative logarithmic density smaller than $1/(1+\epsilon)$ such that
\[
M(r) < \rho(1+\epsilon) \mu(r) \sqrt{2\pi \log \mu(r)}\qquad (r \not\in E_\epsilon).
\]
\end{theorem}

\citeasnoun{MR690698} has characterized those entire functions of finite order for which there are no exceptional radii, that is, for which $E_\epsilon = \emptyset$.
However, we did not bother to check her complicated conditions for any concrete functions. Let us simply assume the weaker condition that the sequence $\bar r_n$  does eventually not belong to $E_\epsilon$ for all $\epsilon > 0$. We would then obtain from Theorems~\ref{thm:mu1} and \ref{thm:wiman2}, and from (\ref{eq:dualres}) and (\ref{eq:Mmubound}),
the asymptotic bound (where $n$ runs only through those indices with $a_n \neq 0$)
\begin{equation}\label{eq:kappabarbound}
\kappa_\diamond(n) \leq \bar\kappa_\diamond(n) \leq \rho\sqrt{2\pi\log\mu(\bar r_n)}  \sim \sqrt{2\pi\rho n} \qquad (n\to\infty).
\end{equation}
Note that this bound is consistent with the results obtained in Example~\ref{ex:diamondexp} for $f(z) = e^z$, in which particular case the bound
of $\bar\kappa_\diamond(n)$ is even sharp; quite a success for such a general approach. In preparation of §\ref{sect:growth}, we rephrase (\ref{eq:kappabarbound})
by introducing yet another growth characteristics of $f$, namely the quantity
\begin{equation}\label{eq:omega}
0 \leq \omega = \limsup_{n\to\infty: a_n \neq 0} \frac{\bar\kappa_\diamond(n)}{\sqrt{2\pi\rho n}} \leq 1.
\end{equation}
See Table~\ref{tab:functions},  and also \citeasnoun[IV.50]{MR0170986}, for some examples of $\omega$.

\section{Relation to the Saddle-Point Method}\label{sect:saddle}

The results of the last section have shown that, for a certain class of entire functions of perfectly regular growth, the quasi-optimal condition number $\kappa_\diamond(n)$
grows at worst like
\[
1 \leq \kappa_\diamond(n) \leq \bar\kappa_\diamond(n) = O(n^{1/2}) \qquad (n\to\infty).
\]
However, as we have seen in Examples~\ref{ex:diamondexp} and \ref{ex:diamondbell}, there are cases where the quasi-optimal condition number is asymptotically
optimal, actually satisfying the best of all possible asymptotic bounds, $\kappa_\diamond(n) \sim 1$. We now develop a methodology which can be used to understand and prove this
highly welcome effect for a large class of entire functions;
concrete such examples will follow in the next sections.

\begin{table}[tbp]
\caption{For $f(z)=\Ai(z)$, a comparison of the quasi-optimal radius $r_\diamond(n)$ with its asymptotic value (\ref{eq:rdiamondasympt}) as taken from Table~\ref{tab:functions}. This asymptotic value is already quite accurate for small $n$. The value of $r_\diamond(n) = |z_n| $ was actually computed by numerically solving the saddle point
equation $z_n f'(z_n)/f(z_n)=n$ in the complex plane. Note that $\lim_{n\to\infty} \kappa_\diamond(n) = 2/\sqrt{3}\doteq 1.15470$, see (\ref{eq:airycond}).}
\vspace*{0mm}
\centerline{%
\setlength{\extrarowheight}{3pt}
{\small\begin{tabular}{rrcrc}\hline
$n$ & $r_\diamond(n)$ & $\kappa_\diamond(n)$ & $n^{2/3}$ & $\kappa(n,n^{2/3})$\\ \hline
$1$ &      $1.21575$ & $1.37413$ &  $  1.00000$ & $1.56499$\\
$10$ &     $4.72421$ & $1.19188$ &  $  4.64159$ & $1.21120$\\
$100$ &   $21.58047$ & $1.15832$ &  $ 21.54435$ & $1.16003$\\
$1000$ & $100.01668$ & $1.15506$ &  $100.00000$ & $1.15523$\\*[0.5mm]\hline
\end{tabular}}}
\label{tab:airy}
\end{table}

\subsection{The Saddle-Point Equation}
The key lies in the observation \cite[Lemma~6]{MR0385095} that the maximum modulus function $M$ of an entire function $f$ satisfies, except for a set of isolated radii (see also
Theorem~\ref{thm:threecircle}), the equation
\[
\left. r \frac{d}{dr} \log M(r) \right|_{r=r_0}= \left. z \frac{d}{dz} \log f(z)\right|_{z=z_0},
\]
where $z_0\in\C$ is one of the points for which $|z_0|=r_0$ and $|f(z_0)| = M(r_0)$. We apply this observation
to the quasi-optimal radius $r_n=r_\diamond(n)$ which, by definition, minimizes $r^{-n} M(r)$. If not accidentally one of those isolated exceptions, this radius must fulfill the
differential optimality condition
\[
\left. r \frac{d}{dr} \log M(r) \right|_{r=r_n} = n.
\]
Thus, there is a complex number $z_n$ with
\[
r_n = |z_n|,\qquad M(r_n) = |f(z_n)|,
\]
that satisfies the transcendental equation
\[
n = \left. z \frac{d}{dz} \log f(z)\right|_{z=z_n} = z_n \frac{f'(z_n)}{f(z_n)}.
\]
This equation can be rewritten in the form
\begin{equation}\label{eq:zn_2}
F'(z_n) = 0,\qquad F(z) = z^{-n} f(z).
\end{equation}
For functions $F(z)$ that are analytic in a neighborhood of a point $z_n$ (with $F(z_n) \neq 0$) it is well known \citeaffixed[§5.2]{MR671583}{see}
that $F'(z_n)=0$ holds if and only if the modulus $|F(z)|$ forms a saddle at $z=z_n$. Since, by construction, $|F(z)|$ has a local maximum at the saddle point $z_n$ in the angular direction,
it must thus show a local minimum in the radial direction there; see Figure~\ref{fig:mountain} for an illustration. On the other hand, by the convexity properties of
the maximum modulus function $M$ stated in Theorem~\ref{thm:threecircle}, any saddle point $z_n$ of $|F(z)|$ satisfying $|f(z_n)| = M(|z_n|)$  such that the saddle is oriented this way (local minimum in the radial direction and local
maximum in the angular direction) will give us in turn the \emph{unique} quasi-optimal radius $r_\diamond(n) =r_n= |z_n|$. We have thus proven the following theorem.

\begin{figure}[tbp]
\begin{center}
\begin{minipage}{0.495\textwidth}
\begin{center}
\includegraphics[width=\textwidth]{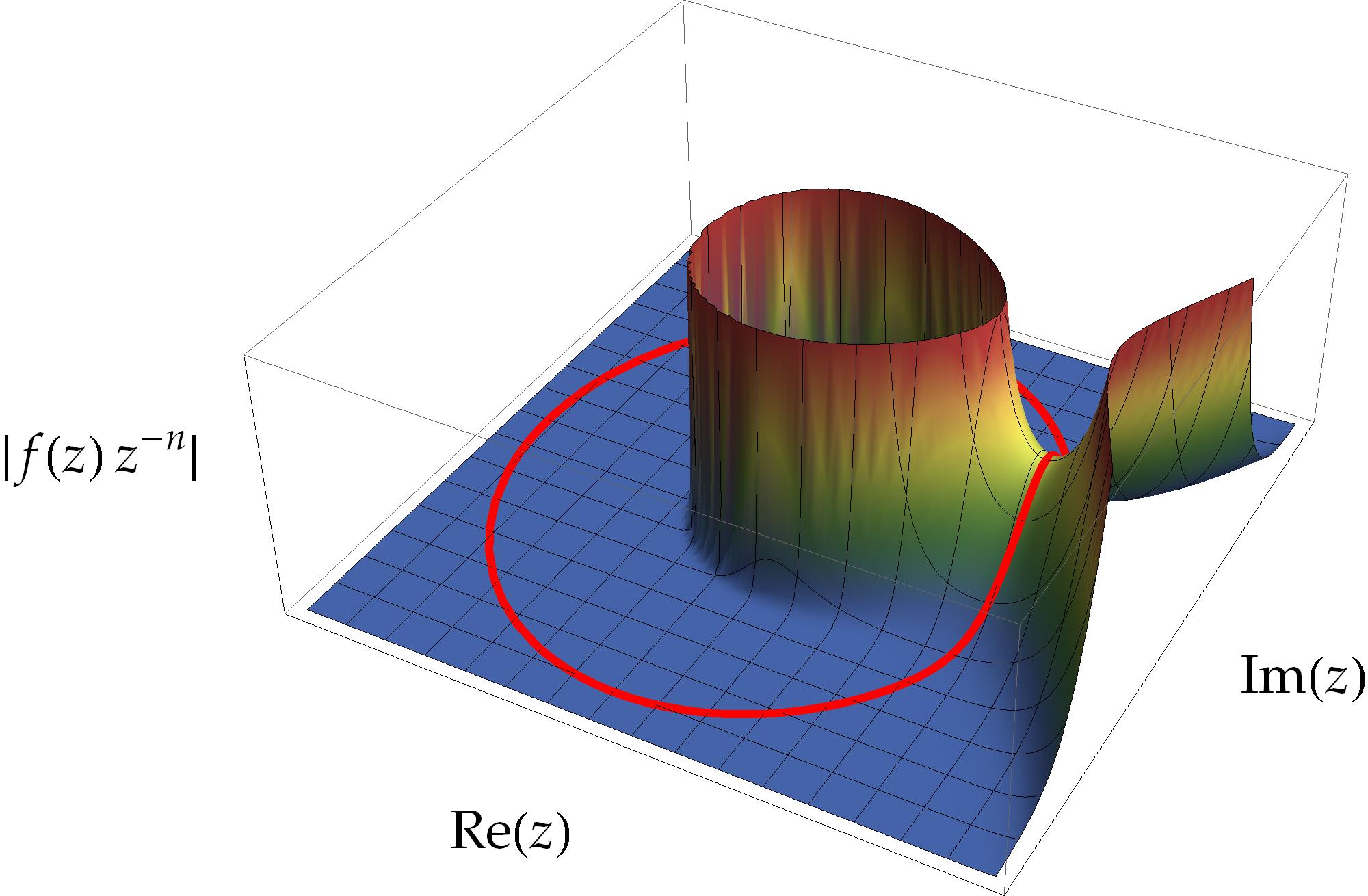}\\*[-1.5mm]
{\tiny a.\;\; $f(z)=e^z$}
\end{center}
\end{minipage}
\hfill
\begin{minipage}{0.495\textwidth}
\begin{center}
\includegraphics[width=\textwidth]{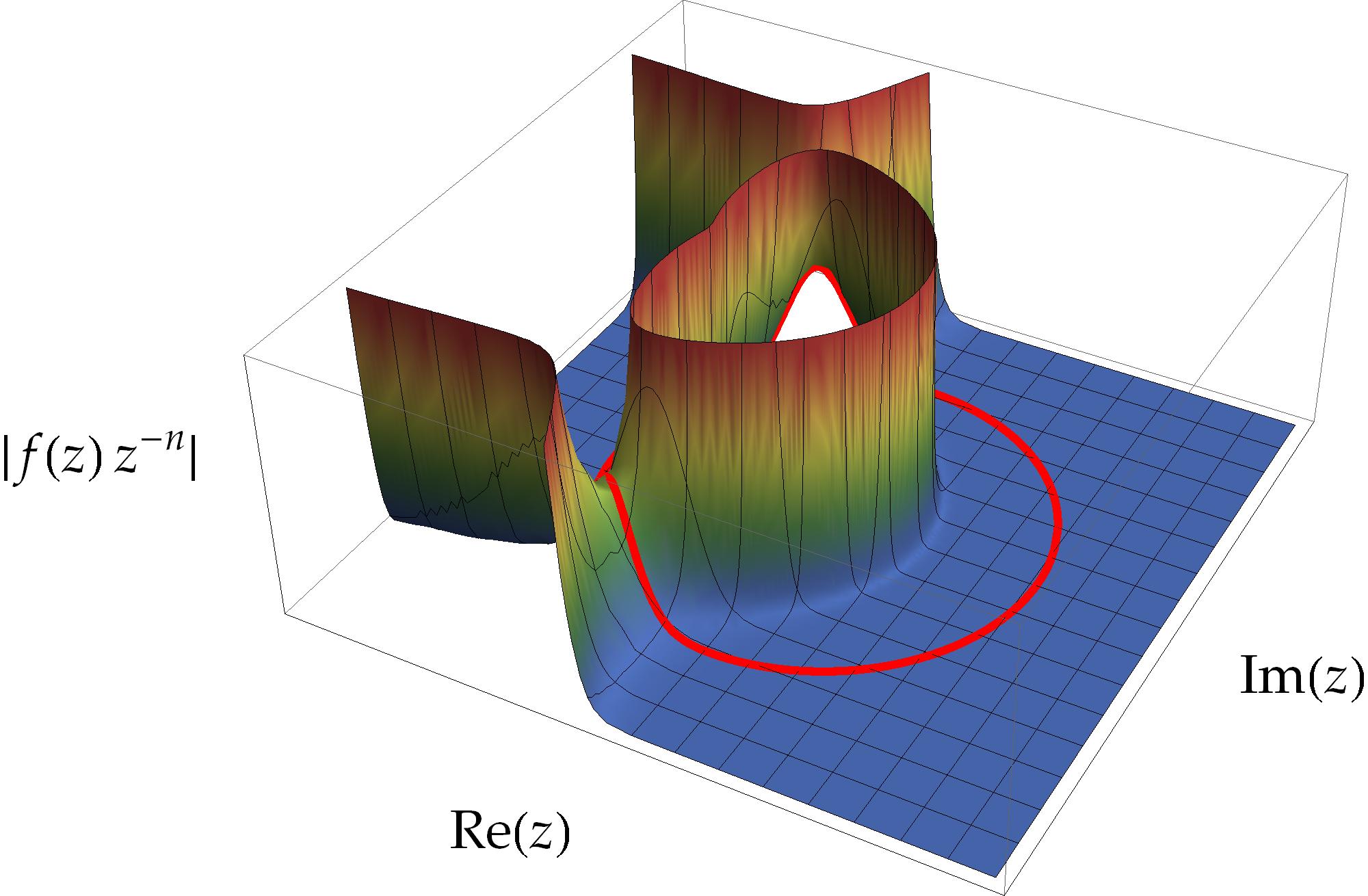}\\*[-1.5mm]
{\tiny b.\;\; $f(z)=\Ai(z)$}
\end{center}
\end{minipage}
\end{center}
\caption{Plots of $|z^{-n} f(z)|$ for $n=31$; left: $f(z)=e^z$; right: $f(z)=\Ai(z)$. The solid curve (red) is the image of the circle $|z|=r_\diamond(n)$;
showing that the maximum modulus along this circle is taken right at some saddle points. Note that the circle leaves these saddle points approximately
in the direction of steepest descent. The left plot explains nicely the qualitative differences between the plots
in Figure~\protect\ref{fig:exp}: where a circle gets close to being a level line of $|z^{-n} f(z)|$, there must be oscillations of the integrand of the Cauchy integral.% For further examples see \protect\citeasnoun[Fig.~VIII.3]{MR2483235}.%
}\label{fig:mountain}
\end{figure}

\begin{theorem}\label{thm:zn} Let $f$ be an entire transcendental function and let $z_n \in \C$ be a solution of the saddle-point equation $F'(z_n)=0$ with $F(z)=z^{-n} f(z)$, that is,
\begin{subequations}\label{eq:rn_zn}
\begin{equation}\label{eq:zn_1}
n = z_n \frac{f'(z_n)}{f(z_n)}.
\end{equation}
If $z_n=r_n e^{i\theta_n}$ satisfies $|f(z_n)| = M(|z_n|)$, $\partial_{\theta\theta} |F(r_n e^{i\theta_n})| < 0$, and
$\partial_{rr} |F(r_n e^{i\theta_n})| > 0$,
then we get the following representation of the quasi-optimal radius:
\begin{equation}
r_\diamond(n) = |z_n|.
\end{equation}
\end{subequations}
On the other hand, if $r_\diamond(n)$ is a point of differentiability of $M(r)$, then there is a solution $z_n$ of the saddle-point equation that satisfies these three conditions.
\end{theorem}

Theorem~\ref{thm:zn} allows us the actual computation of $r_\diamond(n)$; see Tables~\ref{tab:airy}/\ref{tab:Bi}
and §\ref{sect:gamma} for some examples.

\subsection{The Saddle-Point Method}\label{sect:saddlepointmethod}

Taking the quasi-optimal radius $r_\diamond=r_\diamond(n)$ we write the Cauchy integral (\ref{eq:an}) in the form
\[
a_n = \frac{1}{2\pi} \int_0^{2\pi} F(r_\diamond e^{i\theta}) \,d\theta
\]
with $F(z)=z^{-n} f(z)$. If $|f(z)|$, and thus $|F(z)|$, is small for those $z$ on the circle that are not close to the saddle points $z_n$ of Theorem~\ref{thm:zn}, the
integral localizes to the vicinity of these saddle points and we can estimate
\[
a_n \approx \frac{1}{2\pi}\sum_{\substack{\theta_n : z_n=r_\diamond e^{i\theta_n}\\*[1mm]\text{saddle point}}} \int_{\theta\approx\theta_n}  F(r_\diamond e^{i\theta}) \,d\theta.
\]
It is actually possible to estimate each of the integrals
\[
\frac{1}{2\pi}\int_{\theta\approx\theta_n}  F(r_\diamond e^{i\theta}) \,d\theta = \frac{1}{2\pi}\int_{\theta\approx\theta_n}  e^{\log F(r_\diamond e^{i\theta})} \,d\theta
\]
by the Laplace method \citeaffixed[§5.7]{MR671583}{see}. To this end we expand
the function $\log f(r e^{i\theta})$ with respect to the angular variable $\theta$; for $\theta\to\theta_*$ we calculate that
\begin{subequations}\label{eq:taylor_angular_gen}
\begin{equation}\label{eq:taylor_angular}
\log f (r e^{i\theta}) = \log f(z_*) + i a(z_*) (\theta-\theta_*) - \tfrac{1}{2} b(z_*) (\theta-\theta_*)^2 + O(\theta-\theta_*)^3
\end{equation}
with $z_* = r e^{i\theta_*}$ and the coefficients
\begin{equation}\label{eq:taylor_angular_ab}
a(z) = z\,\frac{f'(z)}{f(z)},\qquad b(z) = z a'(z).
\end{equation}
\end{subequations}
By specifying as the expansion point $z_*$ a saddle point $z_n = r_\diamond e^{i\theta_n}$ as in Theorem~\ref{thm:zn} we thus have $a(z_n)=n$ and therefore
\[
\log F(r_\diamond e^{i\theta}) = \log F(z_n) - \tfrac{1}{2} b(z_n) (\theta-\theta_n)^2  + O(\theta-\theta_n)^3;
\]
hence, by taking real parts,
\[
\log |F(r_\diamond e^{i\theta})| =  \log |F(z_n)| - \tfrac{1}{2} \Re\, b(z_n) (\theta-\theta_n)^2  + O(\theta-\theta_n)^3.
\]
In particular, if $|F(z)|$ takes when moving along the circle a \emph{strict} local maximum at the saddle point $z_n$, we infer that necessarily
\begin{equation}\label{eq:b_zn}
\Re\, b(z_n) > 0.
\end{equation}
Thus, the Laplace method is applicable and gives, by ``trading tails'',
\begin{multline}\label{eq:laplace}
\frac{1}{2\pi}\int_{\theta\approx\theta_n}  e^{\log F(r_\diamond e^{i\theta})} \,d\theta \approx \frac{1}{2\pi}\int_{\theta\approx\theta_n} e^{\log F(z_n) - \tfrac12 b(z_n)(\theta-\theta_n)^2} \,d\theta\\
 \approx \frac{1}{2\pi}\int_{-\infty}^\infty e^{\log F(z_n) - \tfrac12 b(z_n)\theta^2} \,d\theta = \frac{F(z_n)}{\sqrt{2\pi b(z_n)}}.
\end{multline}
Summarizing our results so far, we get the following estimate of the Taylor coefficient $a_n$:
\begin{equation}\label{eq:ansaddle}
a_n \approx \frac{1}{\sqrt{2\pi}}\sum_{\substack{\theta : z=r_\diamond e^{i\theta}\\*[1mm]\text{saddle point}}}  \frac{F(z)}{\sqrt{b(z)}}.
\end{equation}
Correspondingly, we estimate the mean modulus by
\begin{multline}\label{eq:Mnsaddle}
\frac{1}{2\pi} \int_0^{2\pi} |F(r_\diamond e^{i\theta})|\,d\theta  \approx\frac{1}{2\pi} \sum_{\substack{\theta_n : z_n=r_\diamond e^{i\theta_n}\\*[1mm]\text{saddle point}}} \int_{\theta\approx \theta_n} |F(r_\diamond e^{i\theta})|\,d\theta\\*[2mm]
= \frac{1}{2\pi} \sum_{\substack{\theta_n : z_n=r_\diamond e^{i\theta_n}\\*[1mm]\text{saddle point}}} \int_{\theta\approx \theta_n}
e^{\Re\log F(r_\diamond e^{i\theta})} \,d\theta \approx
\frac{1}{\sqrt{2\pi}}\sum_{\substack{\theta : z=r_\diamond e^{i\theta}\\*[1mm]\text{saddle point}}}  \frac{|F(z)|}{\sqrt{\Re\, b(z)}}
\end{multline}
and, therefore, the quasi-optimal condition number by
\begin{equation}\label{eq:approxkappa}
\kappa_\diamond(n) = \kappa(n,r_\diamond) = \frac{\displaystyle\int_0^{2\pi} |F(r_\diamond e^{i\theta})|\,d\theta}{\left|\displaystyle\int_0^{2\pi} F(r_\diamond e^{i\theta})\,d\theta\right|}
\approx
\frac{\displaystyle\sum_{\substack{\theta : z=r_\diamond e^{i\theta}\\*[1mm]\text{saddle point}}}  \frac{|F(z)|}{\sqrt{\Re\, b(z)}}}
{\Big|\displaystyle\sum_{\substack{\theta : z=r_\diamond e^{i\theta}\\*[1mm]\text{saddle point}}}  \frac{F(z)}{\sqrt{b(z)}}\Big|}.
\end{equation}
As we will see in the following sections, for some interesting classes of entire functions our reasoning can eventually be sharpened by replacing the somewhat vague ``$\approx$''-signs
 of approximation with rigorous asymptotic equality
as $n \to \infty$. Moreover, the estimate (\ref{eq:approxkappa}) is actually quite precise even for small $n$ as is typical for such asymptotic estimates of integrals; see §\ref{sect:gamma} for an example.

\subsection{Steepest Descent}\label{sect:steepest}

In general, there is not much to further conclude about the approximate values of $\kappa_\diamond(n)$ from the estimate (\ref{eq:approxkappa}).
Thus, to get to a result like $\kappa_\diamond(n)\approx 1$ we need some additional structure:
a look at the examples of Figure~\ref{fig:mountain} tells us that there the circle of radius $r_\diamond$ passes through the saddle points of $|F(z)|$
approximately in the direction of steepest descent.
In the next sections we will explain why this is the case for some larger classes of entire functions.

From general facts about the method of steepest descent in asymptotic analysis\footnote{For a detailed exposition see \citeasnoun[Chap.~5]{MR671583},
\citeasnoun[Chap.~4]{MR2238098}, and \citeasnoun[Chap.~VIII]{MR2483235}. \citeasnoun[§5.5]{MR2383071} explain how steepest descent contours are used as an analytic tool
for obtaining numerically stable integral representations of certain special functions; a topic that is certainly closely related to the theme of this paper.}
we learn \citeaffixed[p.~84]{MR671583}{see} that the circular contour through the saddle point $z_n$
is approximately of steepest descent if and only if $b(z_n)$ is approximately real, that is, if and only if
\begin{equation}\label{eq:imb0}
\Im\, b(z_n) \approx 0.
\end{equation}
Note that this implies that the integrand in (\ref{eq:laplace}) has approximately constant phase.
In fact, geometrically it is straightforward to see that the circle is the contour of steepest descent if and only if the off-diagonal
elements of the Hessian of $G(r,\theta) = \log |F(r e^{i\theta})|$ vanish; at a saddle point $z_n = r_\diamond e^{i\theta_n}$ as in Theorem~\ref{thm:zn} we actually obtain
\begin{equation}\label{eq:hessian}
\hess\, G(r_\diamond,\theta_n) =
\begin{pmatrix}
\Re\, b(z_n)r_\diamond^{-2} & -\Im\,b(z_n) r_\diamond^{-1} \\*[2mm]
-\Im\,b(z_n) r_\diamond^{-1} & -\Re\, b(z_n)
\end{pmatrix}.
\end{equation}
Now, assume additionally that the circle of radius $r_\diamond = r_\diamond(n)$ passes through \emph{just one} saddle-point $z_n$ (this amounts for the case $\Omega=1$ in §\ref{sect:crgcondbounds}).
Then, we infer from the condition number estimate (\ref{eq:approxkappa}) and the steepest descent condition (\ref{eq:imb0}) that
\[
\kappa_\diamond(n) \approx
\frac{\displaystyle\frac{|F(z_n)|}{\sqrt{\Re\, b(z_n)}}}
{\Big|\displaystyle\frac{F(z_n)}{\sqrt{b(z_n)}}\Big|} =\sqrt[4]{1+ \left(\frac{\Im\,b(z_n)}{\Re\, b(z_n)}\right)^2}\approx  1.
\]
This line of reasoning thus explains why the best of all possible results, $\kappa_\diamond(n)\approx 1$, actually may come into place even though the
radius $r_\diamond = r_\diamond(n)$ itself was first introduced by optimizing just the upper bound $\bar\kappa(n,r)$ of the condition number.

%\begin{remark}
%Table~\ref{tab:airy} tells us that $\lim_{n\to\infty} \kappa_\diamond(n) = 2/\sqrt{3} > 1$ for $f(z)=\Ai(z)$, though. Therefore, for this particular function,
%the circular contour of radius $r_\diamond(n)$ must have crossed
%more than just one saddle point; a look at Figure~\ref{fig:mountain}.b shows
%that this is indeed the case.
%\end{remark}

\section{Entire Functions of Completely Regular Growth}\label{sect:growth}

\subsection{The Indicator Function}

The reasoning of §\ref{sect:steepest} relies on the remarkable fact (observed in Figure~\ref{fig:mountain}) that for certain functions the circle
passing through the relevant saddle points is approximately tangential to the contour of steepest descent. This could be understood if $F(z) = z^{-n} f(z)$ happens to grow predominantly in a radial direction. A first hint that this is exactly the right
picture is the existence of the Phragmén--Lindelöf indicator function
\begin{equation}\label{eq:indicator}
h(\theta) = \limsup_{r\to\infty} r^{-\rho} \log |f(r e^{i\theta})|
\end{equation}
for entire functions of finite order $\rho$ and normal type $\tau$. We recall some of its properties; see \citeasnoun[§II.45]{MR0444912} or \citeasnoun[§I.15/16]{MR589888} for proofs:
\begin{itemize}
\item $h(\theta)$ is $2\pi$-periodic;\\*[-3mm]
\item $h(\theta)$ is continuous and has a derivative except possibly on a countable set;\\*[-3mm]
\item if $0< \rho \leq 1/2$, then $0 \leq h(\theta) \leq \tau$; if $\rho>1/2$, then $-\tau \leq h(\theta) \leq \tau$;\\*[-3mm]
\item $\tau = \max_\theta h(\theta)$.
\end{itemize}
As it was convenient in §\ref{sect:prg} to consider the functions of perfectly regular growth, for which the limes superior in the definition (\ref{eq:typeM}) of the type $\tau$
becomes the proper limit (\ref{eq:prg}), we do the same with the limes superior in the definition of the indicator function here:

An entire function of finite order $\rho$ and normal type $\tau$ is called to be of {\em completely regular growth} \cite[Chap.~III]{MR589888} if
\begin{equation}\label{eq:crg}
h(\theta) = \lim_{r\to\infty: r\not\in E} \frac{\log|f(r e^{i\theta})|}{r^\rho},
\end{equation}
uniformly in $\theta$. Here, the exceptional set $E$ is required to have relative linear density zero; it will obviously be related to the zeros of $f$. In fact, if there
are no zeros of $f$ in an open sector containing the ray of direction $\theta$, then (\ref{eq:crg}) holds in a closed subsector without the need of an exceptional set.
An important result of \citeasnoun[p.~142]{MR589888} states that if (\ref{eq:crg}) holds just pointwise for $\theta$ in a set that is dense in $[-\pi,\pi]$, then $f$ is already
of completely regular growth. This criterion can be used to check that all of the functions in the first section of Table~\ref{tab:functions} are of completely regular growth
with the indicator functions given there: one just has to look at the known asymptotic expansions  of $f(z)$ as $z \to\infty$ within certain sectors of the complex plane,
as they are found, e.g., in \citeasnoun{MR0167642}. It is also known that the statement of Theorem~\ref{thm:wiman1} extends to completely regular functions, see \citeasnoun[p.~747]{MR1401944}.

As developed mainly by Pfluger and Levin in the 1930s, there is a deep relation between the angular density of zeros of a function $f$  of completely regular growth and the properties
of its indicator function $h(\theta)$. The following characterization of a density of zero will be of importance to us \cite[p.~155]{MR589888}:
\begin{multline}\label{eq:zerodens}
\lim_{r\to\infty} \frac{\text{\# zeros $|z|\leq r$ of $f$ in an open sector containing the ray at $\theta_0$}}{r^\rho} = 0 \\*[2mm]
\Leftrightarrow\quad \text{$h(\theta)$ is $\rho$-trigonometric in the vicinity of $\theta_0$};
\end{multline}
where a function of $\theta$ is called $\rho$-trigonometric if it is of the form $\alpha \sin(\rho \theta + \beta)$ for some real $\alpha$ and $\beta$.

\subsection{Circles Are Contours of Asymptotic Steepest Descent}

We now look at a direction $\theta_*$ in which there is the predominantly growth of $f$, that is, $h(\theta_*)=\tau$. If there are at most finitely many zeros of $f$ in an open sector
containing the ray at $\theta_*$ (which is the case for all of the functions in the first section of Table~\ref{tab:functions}), then $f$ will also be of perfectly
regular growth and the indicator will be, by (\ref{eq:zerodens}), $\rho$-trigonometric in the vicinity of $\theta_*$. In particular, we get
\begin{equation}\label{eq:h2}
h(\theta_*)=\tau, \quad h'(\theta_*)=0, \quad h''(\theta_*) = -\tau \rho^2.
\end{equation}
By the reasoning of §\ref{sect:saddle} there will be a sequence $z_n = r_\diamond e^{i\theta_n}$ (writing $r_\diamond=r_\diamond(n)$ for brevity) satisfying the saddle-point equation (\ref{eq:zn_1}) with $\theta_n \to \theta_*$ as
$n \to \infty$. To show that the circle passing through $z_n$ is asymptotically a contour of steepest descent there, we look at the Hessian of $\log |F(z_n)|$. From (\ref{eq:crg}) we first
get
\begin{align}
\log |F(z_n)| &=  \log |f(r_\diamond e^{i\theta_n})| - n \log r_\diamond\notag\\*[1mm]
& \sim r_\diamond^\rho h(\theta_n) - n \log r_\diamond\qquad (n\to\infty).\label{eq:Fasympt}
\end{align}
Next, by Theorem~\ref{thm:rdiamond} and (\ref{eq:h2}), the Hessian of the right hand side, $G(r,\theta)=r^\rho h(\theta) - n \log r$, becomes asymptotically diagonal:
\[
\hess\, G(r_\diamond,\theta_n) \sim n\rho
\begin{pmatrix}
r_\diamond^{-2} & 0 \\*[1mm]
0 & -1
\end{pmatrix}
\qquad (n\to\infty);
\]
note that this form of the Hessian is actually consistent with (\ref{eq:hessian}) and (\ref{eq:b_zn}).
Since the off-diagonal terms are zero, the $\theta$-direction is, asymptotically, the direction of steepest descent.
%This finishes, what \citeasnoun[p.~77]{MR671583} calls the first stage of applying the saddle
%point method, which results in choosing the contour of integration: ``it is usually the most difficult one.''

\subsection{Condition Number Bounds}\label{sect:crgcondbounds}

We follow the strategy of §\ref{sect:saddlepointmethod} and apply the Laplace method
to the contour integral with radius $r_\diamond=r_\diamond(n)$. However, instead of using the Taylor expansion (\ref{eq:taylor_angular_gen})
to simplify $\log F(r e^{i\theta})$ we now proceed by first recalling from §\ref{sect:steepest} that contours of steepest descent
yield integrands of an asymptotically constant phase and by next using
the indicator function (\ref{eq:crg}) to simplify $\log |F(r e^{i\theta})|$, asymptotically as $r \to \infty$.
Note that the Laplace method rigorously applies if there is a proper decay of $\log |F(r e^{i\theta})|$, as $r\to\infty$,  for directions $\theta$ far off those $\theta_*$
that belong to the saddle points. Assuming this to be the
case for the given $f$ (it can be checked to be true for all the functions in the first section of Table~\ref{tab:functions}), we get for the
Cauchy integral~(\ref{eq:an}), because of (\ref{eq:Fasympt}), (\ref{eq:h2}) and (\ref{eq:rdiamondasympt}), as $n\to\infty$:
\begin{align*}
a_n &= \frac{1}{2\pi r_\diamond^n} \int_{-\pi}^\pi e^{-i n \theta} f(r_\diamond e^{i\theta}) d\theta = \frac{1}{2\pi} \int_{-\pi}^\pi e^{\log F(r_\diamond e^{i\theta})} \,d\theta\\*[2mm]
&\sim \frac{1}{2\pi} \sum_{\theta:h(\theta)=\tau} e^{i\,\Im \log F(r_\diamond e^{i\theta})} \int_{-\infty}^\infty e^{\Re \log F(r_\diamond e^{i\theta}) + \frac12 r_\diamond^\rho (t-\theta)^2 h''(\theta)}\,dt\\*[2mm]
& =\frac{1}{2\pi} \sum_{\theta:h(\theta)=\tau} \sqrt{\frac{2\pi}{-r_\diamond^\rho\, h''(\theta)}} \,F(r_\diamond e^{i\theta})\\*[2mm]
&\sim \frac{1}{\sqrt{2\pi \rho n} \cdot r_\diamond^n} \sum_{\theta:h(\theta)=\tau} e^{-in\theta} f(r_\diamond e^{i\theta}).
\end{align*}
Likewise, we get, as $n\to \infty$,
\begin{align*}
\frac{M_1(r_\diamond)}{r_\diamond^n} &= \frac{1}{2\pi r_\diamond^n} \int_{-\pi}^\pi |f(r_\diamond e^{i\theta})| d\theta = \frac{1}{2\pi} \int_{-\pi}^\pi e^{\Re \log F(r_\diamond e^{i\theta})} \,d\theta\\*[2mm]
&\sim \frac{1}{\sqrt{2\pi \rho n} \cdot r_\diamond^n} \sum_{\theta:h(\theta)=\tau} |f(r_\diamond e^{i\theta})|
\end{align*}
and certainly
\begin{align*}
\frac{M(r_\diamond)}{r_\diamond^n} &\sim r_\diamond^{-n} \max_{\theta:h(\theta)=\tau} |f(r_\diamond e^{i\theta})|.
\end{align*}
To summarize, we have proven the following theorem.

\begin{theorem}\label{thm:crg} Let $f$ be an entire function of completely regular growth with order $\rho$, type $\tau$, and Phragmén--Lindelöf indicator function $h(\theta)$. If $f$ has
at most finitely many zeros in some sectorial neighborhoods of those rays of direction $\theta$ for which $h(\theta)=\tau$ and if~$|f|$ decays properly, for large radius $r$, in the angular direction off these rays, then
we have
 \begin{equation}\label{eq:barkappaestimatecrg}
\frac{\bar\kappa_\diamond(n)}{ \sqrt{2\pi\rho n}} \sim
\frac{\displaystyle\max_{\theta:h(\theta)=\tau} |f(r_\diamond e^{i\theta})|}{\Big|\displaystyle\sum_{\theta:h(\theta)=\tau} e^{-in\theta} f(r_\diamond e^{i\theta})\Big|}\qquad (n\to \infty: a_n \neq 0)
\end{equation}
and
 \begin{equation}\label{eq:kappaestimatecrg}
\kappa_\diamond(n) \sim \frac{\displaystyle\sum_{\theta:h(\theta)=\tau} |f(r_\diamond e^{i\theta})|}{\Big|\displaystyle\sum_{\theta:h(\theta)=\tau} e^{-in\theta} f(r_\diamond e^{i\theta})\Big|}\qquad (n\to \infty: a_n \neq 0).
\end{equation}
That is, the quasi-optimal condition number $\kappa_\diamond(n)$ of the Cauchy integral is asymptotically equal to the condition number of the finite sum $\sum_{\theta:h(\theta)=\tau} e^{-in\theta} f(r_\diamond e^{i\theta})$.
\end{theorem}

Let us introduce the number of global maxima of the indicator function,
\begin{equation}\label{eq:Omega}
\Omega= \# \{\theta : -\pi < \theta \leq \pi, h(\theta)=\tau\}.
\end{equation}
Now, by Theorem~\ref{thm:crg}, $\Omega=1$ clearly implies that $\lim_{n\to\infty}\kappa_\diamond(n) = 1$ and that the quantity defined in (\ref{eq:omega}) satisfies $\omega = 1$;
this observation is precisely matched by two examples in Table~\ref{tab:functions}. On the other hand, if $\Omega>1$ then it seems, at a first sight, that the condition number of the
finite sum $\sum_{\theta:h(\theta)=\tau} e^{-in\theta} f(r_\diamond e^{i\theta})$ could suffer from severe cancelation.
However, as the next theorem shows, there will  be  generally no such cancelation for the class of functions considered in this section.
(But see §\ref{sect:gamma} for an example of severe resonant cancelations in a different setting.)
\begin{theorem}\label{thm:condcrg} Let $f$ be an entire function of completely regular growth which satisfies the assumptions of Theorem~\ref{thm:crg} as well as those that led to  (\ref{eq:omega}), that is, to $\omega\leq 1$.
Then, this bound can be supplemented by
\begin{equation}\label{eq:barkappaomega}
0 < \Omega^{-1} \leq \liminf_{n\to\infty: a_n \neq 0} \frac{\bar\kappa_\diamond(n)}{\sqrt{2\pi\rho n}} \leq \limsup_{n\to\infty: a_n \neq 0} \frac{\bar\kappa_\diamond(n)}{\sqrt{2\pi\rho n}}
= \omega \leq 1,
\end{equation}
and the quasi-optimal condition number $\kappa_\diamond(n)$ is asymptotically bounded as follows:
 \begin{equation}\label{eq:kappaOmega}
1 \leq \liminf_{n\to\infty: a_n \neq 0} \kappa_\diamond(n)  \leq \limsup_{n\to\infty: a_n \neq 0} \kappa_\diamond(n) \leq  \Omega \cdot \omega.
\end{equation}
In particular, we have
\[
\omega=\Omega^{-1} \quad \Rightarrow \quad \lim_{n\to \infty: a_n\neq 0} \kappa_\diamond(n) = \lim_{n\to \infty: a_n\neq 0}\frac{\bar\kappa_\diamond(n)}{\sqrt{2\pi\rho n}} = 1.
\]
\end{theorem}
\begin{proof}
The obvious estimate
\begin{equation}\label{eq:maxsumbound}
\Big|\displaystyle\sum_{\theta:h(\theta)=\tau} e^{-in\theta} f(r_\diamond e^{i\theta})\Big| \leq \sum_{\theta:h(\theta)=\tau} |f(r_\diamond e^{i\theta})| \leq \Omega\cdot \displaystyle\max_{\theta:h(\theta)=\tau} |f(r_\diamond e^{i\theta})|
\end{equation}
yields, by Theorem~\ref{thm:crg} and (\ref{eq:omega}), the asymptotic bounds asserted in (\ref{eq:barkappaomega}). Moreover,  (\ref{eq:omega}) and (\ref{eq:barkappaestimatecrg})
 imply
\[
\limsup_{n\to\infty: a_n \neq 0} \frac{\displaystyle\max_{\theta:h(\theta)=\tau} |f(r_\diamond e^{i\theta})|}{\Big|\displaystyle\sum_{\theta:h(\theta)=\tau} e^{-in\theta} f(r_\diamond e^{i\theta})\Big|} = \omega \leq 1.
\]
Hence, by using (\ref{eq:maxsumbound}) once more to estimate the numerator in (\ref{eq:kappaestimatecrg}), we get
\[
\limsup_{n\to\infty: a_n \neq 0} \kappa_\diamond(n) \leq \Omega \limsup_{n\to\infty: a_n \neq 0} \frac{\displaystyle\max_{\theta:h(\theta)=\tau} |f(r_\diamond e^{i\theta})|}{\Big|\displaystyle\sum_{\theta:h(\theta)=\tau} e^{-in\theta} f(r_\diamond e^{i\theta})\Big|} \leq \Omega \cdot \omega,
\]
which proves the asserted asymptotic bound (\ref{eq:kappaOmega}).
\end{proof}

\begin{example}
If, by the symmetries of the function $f$ in the complex plane, there is just one \emph{single} phase $\phi_n \in \R$ that allows
us the representation
\begin{equation}\label{eq:symmetry}
e^{i \phi_n} \cdot e^{-in\theta} f(r_\diamond e^{i\theta}) = \max_{\theta_*:h(\theta_*)=\tau} |f(r_\diamond e^{i\theta_*})|
\end{equation}
 for \emph{all} $\theta$ with $h(\theta)=\tau$, then we get by Theorem~\ref{thm:crg} that already the best of all possible bounds holds, namely
\begin{equation}\label{eq:best2}
\lim_{n\to \infty: a_n\neq 0} \kappa_\diamond(n) = 1,\qquad \lim_{n\to \infty: a_n\neq 0} \frac{\bar\kappa_\diamond(n)}{\sqrt{2\pi\rho n}} = \Omega^{-1}.
\end{equation}
We than have, by definition, $\omega = \Omega^{-1}$.
Note that the symmetry relation (\ref{eq:symmetry}) applies to all of the functions of the first section of Table~\ref{tab:functions}, except for the Airy functions $\Ai(z)$ and $\Bi(z)$
which will be dealt with in the next two examples.
\end{example}

\begin{example}\label{ex:airy} The point of departure for discussing the Airy function $\Ai(z)$ is the asymptotic expansion \cite[Eq.~(10.4.59)]{MR0167642}
\begin{equation}\label{eq:airyexp}
\Ai(z) \sim \frac{1}{2\pi} z^{-1/4} e^{-\frac23 z^{3/2}} \sum_{k=0}^\infty \frac{(-1)^k \Gamma(3k+\tfrac12)}{9^k \Gamma(2k+1)} z^{-3k/2}\qquad (z\to\infty: |\arg z|<\pi).
\end{equation}
This implies, by Levin's criterion given above, that $\Ai$ is of completely regular growth. Moreover, we get
\[
|\Ai(r e^{i\theta})| = \frac{1}{2\pi} r^{-1/4} e^{-\tfrac23 r^{3/2} \cos(\tfrac32\theta)} (1+ O(r^{-3/2}))\qquad (r\to\infty: |\theta|<\pi),
\]
from which we can directly read off the order $\rho=3/2$, the type $\tau=2/3$, and the Phragmén--Lindelöf indicator function
\[
h(\theta) = -\tfrac23 \cos(\tfrac32\theta) \qquad (|\theta|<\pi).
\]
Note that this indicator $h(\theta)$, continued as a $2\pi$-periodic function, is $\rho$-trigonometric exactly for $\theta \neq k \pi$ ($k \in \Z$). Thus, by Levin's general theory, there is a positive density of zeros in an arbitrary small sectorial neighborhood of the ray at $\theta=-\pi$; indeed, $\Ai(z)$ has countably many zeros along the negative real axis and
no zeros elsewhere.
We have $h(\theta)=\tau$ for $\theta=\pm\tfrac23\pi$; hence $\Omega=2$.
The expansion (\ref{eq:airyexp}) implies for these maximizing angles that
\[
\Ai(re^{\pm \frac23\pi i}) = \frac{e^{\mp \frac\pi6 i}}{2\sqrt{\pi}}r^{-1/4} e^{\tfrac23 r^{3/2}} (1+ O(r^{-3/2}))\qquad(r\to\infty),
\]
that is
\[
\arg \Ai(r e^{\pm \tfrac23 \pi i}) = \mp \frac{\pi}{6} + O(r^{-3/2})\qquad (r\to\infty).
\]
Hence we obtain, because of $h(-\tfrac23\pi)=h(\tfrac23\pi)$: as $n\to\infty$,
\begin{align*}
\Big| \sum_{\theta:h(\theta)=\tau} e^{-in\theta} \Ai(r_\diamond e^{i\theta})\Big| &\sim \left| e^{\tfrac23\pi n i} e^{\tfrac\pi6 i}+e^{-\tfrac23\pi n i} e^{-\tfrac\pi6 i}\right| \cdot \left|\Ai\left(r_\diamond e^{\tfrac23 \pi i}\right)\right|  \\*[2mm]
&= 2 \left|\cos\left(\tfrac\pi6+\tfrac23\pi n\right)\right| \cdot \left|\Ai\left(r_\diamond e^{\tfrac23 \pi i}\right)\right|,
\end{align*}
and
\[
\sum_{\theta:h(\theta)=\tau} |\Ai(r_\diamond e^{i\theta})| \sim 2 \left|\Ai\left(r_\diamond e^{\tfrac23 \pi i}\right)\right|, \qquad
 \max_{\theta:h(\theta)=\tau} |\Ai(r_\diamond e^{i\theta})| \sim  \left|\Ai\left(r_\diamond e^{\tfrac23 \pi i}\right)\right|.
\]
Now,
\[
\left|\cos\left(\tfrac\pi6+\tfrac23\pi n\right)\right| =
\begin{cases}
\sqrt{3}/2 & \qquad n \not\equiv 2 \pmod{3},\\
0 & \qquad n \equiv 2 \pmod{3},
\end{cases}
\]
in accordance with the fact that the Taylor coefficients of $\Ai(z)$ satisfy $a_n \neq 0$ if and only if $n \not\equiv 2 \pmod{3}$.
Altogether, Theorem~\ref{thm:crg} gives us then
\begin{equation}\label{eq:airycond}
\lim_{n\to\infty: a_n \neq 0} \kappa_\diamond(n) = \frac{2}{\sqrt{3}},\qquad \omega = \lim_{n\to\infty: a_n \neq 0} \frac{\bar\kappa(n)}{\sqrt{2\pi\rho n}} = \frac{1}{\sqrt{3}}.
\end{equation}
We observe that the general upper bound given in (\ref{eq:kappaOmega}) is sharp here. An illustration
of the limit result (\ref{eq:airycond}) by some actual numerical data for various $n$ can be found in Table~\ref{tab:airy}.
\end{example}

\begin{example}\label{ex:Bi} As for $\Ai(z)$ in the last example, the discussion of $\Bi(z)$ begins with its asymptotic expansions \cite[Eq.~(10.4.63--65)]{MR0167642} as $z\to\infty$ in different sectors of the complex plane. Skipping the details, we get that $\Bi$ is of completely regular growth with
order $\rho=\frac32$, type $\tau=\frac23$, and Phragmén--Lindelöf indicator
\[
h(\theta) = \tfrac23 |\cos(\tfrac32\theta)| \qquad (|\theta|<\pi).
\]
Thus, $h(\theta)=\tau$ for $\theta=\pm\tfrac23\pi$ and also for $\theta=0$; hence $\Omega=3$. The asymptotic expansions yield
\[
\Bi(re^{\pm \frac23\pi i}) = \frac{e^{\pm \frac\pi3 i}}{2\sqrt{\pi}}r^{-1/4} e^{\tfrac23 r^{3/2}} (1+ O(r^{-3/2}))\qquad(r\to\infty)
\]
and
\[
\Bi(r) = \frac{1}{\sqrt{\pi}}r^{-1/4} e^{\tfrac23 r^{3/2}} (1+ O(r^{-3/2}))\qquad(r\to\infty),
\]
that is $\arg B(r) = 0$ and
\[
\arg \Bi(re^{\pm \frac23\pi i})  = \pm \tfrac\pi3 + O(r^{-3/2})\qquad (r\to\infty).
\]
Hence, as $r\to\infty$,
\[
|\Bi(re^{\pm \frac23\pi i})| \sim \tfrac12 |\Bi(r)|
\]
and thus, as $n\to\infty$,
\begin{align*}
\Big| \sum_{\theta:h(\theta)=\tau} e^{-in\theta} \Bi(r_\diamond e^{i\theta})\Big| &\sim \left| \frac12 e^{\tfrac23 \pi n i}e^{-\tfrac\pi3 i} +  \frac12 e^{-\tfrac23 \pi n i}e^{\tfrac\pi3 i} + 1\right|\cdot |\Bi(r_\diamond)|\\*[2mm]
&= \left| 1 + \cos(\tfrac\pi3 - \tfrac23 \pi n)  \right|\cdot |\Bi(r_\diamond)|,
\end{align*}
and
\[
\sum_{\theta:h(\theta)=\tau} |\Bi(r_\diamond e^{i\theta})| \sim 2 |\Bi(r_\diamond)|, \qquad
 \max_{\theta:h(\theta)=\tau} |\Bi(r_\diamond e^{i\theta})| \sim  |\Bi(r_\diamond)|.
\]
Now,
\[
\left| 1 + \cos(\tfrac\pi3 - \tfrac23 \pi n)  \right| =
\begin{cases}
3/2 & \qquad n \not\equiv 2 \pmod{3},\\
0 & \qquad n \equiv 2 \pmod{3},
\end{cases}
\]
in accordance with the fact that the Taylor coefficients of $\Bi(z)$ satisfy $a_n \neq 0$ if and only if $n \not\equiv 2 \pmod{3}$. Altogether, Theorem~\ref{thm:crg} gives us then
\begin{equation}\label{eq:Bicond}
\lim_{n\to\infty: a_n \neq 0} \kappa_\diamond(n) = \frac{4}{3},\qquad \omega = \lim_{n\to\infty: a_n \neq 0} \frac{\bar\kappa(n)}{\sqrt{2\pi\rho n}} = \frac{2}{3}.
\end{equation}
An illustration
of the limit result (\ref{eq:Bicond}) by some actual numerical data for various $n$ can be found in Table~\ref{tab:Bi}.
\end{example}

\begin{table}[tbp]
\caption{For $f(z)=\Bi(z)$, a comparison of the quasi-optimal radius $r_\diamond(n)$ with its asymptotic value (\ref{eq:rdiamondasympt}) as taken from Table~\ref{tab:functions}. This asymptotic value is already quite accurate for small $n$. The value of $r_\diamond(n) = |z_n| $ was actually computed by numerically solving the saddle point
equation $z_n f'(z_n)/f(z_n)=n$ in the complex plane. Note that $\lim_{n\to\infty} \kappa_\diamond(n) = 4/3\doteq 1.33333$, see (\ref{eq:Bicond}).}
\vspace*{0mm}
\centerline{%
\setlength{\extrarowheight}{3pt}
{\small\begin{tabular}{rrcrc}\hline
$n$ & $r_\diamond(n)$ & $\kappa_\diamond(n)$ & $n^{2/3}$ & $\kappa(n,n^{2/3})$\\ \hline
$1$ &      $1.36603$ & $1.35408$ &  $  1.00000$ & $1.57640$\\
$10$ &     $4.72421$ & $1.37605$ &  $  4.64159$ & $1.39833$\\
$100$ &   $21.58047$ & $1.33751$ &  $ 21.54435$ & $1.33948$\\
$1000$ & $100.01668$ & $1.33375$ &  $100.00000$ & $1.33394$\\*[0.5mm]\hline
\end{tabular}}}
\label{tab:Bi}
\end{table}

\subsection{A Resonant Case: \boldmath $f(z)=1/\Gamma(z)$ \unboldmath}\label{sect:gamma}

In the statement of Theorem~\ref{thm:crg} the condition on the zeros of $f$ cannot be disposed of: if $f$ possesses infinitely many zeros in the vicinity of
its directions of predominant growth,
then it may happen that a {\em pair} of saddle points recombines in the limit $r\to \infty$ to a {\em single} maximum of the indicator function $h(\theta)$.
That is, even though we have $\Omega=1$ in the limit,
the contributions of the two saddle points may yield resonances in (\ref{eq:approxkappa}) as $n\to \infty$; thus $\kappa_\diamond(n)$, as well as $\kappa_*(n)$, may behave quite irregular.

We demonstrate such a behavior for the entire function $f(z) = 1/\Gamma(z)$, whose zeros are located at $0,-1,-2,-3,\ldots$ This function has order $\rho=1$,
 but is of maximal type $\tau=\infty$ \citeaffixed[p.~27]{MR589888}{see}. Therefore, at a first sight, the results so far do not seem to be applicable at all. However, using Valiron's concept of a \emph{proximate order} $\rho(r)$
 it is possible to extend the definition of functions of completely regular growth and of their indicator functions in such a way that the results cited
above still hold true \citeaffixed[§I.12]{MR589888}{see}.
By Stirling's formula, and Euler's reflection formula
\[
\Gamma(z) \cdot \Gamma(1-z) = \frac{\pi}{\sin(\pi z)},
\]
we get the following asymptotic expansion, valid uniformly in $\theta$:
\begin{equation}\label{eq:gammaind}
\frac{\log |1/\Gamma(re^{i\theta})|}{r\log r} = -\cos\theta + \frac{\cos\theta + \theta\sin\theta}{\log r} + O(r^{-1})\qquad (r\to\infty:r \not\in E),
\end{equation}
where the set $E$ of possible exceptions has relative linear density zero. From this we can read off that $1/\Gamma(z)$ is a function of completely regular growth with
a proximate order $\rho(r)$ given by $r^{\rho(r)} = r\log r$; the indicator function is then
\[
h(\theta) = -\cos\theta.
\]
Now, the problem is that this indicator becomes asymptotically maximal at the {\em single} direction $\theta=\pm \pi$, which is actually the direction of the ray that
contains the countable many zeros of $1/\Gamma(z)$. In fact, a closer look at (\ref{eq:gammaind}) reveals that this single maximum is formed, in the limit $r\to \infty$,
through a recombination of \emph{two} distinct maxima for finite $r$.
And indeed, Figure~\ref{fig:gamma}.a shows quite an irregular behavior of the quasi-optimal condition number $\kappa_\diamond(n)$ (the picture would be essentially the same
for the optimal condition number $\kappa_*(n)$ itself, though much more difficult to compute).

\begin{figure}[tbp]
\begin{center}
\begin{minipage}{0.495\textwidth}
\begin{center}
\includegraphics[width=\textwidth]{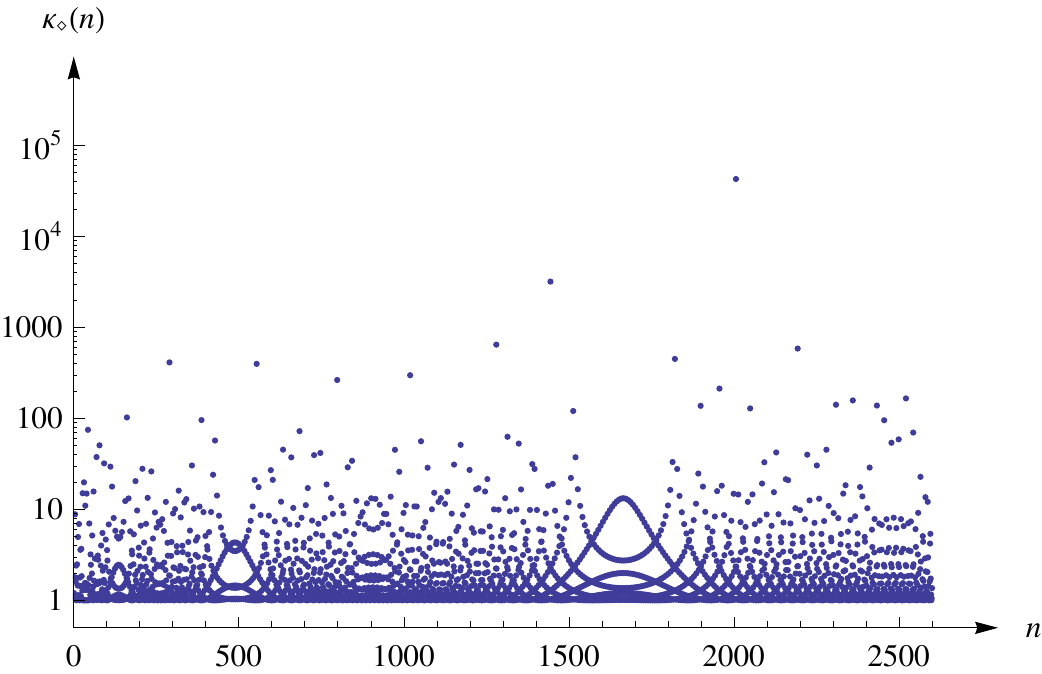}\\*[-1.5mm]
{\tiny a.\;\; quasi-optimal condition number $\kappa_\diamond(n)$ ($1\leq n \leq 2600$)}
\end{center}
\end{minipage}
\hfill
\begin{minipage}{0.495\textwidth}
\begin{center}
\includegraphics[width=\textwidth]{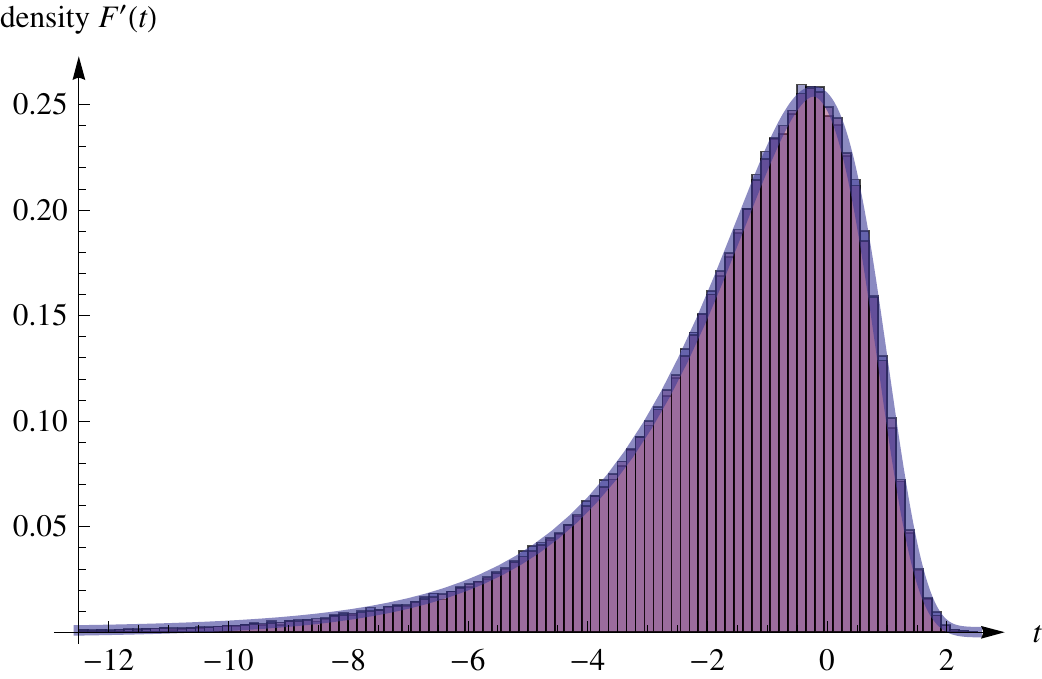}\\*[-1.5mm]
{\tiny b.\;\; histogram of $\log\log\kappa_\diamond(n)$}
\end{center}
\end{minipage}
\end{center}
\caption{Left: plot of the quasi-optimal condition number $\kappa_\diamond(n)$ ($1\leq n \leq 2600$); $f(z)=1/\Gamma(z)$. Within the shown range of $n$,
the maximum is taken for $n=2006$: $\kappa_\diamond(2006) \doteq 47\,067.2$. Note that there is not much to be gained from using the optimal radius $r_*(n)$ instead of $r_\diamond(n)$:
$\kappa_*(2006) \doteq 47\,063.9$. Right: plot of the density histograms of $t = \log\log\kappa_\diamond(n)$ ($1\leq n \leq N$) for $N=100\,000$ and $N=1\,000\,000$ and of the density $F'(t)$
belonging to the distribution $F(t) = \tfrac2\pi\arccos(\exp(-e^t))$, printed transparently on top of each other.
Since there is such a close agreement we are led to conjecture the limit law (\ref{eq:gammaconj})
 and, therefore, $\liminf_{n\to\infty} \kappa_\diamond(n) = 1$ and $\limsup_{n\to\infty} \kappa_\diamond(n)=\infty$.}\label{fig:gamma}
\end{figure}

The quasi-optimal radius $r_\diamond(n)$ can straightforwardly be obtained by means of the saddle-point equation (\ref{eq:zn_1}): that is,
$r_\diamond(n) = |z_n|$ where $z_n$ is one of the \emph{two} complex conjugate solutions of
\[
n= z \frac{d}{dz}\log \frac{1}{\Gamma(z)} = -z\psi(z);
\]
we choose $\Im\, z_n >0$ for definiteness.
Asymptotically, as $n\to\infty$, this saddle-point equation can actually be solved explicitly in terms of the
principal branch of the Lambert $W$-function: using the asymptotic expansion \cite[Eq.~(6.3.18)]{MR0167642} of the digamma function $\psi$ we obtain
\[
-z\psi(z) = -z\log z + \frac{1}{2} + O(z^{-1})\qquad (|\arg z| < \pi),
\]
and therefore, as $n \to \infty$,
\begin{equation}
z_n \sim \frac{\tfrac12 -n}{W(\tfrac12 -n)} = e^{W(\tfrac12 -n)} = r_n e^{i\theta_n},\qquad r_\diamond(n) \sim e^{\Re W(\tfrac12 -n)} = r_n,
\end{equation}
which we take as the definition of the radius $r_n$ and the angle $\pi/2 < \theta_n < \pi$.

A detailed quantitative analysis of $\kappa_\diamond(n)$ can now be based on the well-known fact \citeaffixed[p.~91]{MR0080749}{see} that the saddle-point
analysis of §\ref{sect:saddlepointmethod} is applicable to $f(z)=1/\Gamma(z)$: in fact the approximations (\ref{eq:ansaddle}) and (\ref{eq:Mnsaddle}) are
asymptotic equalities as $n\to\infty$. We find that
they can be recast  in the form
\begin{subequations}
\begin{align}
a_n &\sim \sqrt{\frac{2}{\pi n}}\, \frac{|1/\Gamma(r_n e^{i\theta_n})|}{r_n^n} \cos\phi_n,\\*[1mm]
\bar\kappa_\diamond(n)  &\sim \sqrt{\frac{\pi n}{2}} |\sec\phi_n|,\\*[3mm]
\kappa_\diamond(n) &\sim |\sec\phi_n|\label{eq:kappagamma}
\end{align}
\end{subequations}
with the collective phase approximation\footnote{\citeasnoun[p.~92]{MR0080749} basically states the same results with the much simpler phase approximation
\[
\phi_n = (n-\tfrac12)(\theta_n^{-1}\sin^2\theta_n -\theta_n),
\]
which is, however, numerically far less accurate for small values of $n$ and would not allow such a precise prediction of $\kappa_\diamond(n)$ as in Table~\ref{tab:gamma}.}
\[
\phi_n = \left(n-\frac12\right)\left(\frac{\sin^2\theta_n}{\theta_n}-\theta_n + \frac{\theta_n}{12(n-\frac12)^2}\right) - \frac12 \arccot\left(\cot\theta_n-\theta_n\csc^2\theta_n\right).
\]
The asymptotics (\ref{eq:kappagamma}) does not only explain the very \emph{possibility} of resonances, it actually
gives excellent \emph{numerical} predictions even for rather small values of $n$ such as those illustrated in Table~\ref{tab:gamma}.

\begin{table}[tbp]
\caption{The precision of the asymptotics (\ref{eq:kappagamma}) near some resonances.}
\vspace*{0mm}
\centerline{%
\setlength{\extrarowheight}{3pt}
{\small\begin{tabular}{rrrrrr}\hline
$n$ & $\kappa_\diamond(n)$ & $|\sec(\phi_n)|$ & $n$ & $\kappa_\diamond(n)$ & $|\sec(\phi_n)|$\\ \hline
$2002$ & $1.018$& $1.018$           &\qquad$10931$ &$1.006$       &$1.006$\\
$2003$ & $1.034$& $1.033$           &$10932$ &$1.124$       &$1.124$\\
$2004$ & $1.301$& $1.300$           &$10933$ &$1.498$       &$1.497$\\
$2005$ & $2.354$& $2.352$           &$10934$ &$2.798$       &$2.797$\\
$2006$ & $47067.162$& $42811.637$   &$10935$ &$138149.749$  &$143720.416$\\
$2007$ & $2.355$& $2.353$           &$10936$ &$2.798$       &$2.797$\\
$2008$ & $1.301$& $1.300$           &$10937$ &$1.498$       &$1.497$\\
$2009$ & $1.034$& $1.033$           &$10938$ &$1.124$       &$1.124$\\
$2010$ & $1.018$& $1.017$           &$10939$ &$1.006$       &$1.006$\\*[0.5mm]\hline
\end{tabular}}}
\label{tab:gamma}
\end{table}

Based on Table~\ref{tab:gamma} and Figure~\ref{fig:gamma}.a  it is certainly quite reasonable to \emph{conjecture} that $\liminf_{n\to\infty} \kappa_\diamond(n) = 1$.
On the other hand, by just looking at the rather randomly distributed positions $n$ of the resonances and the corresponding extreme values of $\kappa_\diamond(n)$
we could not really establish any serious conjecture about the probable value of $\limsup_{n\to\infty} \kappa_\diamond(n)$. Instead, we look at the \emph{statistics} of the values of $\kappa_\diamond(n)$ for
 $1 \leq n \leq N$. The very close agreement of the two histograms shown in Figure~\ref{fig:gamma}.b suggests that there should be a limit
 law of  the form
 \begin{subequations}\label{eq:gammaconj}
\begin{equation}\label{eq:gammaconj1}
\lim_{N\to\infty} N^{-1} \cdot \# \{ 1\leq n \leq N : \log\log\kappa_\diamond(n) \leq t\} = F(t).
\end{equation}
If the phases $\phi_n$ were equidistributed modulo $\pi$ (and the empirical data of the first one million instances strongly point into that direction)
we would immediately find from (\ref{eq:kappagamma}) that the distribution would be
\begin{equation}\label{eq:gammaconj2}
F(t) = \frac2\pi\arccos(e^{-e^t}).
\end{equation}
 \end{subequations}
In fact, we observe that the thus given density $F'(t)$ is very well approximated by the histograms in Figure~\ref{fig:gamma}.b and
we therefore conjecture that the limit law (\ref{eq:gammaconj}) is correct. Now, since $F'(t)>0$ for all $t\in \R$, this conjecture would also imply that
\[
\liminf_{n\to\infty} \kappa_\diamond(n) = 1,\qquad \limsup_{n\to\infty} \kappa_\diamond(n) = \infty.
\]
Actually, things are not as bad as such a spread of the condition number might suggest: from $\tfrac2\pi\arccos(1/100)\doteq 0.9936$ we infer that just about $0.64\%$ of all $n$
(in the sense of natural density) have $\kappa_\diamond(n) \geq 100$; that is, as much as at least $99.36\%$ of all the Taylor coefficients $a_n$ enjoy to be computed with a loss of less than
two digits. We find that the asymptotic median of $\kappa_\diamond(n)$ would be as small as $\sqrt{2}$.

\begin{remark} In the same vein, a worst-case analysis based on Figure~\ref{fig:gamma}.a tells us that there will be just a loss of at most \emph{three} digits in computing
the first one thousand of the Taylor coefficients of
\[
\frac{1}{\Gamma(z)} = \sum_{n=0}^\infty a_n z^n
\]
by means of a Cauchy integral with radius $r_\diamond(n)$. Note that the only competitor of this approach, namely using the recursion formula  \citeaffixed[§2.10]{MR0241700}{see}
\[
a_0=0,\quad a_1=1,\quad a_2 = \gamma,\quad a_n = n a_1 a_n - a_2 a_{n-1} + \sum_{k=2}^{n-1} (-1)^k \zeta(k)a_{n-k} \quad (n>2),
\]
is much worse behaved and suffers from severe numerical instability almost right from the beginning: in hardware arithmetic all the digits are lost for $n\geq 27$.
\end{remark}

\section{{\em H}-Admissible Entire Functions}\label{sect:hayman}

The function $f(z) = e^{e^z-1}$ of Example~\ref{ex:diamondbell} is not covered by our results so far: it has order $\rho=\infty$. Nevertheless, the general idea of using the saddle-point method (see §\ref{sect:saddle})
can certainly also be applied to functions that grow even stronger than $f$. \citeasnoun{MR0080749}
has axiomatized an important class of functions (with predominant growth in the direction of the real axis), for which the saddle-point method is applicable along {\em circular} contours and which enjoys nice closure properties.
Expositions of this method can be found in \citeasnoun[§II.79]{MR1016818}, \citeasnoun[§12.2]{MR1373678},
\citeasnoun[§5.4]{MR2172781}, and \citeasnoun[§VIII.5]{MR2483235}.

Hayman's method is based on the Taylor expansion (\ref{eq:taylor_angular_gen}) with the expansion point $z_*=r$, that is, on the Taylor expansion
\begin{subequations}\label{eq:ab}
\begin{equation}\label{eq:ab1}
\log f(r e^{i\theta}) = \log f(r) + i a(r) \theta - \frac{1}{2} b(r) \theta^2 + O(\theta^3)\qquad (\theta\to 0),
\end{equation}
where the coefficients are given by
\begin{equation}\label{eq:ab2}
a(r) = r \,\frac{f'(r)}{f(r)} = r \frac{d}{dr} \log f(r),\qquad b(r) = r a'(r).
\end{equation}
\end{subequations}
Now, an entire function $f(z)$ that is positive on $(r_0,\infty)$
for some $r_0>0$ is said to be {\em $H$-admissible}, if it satisfies the following three conditions:
\begin{itemize}
\item $b(r) \to \infty$ as $r\to \infty$;\\*[-3mm]
\item for some function $\theta_0(r)$ defined over $(r_0,\infty)$ and satisfying $0< \theta_0(r) < \pi$, one has,
uniformly in $|\theta| \leq \theta_0(r)$,
\[
f(r e^{i\theta}) \sim f(r) e^{i\theta a(r) - \theta^2 b(r)/2}\qquad (r\to \infty);
\]
\item uniformly in $\theta_0(r) \leq |\theta| \leq \pi$
\[
f(re^{i\theta}) = \frac{o(f(r))}{\sqrt{b(r)}}.
\]
\end{itemize}
However, one rarely checks these conditions directly but relies on the following closure properties instead.

\begin{theorem}[Hayman 1956]\label{thm:closure} Let $f$ and $g$ be $H$-admissible entire functions and
let $p$ be a polynomial with real coefficients. Then:
\begin{itemize}
\item[(a)] the product $f(z)g(z)$ and the exponential $e^{f(z)}$ are admissible;\\*[-3mm]
\item[(b)] the sum $f(z) + p(z)$ is admissible;\\*[-3mm]
\item[(c)] if the leading coefficient of $p$ is positive then
$f(z) p(z)$ and $p(f(z))$ are admissible;\\*[-3mm]
\item[(d)] if the Taylor coefficients of $e^{p(z)}$ are eventually positive then $e^{p(z)}$ is admissible.
\end{itemize}
\end{theorem}
For instance, with the help of this theorem it is fairly obvious to see that the functions $f(z)=e^z$ and $f(z)=e^{e^z-1}$ are both $H$-admissible. On the
other hand, the $H$-admissibility of functions like $f(z)=z^{-k/2}I_k(2\sqrt{z})$ has to be inferred more labor-intensive from the definition.

From the definition of $H$-admissibility we immediately read off that the maximum modulus function is given, for $r$ large enough, by
\begin{equation}\label{eq:MHayman}
M(r) = f(r),
\end{equation}
which, by the strict convexity of $\log M(r)$ with respect to $\log r$,\footnote{Note that this strict convexity implies $b(r) = (r \frac{d}{dr})^2 \log f(r) >0$ for all $r>r_0$.} by Theorems~\ref{thm:rdiamondgrowth} and
\ref{thm:zn}, implies that the quasi-optimal radius $r_\diamond = r_\diamond(n)$ is the \emph{unique}
solution of
\begin{equation}\label{eq:rnfroma}
a(r_\diamond) = n
\end{equation}
for $n$ large enough. Hayman's main results are summarized in the following theorem.

\begin{theorem}[Hayman 1956]\label{thm:HaymanMain}
Let $f$ be an entire $H$-admissible function. Then, for the quasi-optimal radius $r_\diamond=r_\diamond(n)$, we have\footnote{Note that
(\ref{eq:Hayman_an}) can be thought of as being a generalization of Stirling's formula,
cf. Examples~\ref{ex:exp} and \ref{ex:diamondexp}: this was the original headline of \possessivecite{MR0080749} work.}
\begin{equation}\label{eq:Hayman_an}
a_n \sim r_\diamond^{-n} M_1(r_\diamond) \sim\frac{f(r_\diamond)}{r_\diamond^n \sqrt{2\pi b(r_\diamond)}} \qquad (n \to \infty);
\end{equation}
in particular, we get $a_n > 0$ for $n$ large enough. Moreover, we have, uniformly in the integers $n$,\footnote{Because of $f(r)=\sum_{k=0}^\infty a_k r^k$, the quantities $a_n r^n/f(r)$ form, if
$a_n\geq0$ for all $n$, a probability distribution in the discrete variable $n$. The result (\ref{eq:boltzmann}) thus tells us that this probability
distribution is asymptotically, in the limit of large radius $r\to\infty$, Gaussian with mean $a(r)$ and variance~$b(r)$.}
\begin{equation}\label{eq:boltzmann}
\frac{a_n r^n}{f(r)} = \frac{1}{\sqrt{2\pi b(r)}}\left( \exp\left(-\frac{(n-a(r))^2}{2b(r)}\right) + o(1) \right)\qquad (r \to \infty).
\end{equation}
Finally,  the ratio $a_n/a_{n+1}$ forms an eventually increasing sequence since\footnote{\label{ft:divakar}Note that the asymptotic representation (\ref{eq:divakar}) of the quasi-optimal radius holds for the generalized hyperbolic
functions (\ref{eq:pFq}) with $p\leq q$, too: namely, we have by Theorem~\ref{thm:rdiamond}, Example~\ref{ex:hyper} and Remark~\ref{rem:hyper} that
\begin{equation}\label{eq:divakar2}
r_\diamond(n) \sim n^{q+1-p} \sim \left|\frac{a_n}{a_{n+1}}\right| \sim \left|\frac{a_{n-1}}{a_n}\right| \qquad (n\to\infty).
\end{equation}
On the other hand, such a representation is not valid for the function of Example~\ref{ex:detmock}. However, there the following corollary of (\ref{eq:divakar2}) is nevertheless correct:
\begin{equation}\label{eq:divakar3}
r_\diamond(n) \sim \sqrt{\left|\frac{a_{n-1}}{a_{n+1}}\right|} \qquad (n\to\infty).
\end{equation}
Hence, if we restrict ourselves to those $n$ for which $a_{n-1},a_{n+1}\neq 0$, we observe that (\ref{eq:divakar3}) does in fact hold for all the functions of Table~\ref{tab:functions}
but the function $1/\Gamma(z)$. Whether this fact is just
a contingency or whether it is for some deeper structural reason, we do not yet know.
}
\begin{equation}\label{eq:divakar}
r_\diamond(n) \sim \frac{a_n}{a_{n+1}} \sim \frac{a_{n-1}}{a_n} \qquad (n\to\infty).
\end{equation}
\end{theorem}

As for the conditions numbers, we straightforwardly get the following corollary; for reasons of a better comparison
we have also included the quantities of the Wiman--Valiron theory as introduced in §\ref{sect:WimanValiron} (their
asymptotics can directly be read off from (\ref{eq:boltzmann})).

\begin{cor}\label{cor:admissible} Let $f$ be an entire $H$-admissible function. Then
\[
\lim_{n\to\infty}\kappa_\diamond(n) = 1,\qquad \lim_{n\to\infty}\frac{\bar\kappa_\diamond(n)}{\sqrt{2\pi b(r_\diamond(n))}}= 1.
\]
Moreover we have $\nu(r) \sim a(r)$ as $r \to \infty$ and
\[
M(r_\diamond(n)) \sim \sqrt{2\pi b(r_\diamond(n))} \, \mu(r_\diamond(n))\qquad (n\to \infty).
\]
\end{cor}

\smallskip

Applications have already been discussed in Examples~\ref{ex:diamondexp} and
\ref{ex:diamondbell}.

\begin{remark} If the entire $H$-admissible function $f$ is of finite order $\rho$ with normal type
$\tau$, it is instructive to compare Corollary~\ref{cor:admissible} with Theorem~\ref{thm:condcrg}. From the
definition of $H$-admissibility it then follows that:
\begin{itemize}
\item $f$ is of perfectly and of completely regular growth;\\*[-3mm]
\item there is just one direction of predominant growth, $\Omega=1$ with $h(0)=\tau$;\\*[-3mm]
\item $f$ has at most finitely many zeros in the vicinity the positive real axis.
\end{itemize}
Thus, $f$ satisfies the assumptions of Theorem~\ref{thm:condcrg} and also those that have led to the
definition (\ref{eq:omega}) of $\omega$.
Therefore, by $\Omega=1$ we get from (\ref{eq:barkappaomega}) and (\ref{eq:kappaOmega}) that $\omega=1$ and
\[
\lim_{n\to\infty} \kappa_\diamond(n) =1,\qquad \lim_{n\to\infty}\frac{\bar\kappa_\diamond(n)}{\sqrt{2\pi \rho n}} = 1.
\]
(Recall that $H$-admissible functions have $a_n>0$ for $n$ large enough.) Further, by Theorem~\ref{thm:mu1} we have
$\nu(r) \sim \tau \rho r^{\rho}$. These results are consistent with Corollary~\ref{cor:admissible}; a comparison
gives, by using (\ref{eq:rdiamondasympt}), the asymptotic equations
\begin{equation}\label{eq:abrhotau}
a(r) \sim \tau \rho r^{\rho},\qquad b(r) \sim \tau \rho^2 r^\rho\qquad (r\to\infty).
\end{equation}
Formally, as suggested by (\ref{eq:ab2}), these equations could have been obtained from differentiating the
asymptotic equation $\log f(r) = \log M(r) \sim \tau r^\rho$ (which just states the perfectly regular
growth of the function $f$, see (\ref{eq:prg}) for the definition). The differentiability of these
asymptotic equations has also been observed by \citeasnoun[IV.70/71]{MR0170986} under the weaker assumption that $f$ is a function of perfectly regular growth with $a_n \geq 0$.
\end{remark}

\section{Entire Functions with Non-Negative Taylor Coefficients}\label{sect:positive}

In this final section we consider entire transcendental functions $f$ which have non-negative Taylor coefficients: $a_n \geq 0$ for all $n$. Such functions are typically met as generating
functions in combinatorial enumeration or in probability theory. The non-negativity of the Taylor coefficients implies at once that
\begin{equation}\label{eq:Mnonneg}
M(r) = f(r) \qquad (r>0).
\end{equation}
Thus, by Theorem~\ref{thm:threecircle}, we infer that $\log f(r)$ and hence $\log(r^{-n} f(r))$ are strictly convex functions of $ \log r$. Moreover,
since
\[
\frac{d^2}{dr^2} r^{-n} f(r) = r^{-n-2} \sum_{k=0}^\infty (n+1-k)(n-k) a_k r^k > 0 \qquad (r>0),
\]
we conclude that the function $r^{-n} f(r)$ itself is strictly convex, too. The same reasoning that led to (\ref{eq:rnfroma}) in the last section proves the following simplification of Theorem~\ref{thm:zn}.

\begin{theorem}\label{thm:nonneg}
Let $f$ be an entire transcendental function with non-negative Taylor coefficients: $a_n \geq 0$ for all $n$. Then, the quasi-optimal radius $r_\diamond = r_\diamond(n)$ is given as the unique
 solution of the convex optimization problem
\begin{equation}\label{eq:rdiamond_min_nonneg}
r_\diamond = \argmin_{r>0} r^{-n} f(r),
\end{equation}
and, equivalently, as the unique solution of the real saddle-point equation
\begin{equation}\label{eq:rdiamond_nonneg}
r_\diamond \frac{f'(r_\diamond)}{f(r_\diamond)} = n\qquad (r_\diamond >0).
\end{equation}
\end{theorem}

\begin{remark}\label{rem:rdiamond} If $f$ is a function of perfectly regular growth (of order $\rho$ and type~$\tau$) with non-negative Taylor coefficients, Theorem~\ref{thm:nonneg} yields the assertion of Theorem~\ref{thm:rdiamond} with a proof that is much shorter than the one of the general result given there:
first, from the definition of perfectly regular growth in (\ref{eq:prg}) and from (\ref{eq:Mnonneg}) we get
\[
\log f(r) \sim \tau r^{\rho}\qquad (r\to \infty);
\]
next, since the Taylor coefficients are non-negative, we may differentiate this asymptotic equation \citeaffixed[IV.70]{MR0170986}{see} and obtain
\[
r\, \frac{f'(r)}{f(r)} \sim \tau \rho  r^{\rho}\qquad (r\to \infty);
\]
therefore, by recalling $r_\diamond(n)\to\infty$ as $n\to \infty$ (see Theorem~\ref{thm:rn}), the saddle-point equation (\ref{eq:rdiamond_nonneg}) is asymptotically solved by
\[
r_\diamond(n) \sim \left(\frac{n}{\tau\rho}\right)^{1/\rho}\qquad (n\to\infty),
\]
which is, finally, the assertion of Theorem~\ref{thm:rdiamond}.
\end{remark}

\begin{example}\label{ex:rmt2} We pick up the computation of the gap probabilities $E_2(n;s)$ as discussed in Example~\ref{ex:rmt1}, this time striving for small relative errors instead of absolute errors.
As we have seen, the generating function is given by the Fredholm determinant
\[
f(z) = \sum_{k=0}^\infty E_2(k;s) \, z^k = \det\left(I-(1-z) K\projected{L^2(0,s)}\right),\quad K(x,y) = \sinc(\pi(x-y)),
\]
which is known to be, as a function of $z$, an entire function of order $\rho=0$.\footnote{Generally, if the kernel $K(x,y)$ satisfies a Hölder condition with exponent $\alpha$, with respect to
either $x$ or $y$, then $f(z) = \det(I-z K)$ is an entire function of order $\rho \leq 2/(1+2\alpha)$; see, e.g., \citeasnoun[Lemma~24.10]{MR1892228}.}
Figure~\ref{fig:rmt2}.a shows that, for $n=10$, the quasi-optimal radius $r_\diamond$ (as computed from~(\ref{eq:rdiamond_min_nonneg}) by means of Matlab's {\tt fminbnd} command) varies over about 20 orders of magnitude as
the parameter $s$ runs through the interval $2\leq s \leq 18$. The corresponding
condition number satisfies $\kappa_\diamond \doteq 1$, up to machine precision throughout. We will explain this optimal condition number result and the strong variability of the radius by
discussing a ``mock-up'' model in the
next example. Note that even though Figure~\ref{fig:rmt2}.b shows a significant accuracy improvement in the tails, we do not get the full accuracy that we would have expected from $\kappa_\diamond\doteq 1$. The reason
is simply that the numerical evaluation of the Fredholm determinants does not satisfy the model assumption~(\ref{eq:relmodel}); see \citeasnoun[§4]{Bornemann}.
\end{example}

\begin{figure}[tbp]
\begin{center}
\begin{minipage}{0.495\textwidth}
\begin{center}
\includegraphics[width=\textwidth]{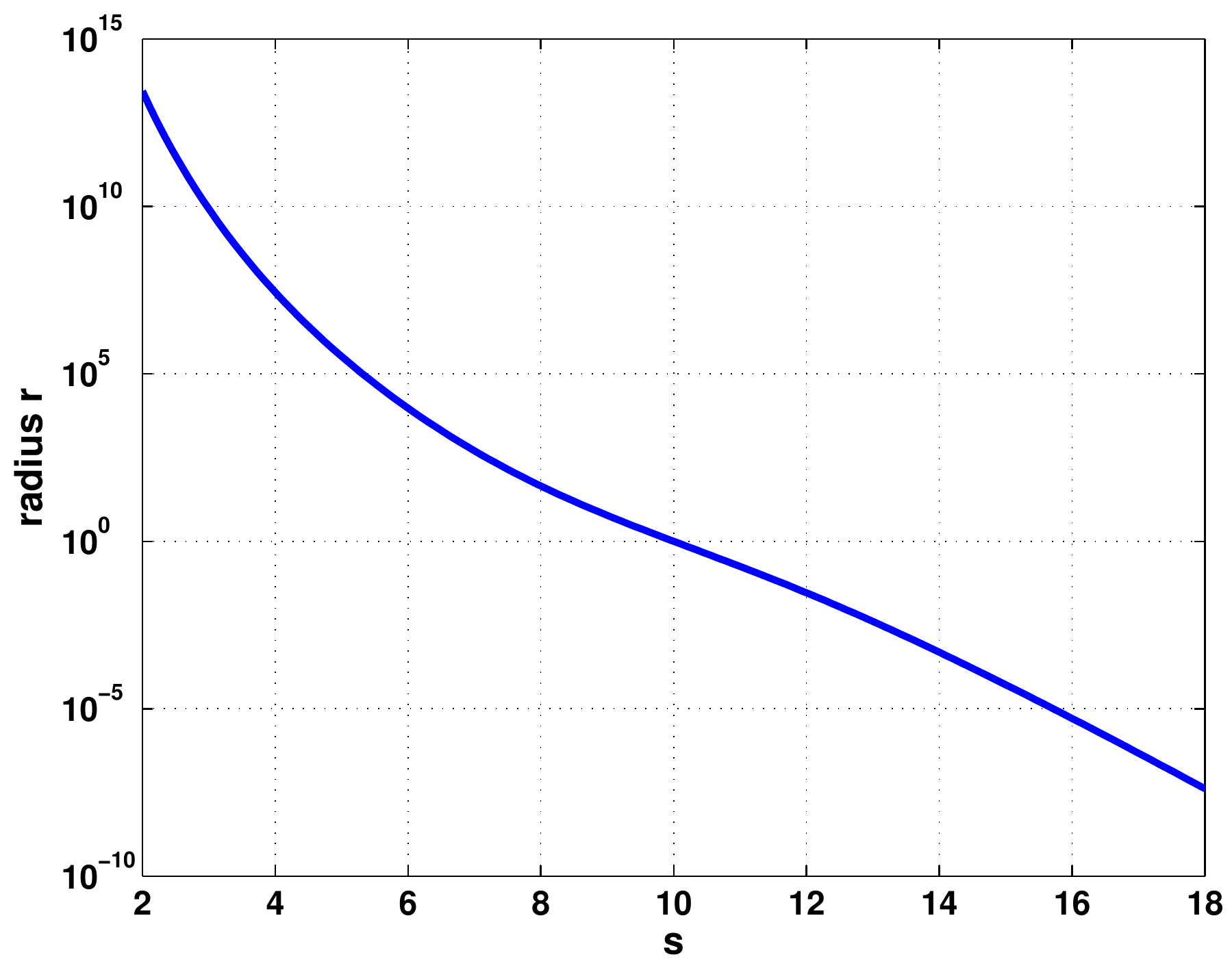}\\*[-1.5mm]
{\tiny a.\;\; quasi-optimal radius $r = r_\diamond$ as a function of $s$}
\end{center}
\end{minipage}
\hfill
\begin{minipage}{0.495\textwidth}
\begin{center}
\includegraphics[width=\textwidth]{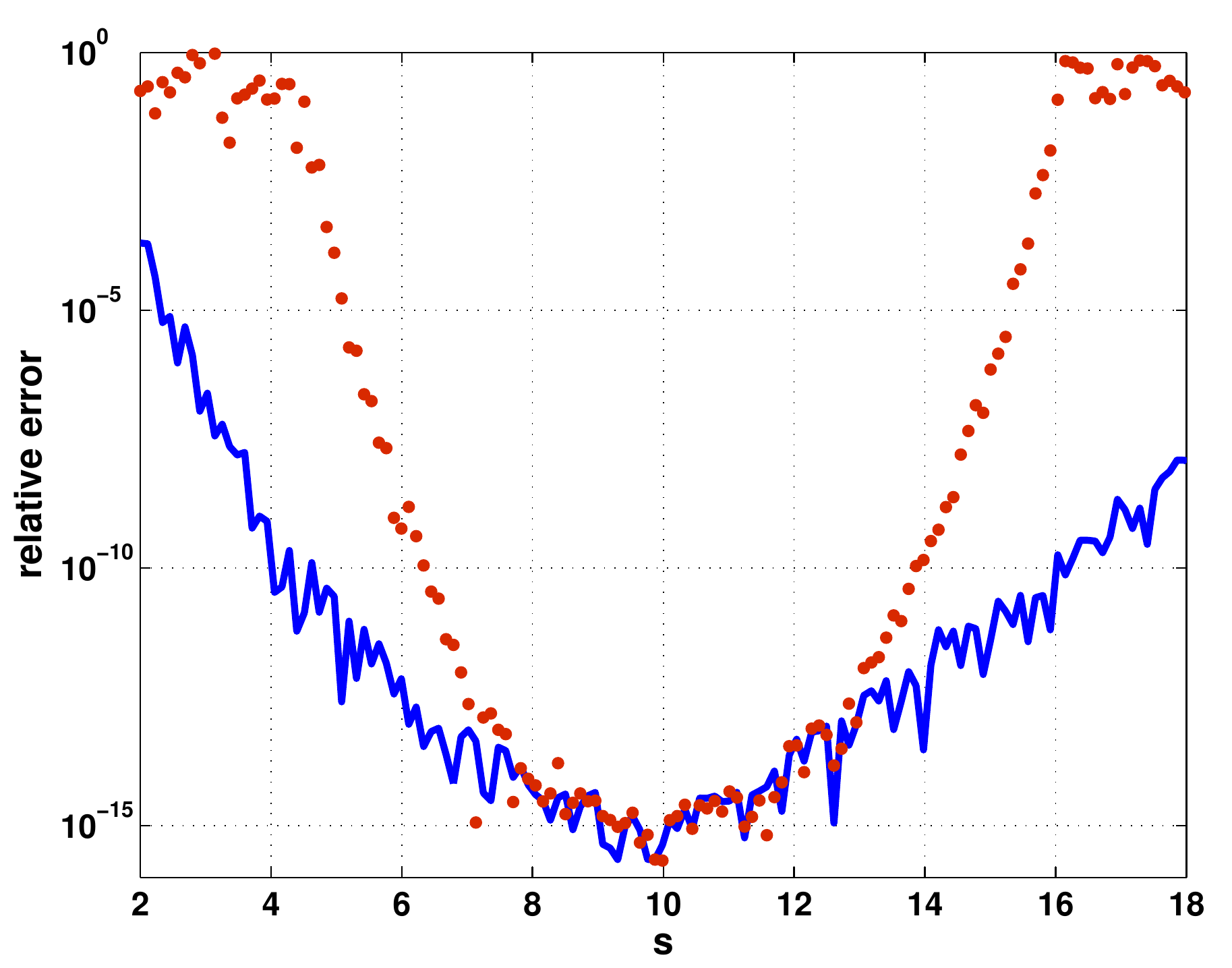}\\*[-1.5mm]
{\tiny b.\;\; relative error}
\end{center}
\end{minipage}
\end{center}
\caption{Left: the quasi-optimal radius $r_\diamond$ as a function of $s$ for calculating the gap probability $E_2(10;s)$ of GUE as the $10$-th Taylor coefficient of a Fredholm determinant. Right: the relative
 error of the calculation; the dotted line (red) shows the errors for the radius $r=1$, the solid line (blue) shows the errors for the radius $r_\diamond$ (see also Example~\ref{ex:rmt1} and Figure~\ref{fig:rmt1}). Note that, though $\kappa_\diamond \doteq 1$ throughout the range of $s$, there is still a noticeable loss of accuracy in the tails: this is because the
 Fredholm determinant is not computed to small relative but small absolute errors; hence the model assumption~(\ref{eq:relmodel}) is violated. Nevertheless, $r_\diamond$
 gives a significant improvement over the fixed radius $r=1$ that belongs to the concept of absolute errors.}\label{fig:rmt2}
\end{figure}

\begin{example}\label{ex:detmock} Lacking a proof that $\kappa_\diamond \approx 1$ in Example~\ref{ex:rmt2} we analyze a ``mock-up'' Fredholm determinant in full detail, namely the $q$-series
\[
f(z) = (-z;q)_\infty = \prod_{k=0}^\infty (1+z q^k)\qquad (0<q<1).
\]
By a result of Euler \citeaffixed[Cor.~10.2.2]{MR1688958}{see} it is known that
\[
f(z) = \sum_{k=0}^\infty \frac{q^{\binom{k}{2}}}{(q;q)_k} \,z^k,
\]
where
\[
(q;q)_k = (1-q)(1-q^2)\cdots(1-q^k);
\]
in particular, $f$ has positive Taylor coefficients. By using (\ref{eq:order}), $f$ is easily seen to be an entire function of order $\rho=0$.
Now, a numerical experiment shows that $\kappa_\diamond(n) \doteq 1$ up to machine precision for $n=20$ and $q=1/2$. Hence we aim at proving that $\kappa_\diamond(n) \to 1$ as $n\to \infty$.

A natural first try would be to check $f$ for $H$-admissibility, see Corollary~\ref{cor:admissible}. This approach is doomed to fail, however, since we get
the following asymptotics from an application of
the Euler--Maclaurin sum formula:
\begin{equation}\label{eq:amockup}
a(r) = r\,\frac{f'(r)}{f(r)} =  \sum_{k=0}^\infty \frac{r q^k}{1+r q^k} = \frac{\log r}{\log(1/q)} + \frac{1}{2} + O(r^{-1})\qquad (r\to\infty);
\end{equation}
but
\[
b(r) = r a'(r) = \sum_{k=0}^\infty \frac{r q^k}{(1+r q^k)^2} = \frac{1}{\log(1/q)} + O(r^{-1})\qquad (r\to\infty),
\]
which remains bounded (recall that $H$-admissibility would require $b(r)\to\infty$).
Nevertheless, from (\ref{eq:rdiamond_nonneg}) and (\ref{eq:amockup}) we get the strong estimate
\begin{equation}
r_\diamond(n) = q^{1/2-n} + O(1)\qquad (n\to\infty),
\end{equation}
which not only shows that $r_\diamond(n)$ grows exponentially fast for this slowly growing function $f(z)$, but also that $r_\diamond$ varies very strongly with respect to the
parameter $q$ (an effect that we had already observed in Example~\ref{ex:rmt2}).

To study $\kappa_\diamond(n)$ we look more closely at $\log f(re^{i\theta})$ for large values of $r$. Applying the Euler--Maclaurin sum formula once more and using a uniformity criterion of \citeasnoun[p.~142]{MR589888}, we get
that
\begin{subequations}
\begin{equation}\label{eq:Remock}
\Re \log f(re^{i\theta}) = \frac{1}{2}\frac{\log(r)^2}{\log(1/q)} + \frac{1}{2}\log r + \frac{1}{12} \log(1/q) + \frac{\pi^2 - 3\theta^2}{6\log(1/q)} + O(r^{-1}),
\end{equation}
and
\begin{equation}\label{eq:Immock}
\Im \log f(re^{i\theta}) = \theta\frac{\log r}{\log(1/q)} + \frac{1}{2}\theta + \frac{q }{1-q}\sin(\theta) r^{-1} + O(r^{-2}),
\end{equation}
\end{subequations}
uniformly in $\theta$ as $r\to\infty$ ($r \not \in E$) with the possible exception of a set $E$ of relative linear density zero. The first asymptotics, (\ref{eq:Remock}), means that $f$ is of completely regular growth with the proximate
order $\rho(r)$ \citeaffixed[§I.12]{MR589888}{see},\footnote{Note that, consistent with $\rho=0$, we have $\rho(r)\to0$ as $r\to\infty$.}
\[
r^{\rho(r)} =  \frac{1}{2}\frac{\log(r)^2}{\log(1/q)}
\]
and {\em constant} indicator function $h(\theta) = 1$. This implies that the growth of $f$ is not localized enough in the angular direction as to hope for an application of the Laplace method to estimate the Cauchy integrals.
In other words, the second stage of using the saddle-point method \cite[p.~77]{MR671583} seems to be about to fail. However, this is not the case here, since the \emph{whole} circular contour of radius $r_\diamond=r_\diamond(n)$
is approximately a level line of $\Im \log z^{-n}f(z)$, not just the segment near the saddle point itself: in fact, from  (\ref{eq:Immock}) we get that
\begin{equation}\label{eq:phasemock}
\Im \log f(r_\diamond e^{i\theta}) - n \theta = \frac{q^{n+1/2}}{1-q} \sin\theta + O(q^{2n}) \qquad (n\to \infty),
\end{equation}
which is exponentially close to zero. Note that we can arrange for $r_\diamond(n) \not\in E$ since~$E$ is built from sets of increasingly small neighborhoods of the radii of the zeros of $f$, which are located at $-q^{-k}$, $k \in \N_0$. That is, (\ref{eq:phasemock}) holds uniformly in $\theta$.
Hence we get
\begin{multline*}
a_n = \frac{1}{2\pi r_\diamond^n} \int_{-\pi}^\pi e^{-in\theta} f(r_\diamond e^{i\theta})\,d\theta =  \frac{1}{2\pi r_\diamond^n} \int_{-\pi}^\pi e^{i(\Im \log f(r_\diamond e^{i\theta}) - n \theta)} |f(r_\diamond e^{i\theta})|\,d\theta\\*[2mm]
= \frac{1}{2\pi r_\diamond^n} \int_{-\pi}^\pi (1+ \frac{i q^{n+1/2}}{1-q}\sin\theta + O(q^{2n}))\cdot |f(r_\diamond e^{i\theta})|\,d\theta \\*[2mm]
= \frac{1}{2\pi r_\diamond^n} \int_{-\pi}^\pi  |f(r_\diamond e^{i\theta})|\,d\theta \cdot (1+ O(q^{2n})),
\end{multline*}
since the contribution of the odd function $\sin\theta |f(r_\diamond e^{i\theta})|$ to the integral is zero.
Therefore, we obtain the approximation
\begin{equation}\label{eq:kappaoneexponent}
\kappa_\diamond(n) = 1 + O(q^{2n})\qquad (n\to\infty),
\end{equation}
whose exponentially small error term helps to understand the excellent condition numbers observed in Example~\ref{ex:rmt2}.
\end{example}

\begin{table}[tbp]
\caption{For $f(z)=\phi_\lambda(z)$, a comparison of the quasi-optimal radius $r_\diamond(n)$ with its asymptotic value (\ref{eq:phirdiamond}). Note that this asymptotic value is not necessarily useful
 in practice. The value of $r_\diamond(n) = \argmin r^{-n} f(r)$ was actually computed by using Matlab's {\tt fminbnd} command.}
\vspace*{0mm}
\centerline{%
\setlength{\extrarowheight}{3pt}
{\small\begin{tabular}{rrrcrc}\hline
$n$ & $\lambda$ & $r_\diamond(n)$ & $\kappa_\diamond(n)$ & $(n/\lambda)^2$ & $\kappa(n,(n/\lambda)^2)$\\ \hline
$20$ &  $3$ &    $55.08575$ & $1.00005$ &  $  44.44444$ & $1.39833$\\
$100$ & $15$ &    $108.74559$ & $1.00000$ &  $  44.44444$ & $5.17900\cdot 10^{11}$\\*[0.5mm]\hline
\end{tabular}}}
\label{tab:incr}
\end{table}

\begin{example}\label{ex:gessel}
 We close the paper with a nontrivial example from the theory of random permutations. Let us denote the length of the longest increasing subsequence\footnote{For instance,
the longest increasing subsequence of $\sigma=(3,7,10,5,9,6,8,1,4,2) \in \mathcal{S}_{10}$ is given by $(3,5,6,8)$, hence $\ell(\sigma)=4$ in this case.} of a permutation $\sigma \in \mathcal{S}_n$ by $\ell(\sigma)$.
The probability distribution of $\ell(\sigma)$ that is induced by the uniform distribution on $\mathcal{S}_n$ can be encoded in a family of exponentially generating functions $\phi_\lambda(z)$ via
\begin{equation}\label{eq:distribution}
\prob(\sigma \in \mathcal{S}_n : \ell(\sigma) \leq \lambda) = \left.\frac{d^n}{dz^n} \phi_\lambda^{(n)}(z)\right|_{z=0}\qquad (\lambda,n \in \N).
\end{equation}
Now, the  seminal work of \citeasnoun{MR1041448} shows that $\phi_\lambda(z)$ can be expressed in terms of a Toeplitz determinant,
\begin{equation}\label{eq:toeplitz}
\phi_\lambda(z) = \det\left(I_{|j-k|}(2\sqrt{z})\right)_{j,k=0}^{\lambda-1}.
\end{equation}
Since the modified Bessel functions $z^{-k/2} I_{k}(2\sqrt{z})$ ($k\in \N_0$) are entire functions of perfectly regular growth (of order $\rho=1/2$ and type $\tau=2$, see Table~\ref{tab:functions}), $\phi_\lambda$ must
also be an entire function of perfectly regular growth; its order and type are easily inferred to be
$\rho=1/2$ and $\tau = 2\lambda$. Likewise, we obtain that the Phragmén--Lindelöf indicator of $\phi_\lambda(z)$ is given by
\[
h(\theta) = 2\lambda \cos(\theta/2)\qquad (|\theta|\leq \pi).
\]
Hence, we have $\Omega=1$ and, since there are no zeros of $\phi_\lambda(z)$ in the vicinity of the real axis, $\omega=1$ and $\lim_{n\to\infty} \kappa_\diamond(n) = 1$ by Theorem~\ref{thm:condcrg}.
This explains the very well behaved quasi-optimal condition numbers shown in Table~\ref{tab:incr}.
Theorem~\ref{thm:rdiamond} yields the following  asymptotics of the quasi-optimal radius:
\begin{equation}\label{eq:phirdiamond}
r_\diamond(n) \sim (n/\lambda)^2\qquad (n\to\infty).
\end{equation}
However, as we can see from Table~\ref{tab:incr}, this asymptotics does probably not hold uniformly in $\lambda$ and is therefore of limited practical use.
Hence, one has to compute the value of the radius $r_\diamond(n)$ itself by
numerically solving (\ref{eq:rdiamond_min_nonneg}). Using these radii and high-precision arithmetic we were able to reproduce \emph{numerically} the exact rational values of the distributions (\ref{eq:distribution}) for $n=15$, $30$,
 $60$, $90$, and $120$ as tabulated by \citeasnoun{Odlyzko},\footnote{For $n=30$, $60$, and $90$, these tables can be found in print in the book of \citeasnoun[pp.~464--467]{MR2129906}.} who has used the \emph{combinatorial} methods exposed in \citeasnoun{MR1771285} for his calculations.
 \begin{remark} The numerical evaluation of $\phi_\lambda(z)$ as given by the Toeplitz determinant (\ref{eq:toeplitz}) turns out to suffer from severe numerical instabilities.
Instead, we suggest to take one of the famous equivalent expressions in terms of a Fredholm determinant; such as the one given by \citeasnoun[p.~391]{MR1780118}
 \begin{multline*}
 \phi_\lambda(z) = e^z \det\left(I-K\projected{\ell^2(\lambda,\lambda+1,\ldots)}\right),\\*[2mm]
  K(j,k)=\sqrt{z}\,\frac{J_j(2\sqrt{z})J_{k+1}(2\sqrt{z})-J_{j+1}(2\sqrt{z})J_{k}(2\sqrt{z})}{j-k};\qquad
 \end{multline*}
 or the one given by \citeasnoun[p.~629]{MR1866169}
 \[
 \phi_\lambda(z) = 2^{-n} e^z \det(I-K\projected{L^2(C_1)}),\qquad K(t,s)=\frac{1-t^n e^{\sqrt{z}(t-t^{-1})} s^{-n} e^{-\sqrt{z}(s-s^{-1})}}{2\pi i (t-s)};
 \]
 see also \citeasnoun{MR1780119} and \citeasnoun{MR1935759}. Both expressions can be evaluated in a numerically stable way; the first using the projection method, the second using the quadrature method exposed in \citeasnoun[§§5 and 6]{Bornemann}.
 \end{remark}
\end{example}

\section*{Acknowledgements}

I am grateful to Ken McLaughlin and Peter Miller who suggested that there should be a relation of the asymptotic formula (\ref{eq:rdiamondasympt}) for the quasi-optimal radius $r_\diamond(n)$ to the saddle-point method. This has turned out to
be the ``missing link'' to really understand the previously mysterious (for me at least) fact that there is $\kappa_\diamond(n) \approx 1$ for so many entire functions I had been looking to.
The observation stated in Footnote~\ref{ft:divakar} is owed to a stimulating discussion with Divakar Viswanath.
\bibliographystyle{kluwer}
\bibliography{D}

\end{document}